\documentclass[times,sort&compress,3p]{elsarticle}
\journal{Journal of Multivariate Analysis}
\usepackage[labelfont=bf]{caption}

\usepackage{amsmath,amsfonts,amssymb,amsthm,booktabs,color,epsfig,graphicx,hyperref,url,multirow,latexsym,bm,srcltx,natbib,wasysym}

\theoremstyle{plain}
\newtheorem{theorem}{Theorem}

\newtheorem{lemma}{Lemma}

\theoremstyle{definition}

\def\bZ{\boldsymbol{Z}}
\def\btau{\boldsymbol{\tau}}

\def\ss{\boldsymbol{s}}

\def\R{{\rm I\! R}}
\def\N{{\rm I\! N}}

\begin{document}

\begin{frontmatter}

\title{Unifying Compactly Supported and Mat{\'e}rn Covariance Functions in Spatial Statistics}

\author[1]{Moreno Bevilacqua\corref{mycorrespondingauthor}}
\author[2]{Christian Caama\~no-Carrillo}
\author[3]{Emilio Porcu}

\address[1]{Facultad de Ingenier\'ia y Ciencias, Universidad Adolfo Ib\'a\~nez, Vi\~na del Mar, Chile.}
\address[2]{Departamento de Estad\'istica, Universidad del B\'io-B\'io, Concepci\'on, Chile.}
\address[3]{Department of Mathematics, Khalifa University, Abu Dhabi}

\cortext[mycorrespondingauthor]{Corresponding author. Email address: moreno.bevilacqua@uai.cl (M. Bevilacqua)\url{}}

\begin{abstract}
The Mat{\'e}rn family  of covariance functions has played a central role in spatial statistics for decades, being  a flexible parametric class with one parameter determining the smoothness of the paths of the underlying spatial field.
This paper proposes a  family of spatial covariance functions, which stems from a reparameterization of the
 generalized Wendland family. As for the Mat{\'e}rn case, the proposed family allows for a continuous parameterization of the smoothness of the underlying Gaussian random field, being additionally compactly supported.

More importantly, we show that the  proposed  covariance  family generalizes the Mat{\'e}rn  model which is attained as  a special limit case.
This implies that    the (reparametrized) Generalized Wendland  model is more flexible than the Mat{\'e}rn model
with an extra-parameter that allows for switching from compactly to globally  supported covariance functions.

Our
numerical experiments elucidate the
speed of convergence of the proposed model to the  Mat{\'e}rn model. We also inspect the asymptotic  distribution of the  maximum
likelihood method when estimating the parameters of the proposed covariance models  under  both increasing and fixed domain asymptotics.
The effectiveness of our proposal   is illustrated by analyzing a georeferenced dataset of mean temperatures over a region of French, and performing a re-analysis of a large spatial point referenced dataset of yearly total precipitation anomalies.


\end{abstract}


\end{frontmatter}

\section{Introduction\label{sec:1}}


Many applications of statistics across a wide range of disciplines rely on the estimation of the
spatial dependence of a physical process based on irregularly spaced observations and
predicting the process at some unknown spatial locations. Gaussian random fields (RFs)
are fundamental to spatial statistics and several other disciplines, such as machine learning,
computer experiments and image analysis, as well as in other branches of applied mathematics including numerical analysis and interpolation theory.

The Gaussian assumption implies  the finite dimensional distributions to be completely
specified through the mean and covariance function. A necessary and sufficient requirement
for a given function to be the covariance function of a Gaussian RF is that it is positive definite.
 Such a requirement is traditionally ensured by selecting a parametric family of covariance functions \citep{Stein:1999}.

 Covariance functions depending exclusively on the distance between any pair of points
located over the spatial domain are called isotropic.
 There is a rich catalog of available spatially isotropic covariance functions  \citep{Stein:1999,Banerjee-Carlin-Gelfand:2004,Cressie:Wikle:2011}, and we make an explicit point in that covariance functions might be globally or compactly supported. The former means that the covariance function does not vanish in the domain of reference, and the latter means that the covariance function vanishes outside  a ball with given radii embedded in a $d$-dimensional Euclidean space.
The use of compactly supported covariance models has been advocated
 when working with (but not necessarily) large spatial datasets    \citep{Furrer:2006,Kaufman:Schervish:Nychka:2008,bd2012,Bevilacqua:20189}  since well-established and implemented algorithms for sparse matrices can be used when estimating the covariance and/or predicting at
 some unknown spatial location
 (see \cite{Furrer:Sain:2010} and the references therein).

Among covariance models with global support, the Mat{\'e}rn  family  \citep{Matern:1960, guttorp2006studies} is the most  popular,
 as it allows for  parameterizing in a continuous fashion the differentiability  of the sample paths of the associated Gaussian RF. Furthermore, it has a very simple form for the associated spectral density, which is crucial for studying the properties of maximum likelihood (ML) estimation \citep{Zhang:2004}, and kriging prediction \citep{Stein:1988, Stein:1990, Furrer:2006}  under fixed domain asymptotics. The Mat{\'e}rn family includes interesting special cases, such as the exponential model, and a rescaled version of the Mat{\'e}rn family converges to the
Gaussian covariance model   \citep{guttorp2006studies}. Additionally, the Mat{\'e}rn model is associated with a class of   stochastic partial differential equations \citep{whittle1954stationary} that has inspired a fertile body of  literature on the  approximation of continuously indexed Gaussian RFs through Markov Gaussian RFs \citep{Lindgren:Rue:Lindstrom:2011}.
Finally, most of the literature on modeling  
 spatiotemporal or multivariate data modeling  is based on the Mat{\'e}rn model as a building block (see \cite{ste2005}, {\cite{pa2006} and \cite{Gneiting:Kleibler:Schlather:2010}, to name a few).

 From a computational perspective, a drawback  of the globally supported   Mat{\'e}rn  family  is that, for a given collection of  $n$ scattered spatial points,  the associated covariance matrix is dense and in this case   the evaluation  of the multivariate Gaussian density  and/or of the optimal predictor
is impractical when $n$ is large. Various scalable estimating/prediction  methods for massive spatial  data have been proposed to reduce the computational burden   (see  \cite{Heatonetal:2019}  and the references therein for a recent review). One of these method is
 the  covariance tapering  technique  proposed in \citep{Furrer:2006,Kaufman:Schervish:Nychka:2008,Stei:13,Wang:Loh:2011}. This kind of approximation  is obtained by specifying a covariance model as the product
of the Matérn model with a  compactly supported   correlation function (the taper function). This allows
  to achieve  a prefixed level of sparseness  in the (misspecified) covariance matrix that can be handled using  algorithms for sparse matrices.

As  recently shown in \cite{Bevilacqua:20189}, a  more appealing  approach with respect to the covariance tapering technique is
to work  with flexible compactly supported covariance models. In particular they study   the generalized Wendland family  introduced in the seminal paper of  \cite{Gneiting:2002b} (see also \cite{Wendland:1995} and \cite{zastavnyi2000positive}).
This class of covariance functions is compactly supported over balls with given radius embedded in $\R^d$ and  it allows for the parameterization of the differentiability of the sample paths of the  underlying Gaussian RF in the same fashion as the Mat{\'e}rn model.
The fact that it is compactly supported manifests a clear practical computational advantage with respect to a globally supported covariance Mat{\'e}rn model.
\cite{Bevilacqua:20189} show, additionally, that under some specific conditions, the Gaussian measures induced by the Mat{\'e}rn and generalized Wendland families
are equivalent.
As a consequence, the kriging   predictors
using  these two covariance models,
have asymptotically the same efficiency under
fixed domain asymptotics \citep{Stein:1999}.



Both  Mat{\'e}rn and generalized Wendland models have three parameters indexing variance, spatial scale (compact support parameter  for the second) and smoothness of the underlying Gaussian RF. Additionally, the
 generalized Wendland model has an extra-parameter that has been conventionally fixed in applications involving spatial data
 and whose  interpretation has not been well understood so far.

 This paper shows that this additional parameter serves a crucial role in proposing a  class of spatial covariance models that unifies the most common covariance models, whatever their support. Specifically, we consider a specific reparameterized  version of the generalized Wendland model, and we show that the Mat{\'e}rn model is  attained as special case  when   the limit to infinity of the additional parameter is considered.
  Hence,  for the first time, we unify compactly and globally supported models under a unique flexible class of spatially isotropic covariance models.
In other words,   the proposed family is a generalization of the the Mat{\'e}rn model
 with an additional parameter
that, for  given smoothness and spatial dependence parameters, allows    for switching from the world  of flexible compactly supported covariance functions
 to the world of flexible globally supported covariance functions.


Our
numerical experiments examine the
speed of convergence of the proposed model to the  Mat{\'e}rn model and  then  we focus
on assessing the asymptotic distribution of the  ML estimator under both increasing and fixed domain asymptotics when estimating the parameters of the proposed covariance model.


While the use of compactly rather than globally supported models implies considerable computational gains \citep{Bevilacqua:20189, Furrer:2006}, it is common belief that compactly supported models are generally associated with a poorer finite sample performance in both terms of maximum likelihood estimation as well as best linear unbiased prediction. Our real data illustrations show that the reparameterized generalized Wendland model can even outperform the Mat{\'e}rn model in terms of both   model fitting and prediction performance. This fact is particularly shown in the first  application.
The second application 
 emphasizes
 the computational savings of the proposed model  with respect to the Mat{\'e}rn model.
The proposed model has been implemented in the \texttt{GeoModels} package \citep{Bevilacqua:2018aa} for the  open-source  \textsf{R} statistical environment.

 The remainder of this paper is organized as follows.
Section \ref{sec2} provides background material about the Mat{\'e}rn and generalized Wendland  covariance models. Section \ref{sec3} provides the main theoretical results  of this paper. In particular, we propose a reparametrization of the Generalized Wendland class  and we show that the Mat{\'e}rn model becomes a special limit case of this class.
   Section \ref{sec5} provides numerical experiments on the speed of convergence of the proposed model to the  Mat{\'e}rn model. We also inspect  the asymptotic distribution of the  ML estimator under both increasing and fixed domain asymptotics.
   In Section \ref{sec6} we analyze a georeferenced dataset of mean temperatures
over a specific region of French and perform a re-analysis of a large spatial point
referenced dataset of yearly total precipitation anomalies.
Finally,  Section
\ref{sec7}  provides some conclusions.

\section{Mat{\'e}rn and generalized Wendland covariance models} \label{sec2}

\subsection{{\bf Gaussian RFs and Isotropic covariance Functions}}

We denote $Z=\{Z(\ss), \ss \in D \} $ as a zero-mean Gaussian RF on a
bounded set $D$ of $\R^d$,  $d=1,2,\ldots$ with stationary covariance function $C:\R^d
\to \R$. The function $C$ is called isotropic when
\begin{equation*} \label{generator}
{\rm cov} \left ( Z(\ss_1), Z(\ss_2) \right )= C(\ss_1-\ss_2)=  \sigma^2 \phi(\|\ss_2 -\ss_1 \|),
\end{equation*}
with $\ss_i \in D$, $i=1,2$, and $\|\cdot\|$ denoting the Euclidean norm, $\sigma^2$ denoting the variance of $Z$, and $\phi:[0,\infty) \to \R$
 with $\phi(0)=1$. For the remainder of the paper, we shall be ambiguous when calling $\phi$ a correlation function. Additionally, we use $r$ for $\|\mathbf{x}\|$, $\mathbf{x} \in \R^d$. \\
Spectral representation of isotropic correlation functions  is available thanks to
\cite{Shoe38}, who showed that the function $\phi$ can be uniquely written as
$$ \phi(r)= \int_{0}^{\infty} \Omega_{d}(r z) F({\rm d} z), \qquad r \ge 0,$$ where
 $\Omega_{d}(r)= r^{1-d/2}J_{d/2-1}(r)$ and $J_{\nu}$ is a Bessel function of order $\nu$. Here, $F$ is a probability measure and is called isotropic {\em spectral measure}. If $F$ is absolutely continuous, then Fourier inversion in concert with arguments in \citet{Yaglom:1987} and \citet{Stein:1999} allow to define the isotropic spectral density, $\widehat{\phi}$, as
\begin{equation} \label{FT}
 \widehat{\phi}(z)= \frac{z^{1-d/2}}{(2 \pi)^d} \int_{0}^{\infty} u^{d/2} J_{d/2-1}(uz)  \phi(u) {\rm d} u, \qquad z \ge 0.
\end{equation} A sufficient condition for  $\widehat{\phi}$ to be well-defined is that $\phi(r)r^{d-1}$ is absolutely integrable.
We now  focus on two  flexible parametric families of isotropic correlation functions.


\subsection{{\bf The Mat{\'e}rn Family}}

The Mat{\'e}rn family of isotropic correlation functions  \citep{Stein:1999}  is defined as follows:
\begin{equation*}
{\cal M}_{\nu,\beta}(r)= \frac{2^{1-\nu}}{\Gamma(\nu)} \left (\frac{r}{\beta}
  \right )^{\nu} {\cal K}_{\nu} \left (\frac{r}{\beta} \right ),
  \qquad r \ge 0,
\end{equation*}
for  $\nu>0,\beta>0$, and it is positive definite in any dimension $d=1, 2, \ldots$. Here, $\Gamma$ is the gamma function and ${\cal K}_\nu$ is the modified Bessel function of the second kind \citep{Abra:Steg:70} of the order $\nu$.
 The parameter $\nu$ indexes the mean squared differentiability  of a Gaussian RF having a Mat{\'e}rn correlation function and its associated sample paths.   In particular, for a
positive integer $k$, the sample paths are $k$ times differentiable, in any direction, if
and only if $\nu>k$ \citep{Stein:1999,BANERJEE200385}. The associated isotropic spectral density  is given by:
\begin{equation} \label{stein1}
\widehat{{\cal M}}_{\nu,\beta}(z)= \frac{\Gamma(\nu+d/2)}{\pi^{d/2} \Gamma(\nu)}
\frac{\beta^d}{(1+\beta^2z^2)^{\nu+d/2}}
, \qquad z \ge 0.
\end{equation}
 When $\nu=m+1/2$ for $m$ a nonnegative integer, then
 $ {\cal M}_{\nu,\beta}$ factors into the product of a negative exponential with a polynomial of degree $m$.  
  For instance, $m=0$ and $m=1$
   correspond, respectively,
   to $ {\cal M}_{1/2,\beta}(r)= \exp(-r/\beta)$ and $ {\cal M}_{3/2,\beta}(r)= \exp(-r/\beta)(1+ r/\beta )$ (see Table \ref{tab1}).
Another relevant fact is that a reparametrized version of the Mat{\'e}rn model converges to the square exponential (or Gaussian) correlation model:
\begin{equation}
\label{convergencetogaussian}
 {\cal M}_{\nu, \beta/(2 \sqrt{\nu})}(r) \xrightarrow[\nu \to \infty]{}  \exp(- r^2/\beta^2),
\end{equation} with convergence being uniform on any compact set of $\R^d$. 

\subsection{{\bf The Generalized Wendland Family }}

The generalized Wendland family  of isotropic correlation functions \citep[][with the references therein]{Bevilacqua:20189} is defined for $\nu>0$ as
\begin{equation} \label{WG2*}
{\cal GW}_{\nu,\mu,\beta}(r):= \begin{cases}  \frac{1}{B(2\nu,\mu+1)} \int_{r/\beta}^{1} u(u^2-(r/\beta)^2)^{\nu-1} (1-u)^{\mu}\,{\rm d}u  ,& 0 \leq r \leq \beta,\\ 0,&r > \beta, \end{cases}
\end{equation}
and for $\nu=0$ as the Askey function \citep{Askey:1973}:
\begin{equation} \label{WG21*}
{\cal GW}_{0,\mu,\beta}(r):= \begin{cases}  \left(1-\frac{r}{\beta}\right)^{\mu}  ,& 0 \leq r \leq \beta,\\ 0,&r > \beta. \end{cases}
\end{equation}
Arguments in \cite{Zastavnyi:2002} show that ${\cal GW}_{\nu,\mu,\beta}$ is positive definite in $\R^d$ for   $\mu\geq \lambda(d,\nu):= (d+1)/2+\nu$ and $\nu\geq 0$ and for a positive compact support parameter $\beta$.
Using results in \cite{H2012},
an alternative useful representation of the generalized Wendland function for $\nu > 0 $, in terms of hypergeometric Gaussian function ${}_2F_1$, is given by:
\begin{equation} \label{WG4*}
	{\cal GW}_{\nu,\mu, \beta}(r)=
	\begin{cases}
	K \left( 1- \left(\frac{r}{\beta} \right )^2\right)^{\nu+\mu}
	    {}_2F_1\left(\frac{\mu}{2},\frac{\mu+1}{2};\nu+\mu+1;1- \left(\frac{r}{\beta} \right )^2 \right)
	& 0\leq r \leq \beta\\
	0 & r > \beta ,\end{cases}
	\end{equation}
	with $K=\frac{\Gamma(\nu)\Gamma(2\nu+\mu+1)}{\Gamma(2\nu)\Gamma(\nu+\mu+1)2^{\mu+1}}$.
The associated isotropic spectral density for $\nu \ge 0$ is given by  the following \citep{Bevilacqua:20189}:
\begin{equation} \label{WG2s}
\widehat{{\cal GW}}_{\nu,\mu, \beta}(z)=L\beta^{d}\mathstrut_1 F_2\Big( \lambda(d,\nu); \lambda(d,\nu)+\frac{\mu}{2}, \lambda(d,\nu)+\frac{\mu+1}{2} ;-\frac{(z\beta)^{2}}4\Big), \quad  z > 0,
\end{equation}
where $L={2^{-d}\pi^{-\frac{d}{2}}\Gamma(\mu+2\nu+1)\Gamma(2\nu+d)\Gamma(\nu)}/\left (  {\Gamma\left(\nu+{d}/{2}\right)\Gamma(\mu+2\nu+d+1)\Gamma(2\nu)} \right )$. Note that the
spectral density  is well-defined   when $\nu=0$ as  $\lim_{x\to 0} \Gamma(\nu)/\Gamma(2\nu)=2$.

The functions ${}_2F_1$ and ${}_1F_2$ are special cases of the generalized hypergeometric functions $\mathstrut_p F_q$ \citep{Abra:Steg:70} given by:
$${}_pF_q(a_1,a_2,\ldots,a_p;b_1,b_2,\ldots,b_q;x):=\sum\limits_{k=0}^{\infty}
\frac{(a_1)_k,(a_2)_k,\ldots,(a_p)_k}{(b_1)_k,(b_2)_k,\ldots,(b_q)_k}\frac{x^k}{k!}\;\;\;\text{for}\;\;\;p,q=0,1,2,\ldots$$
and $(a)_{k}:=  \Gamma(a+k)/\Gamma(a)$, for $k\in \N \cup \{0\} $, is the Pochhammer symbol.
 Similarly to the  Mat{\'e}rn  model, closed-formed solutions can be obtained   when  $\nu=k$ is a nonnegative integer \citep{Gneiting:2002b}.
In particular in this case ${\cal GW}_{\nu,\mu, \beta}$
 factors into the product of the Askey function ${\cal GW}_{0,\mu+k,\beta}$ in Equation (\ref{WG21*}), with a polynomial of degree $k$ (see Table 1).
 Other closed form solutions can be obtained when $\nu=k+0.5$, using some results in \cite{Schaback:2011}.


More importantly, the generalized Wendland  model, as in the Mat{\'e}rn   case,  allows for parameterization in a continuous  fashion  of the
mean squared differentiability  of the underlying Gaussian RF and its associated sample pathsthrough the smoothness parameter $\nu$.
 Specifically,
the sample paths of the generalized-Wendland model are $k$ times differentiable, in any direction, if
and only if $\nu>k-0.5$.
A thorough comparison between the generalized Wendland  and Mat{\'e}rn models with respect to indexing mean squared differentiability is provided by \cite{Bevilacqua:20189}.



\subsection{{\bf Equivalence of Gaussian Measures}}

Denote by $P_i$, $i=0,1$, two probability measures defined on the same
 measurable space $\{\Omega, \cal F\}$. $P_0$ and $P_1$ are called equivalent (denoted $P_0 \equiv P_1$) if $P_1(A)=1$ for any $A\in \cal F$ implies $P_0(A)=1$, and vice versa. For a RF  $Z=\{ Z(\ss), \ss \in D \subset \R^d \}$,  we restrict the event $A$ to the $\sigma$-algebra generated by $Z$ and we  emphasize this restriction by saying that the
two measures are equivalent on the paths of $Z$.

The equivalence of Gaussian measures is a  fundamental tool when studying Gaussian RFs under fixed domain asymptotics and has important implications
on both  estimation and  prediction. 
For instance, using equivalence of Gaussian measures,  \cite{Zhang:2004} has shown that, for the Mat{\'e}rn covariance model $\sigma^2{\cal M}_{\nu,\beta}$, variance and scale cannot be   consistently estimated (for fixed $\nu$). Instead, the parameter $\sigma^2\beta^{-2\nu}$ can be estimated consistently. Similarly, for the generalized Wendland covariance model $\sigma^2{\cal GW}_{\nu,\mu, \beta}$, \cite{Bevilacqua:20189} have shown  that the parameter $\sigma^2 \mu \beta^{-(2\nu+1)} $
can be estimated consistently.
We call those parameters that can be estimated consistently microergodic. Another important implication of the equivalence of Gaussian measures is that the {\em true} (under $P_0$) and  misspecified
(under $P_1$) kriging prediction attain the same asymptotic prediction efficiency \citep{Stein:1999} when $P_0 \equiv P_1$.

Henceforth  we write $P(\sigma^2 \phi )$ for zero-mean Gaussian measures with variance parameter  $\sigma^2$ and an  isotropic correlation function $\phi$.
The following result is taken from \cite{Bevilacqua:20189}   and provides  sufficient conditions for the equivalence of two Gaussian measures
having Mat{\'e}rn and generalized Wendland correlation functions and sharing the same variance.
\begin{theorem}\label{ThmX}
For given $\nu_0  \geq 1/2$ and $\nu_1 \geq 0$,   let $P(\sigma^2 {\cal M}_{\nu_0,\beta})$ and $P(\sigma^2 {\cal GW}_{\nu_1,\mu, \delta})$
  be two zero-mean Gaussian measures. If $\nu_0=\nu_1+1/2$, $\mu > \lambda(d,\nu_1)+d/2$, and
\begin{equation}\label{true}
 \delta
=\beta \left(\frac{\Gamma(\mu+2\nu_1+1)} {\Gamma(\mu)}\right)^{\frac{1}{1+2\nu_1}},
\end{equation}
then for any
bounded infinite set $D\subset \R^d$, $d=1, 2, 3$, $P(\sigma^2 {\cal M}_{\nu_0,\beta})
  \equiv P(\sigma^2 {\cal GW}_{\nu_1,\mu, \delta})$ on the paths of $Z$.
\end{theorem}


\section{A Class of Isotropic Correlations that Unifies Compact and Global Supports} \label{sec3}
This Section provides the main theoretical result of the paper.
Theorem \ref{ThmX} is 
the crux for the subsequent construction. Using Equation (\ref{true}),
we now define the mapping $\delta_{\nu,\mu,\beta}$ through the identity
\begin{equation}
\label{repa} \delta_{\nu,\mu,\beta} = \beta \left(\frac{\Gamma(\mu+2\nu +1)} {\Gamma(\mu)}\right)^{\frac{1}{1+2\nu}},
\end{equation} where $\nu \ge 0$, $\beta>0$ and $\mu \ge 0$ 
and we  define the  $\varphi_{\nu,\mu,\beta}$ class of isotropic correlation models as:
\begin{equation}\label{newbb}
\varphi_{\nu,\mu,\beta}(r) := {\cal GW}_{\nu,\mu,\delta_{\nu,\mu,\beta}}(r), \qquad r \ge 0.
\end{equation}
The  model $\varphi_{\nu,\mu,\beta}$ is a  reparameterization of the generalized Wendland family and, as a consequence, it  is positive definite in $\R^d$ under the conditions $\mu\geq \lambda(d,\nu)$, $\beta>0$, $\nu\geq0$.
 Under this parameterization, the compact support is jointly specified by   $\nu$, $\beta$ and $\mu$, and
 basic properties of the Gamma function show that $\delta_{\nu,\cdot,\beta}$,  $\delta_{\cdot,\mu,\beta}$ and    $\delta_{\nu,\mu,\cdot}$ are  strictly increasing  on $[\lambda(d,\nu),\infty)$, $[0,\infty)$
 and $(0,\infty)$
 respectively.
  Hereafter, we  use  $\varphi_{\nu,\mu,\beta}$ or  ${\cal GW}_{\nu,\mu,\delta_{\nu,\mu,\beta}}$ depending on the context and whenever there is no confusion.

\begin{table}[t!]
\caption{The $\varphi_{\nu,\mu,\beta}$ model
with compact support $\delta_{\nu,\mu,\beta}$ (see Equation \ref{repa})
 for $\nu=0,1,2, 3$ and the associated limit case when $\mu \to \infty$  {\em i.e.}, the Mat{\'e}rn  model
  ${\cal M}_{\nu+1/2,\beta}$.}
\label{tab1}
\begin{center}
\scalebox{0.85}{
\begin{tabular}{|c|l|c|l}
\hline
~$\nu$~& $\varphi_{\nu,\mu,\beta}(r)$  &$ {\cal M}_{\nu+1/2,\beta}(r)$  \\
\hline
\hline
$0$ & $\left(1-\frac{r}{\delta_{0,\mu,\beta}} \right)^{\mu}_{+}$ &    $e^{-\frac{r}{\beta}}$ \\
\hline
$1$ & $\left(1-\frac{r}{\delta_{1,\mu,\beta}}\right)^{\mu+1}_{+}\left(1+\frac{r}{\delta_{1,\mu,\beta}}(\mu+1)\right)$  & $ e^{-\frac{r}{\beta}}(1+\frac{r}{\beta})$  \\
\hline
$2$ & $\left(1-\frac{r}{\delta_{2,\mu,\beta}}\right)^{\mu+2}_{+}\left(1+\frac{r}{\delta_{2,\mu,\beta}}(\mu+2)+\left(\frac{r}{\delta_{2,\mu,\beta}}\right)^2(\mu^2+4\mu+3)\frac{1}{3}\right)$  &$ e^{-\frac{r}{\beta}}(1+\frac{r}{\beta}+\frac{r^2}{3\beta^2})$\\
\hline
\multirow{2}{*}{$3$} & $\left(1-\frac{r}{\delta_{3,\mu,\beta}}\right)^{\mu+3}_{+}
\big( 1+\frac{r}{\delta_{3,\mu,\beta}}(\mu+3)+\left(\frac{r}{\delta_{3,\mu,\beta}}\right)^2(2\mu^2+12\mu+15)\frac{1}{5}\qquad$&\multirow{2}{*}{$ e^{-\frac{r}{\beta}}(1+\frac{r}{\beta}+\frac{2r^2}{5\beta^2}+\frac{r^3}{15\beta^3})$}\\
& $\hspace*{2.8cm}+\left(\frac{r}{\delta_{3,\mu,\beta}}\right)^3(\mu^3+9\mu^2+23\mu+15)\frac{1}{15}\big)$&\\
\hline
\end{tabular}}
\end{center}
\end{table}

We now show that  this  new parameterization of the generalized Wendland model  is  very flexible, as it allows us to consider, under the same umbrella,
 compactly and globally supported   correlation functions.
 In particular, we show that  the Mat{\'e}rn family ${\cal M}_{\nu+1/2,\beta}$ is a special case of the $\varphi_{\nu,\mu,\beta}$ model when $\mu \to \infty$.
 Table $\ref{tab1}$ is taken from  \cite{Bevilacqua:20189}  and it reports the  $\varphi_{\nu,\mu,\beta}$ correlation model for the special cases $\nu=0,1, 2,  3$ and its associated limit case when $\mu \to \infty$
 i.e., the Mat{\'e}rn  correlation model ${\cal M}_{\nu+1/2,\beta}$.

 Two preliminary results are needed for the proof of our main result. Our first preliminary result is of its own interest and  establishes the convergence of the spectral density associated with the $\varphi_{\nu,\mu,\beta}$ model to the spectral density of the  Mat{\'e}rn family ${\cal M}_{\nu+1/2,\beta}$
when $\mu \to \infty$,
uniformly for $z$ in an arbitrary bounded subinterval of  the positive real line.
\begin{theorem}\label{the3}
For $\nu \geq 0$, let $\widehat{\varphi}_{\nu,\mu,\beta}$ be the isotropic spectral density of the   correlation function  $\varphi_{\nu,\mu,\beta}$   defined in Equation (\ref{newbb}), and determined according to (\ref{WG2s}).   Let $\widehat{{\cal M}}_{\nu+0.5,\beta}$ be
the isotropic spectral density of the     correlation function  ${\cal M}_{\nu+1/2,\beta}$  as defined through (\ref{stein1}).
 Then,
\begin{equation} \label{llim}
\lim_{\mu\to\infty}\widehat{\varphi}_{\nu,\mu,\beta}(z)=\widehat{{\cal M}}_{\nu+0.5,\beta}(z),\quad \nu\geq 0
\end{equation}
uniformly for $z$
in an arbitrary bounded subinterval of  the positive real line.
\end{theorem}
\begin{proof}
We provide a constructive proof.
We first calculate the spectral density $\widehat{\varphi}_{\nu,\mu,\beta}$ associated with $\widehat{\varphi}_{\nu,\mu,\beta}$. To do so, we use 
Equation (7), in concert with basic properties of Fourier calculus to obtain
\begin{footnotesize}
\begin{eqnarray}\label{ec1}
\widehat{\varphi}_{\nu,\mu,\beta}(z)=\frac{2^{-d}\pi^{-\frac{d}{2}}\Gamma(\mu+2\nu+1)\Gamma(2\nu+d)\Gamma(\nu)\delta^d_{\nu,\mu,\beta}}{\Gamma\left(\nu+\frac{d}{2}\right)\Gamma(\mu+2\nu+d+1)\Gamma(2\nu)} \mathstrut_1 F_2\left( \lambda(d,\nu); \lambda(d,\nu)+\frac{\mu}{2}, \lambda(d,\nu)+\frac{\mu+1}{2} ;-\frac{(z\delta_{\nu,\mu,\beta})^{2}}4\right).
\end{eqnarray}
\end{footnotesize}
We use the duplication formula for the Gamma function to obtain $\Gamma(x)\Gamma\left(x+1/2\right)=2^{1-2x}\Gamma(2x)$.  We now invoke the series expansion of hypergeometric function $\mathstrut_1 F_2$, and since $ \lambda(d,\nu)=0.5(d+1)+\nu$, we obtain
\begin{footnotesize}
\begin{eqnarray}\label{ec2}
\widehat{\varphi}_{\nu,\mu,\beta}(z)&=&\frac{2^{-d}\pi^{-\frac{d}{2}}\Gamma(\mu+2\nu+1)\Gamma(2\nu+d)\Gamma(\nu)\delta^d_{\nu,\mu,\beta}}{\Gamma\left(\nu+\frac{d}{2}\right)\Gamma(\mu+2\nu+d+1)\Gamma(2\nu)} \sum\limits_{n=0}^{\infty}\frac{\left(\frac{d+1}{2}+\nu\right)_n\delta^{2n}_{\nu,\mu,\beta}}{n!\left(\frac{d+\mu+1}{2}+\nu\right)_n\left(\frac{d+\mu}{2}+\nu+1\right)_n}\left(\frac{-z^2}{4}\right)^n\nonumber\\
&=&2^{-d}\pi^{-\frac{d}{2}}\sum\limits_{n=0}^{\infty}\frac{\Gamma(2\nu+2n+d)\Gamma(\mu+2\nu+1)\Gamma(\nu)\delta^{2n+d}_{\nu,\mu,\beta}}{n!\Gamma(2\nu)\Gamma(\mu+2\nu+2n+d+1)\Gamma\left(\nu+\frac{d}{2}+n\right)}\left(\frac{-z^2}{4}\right)^n\nonumber\\
&=&2^{-d}\pi^{-\frac{d}{2}}\sum\limits_{n=0}^{\infty}\omega_n(\nu)\left(\frac{-z^2}{4}\right)^n,
\end{eqnarray}
\end{footnotesize}
where
$$\omega_n(\nu):=\frac{\Gamma(2\nu+2n+d)\Gamma(\mu+2\nu+1)\Gamma(\nu)\delta^{2n+d}_{\nu,\mu,\beta}}{n!\Gamma(2\nu)\Gamma(\mu+2\nu+2n+d+1)\Gamma\left(\nu+\frac{d}{2}+n\right)}.$$
The ratio test shows that $\sum\limits_{n=0}^{\infty}\omega_n(\nu)\left(\frac{-z^2}{4}\right)^n$ is absolutely convergent for all $z \in \R^+$.
As a consequence, by the dominated convergence Theorem, we can take the limit as $\mu\to\infty$ inside the infinite sum in Equation (\ref{ec2}), giving
\begin{eqnarray}\label{ec3}
\lim_{\mu\to\infty}\widehat{\varphi}_{\nu,\mu,\beta}(z)&=&2^{-d}\pi^{-\frac{d}{2}}\sum\limits_{n=0}^{\infty}\lim_{\mu\to\infty}\omega_n(\nu)\left(\frac{-z^2}{4}\right)^n.
\end{eqnarray}
By the Stirling formula we have $\frac{\Gamma(x+a)}{\Gamma(x+b)}\sim x^{a-b}$, and using the definition of the Pochhammer symbol \citep{Abra:Steg:70}, we have
\begin{footnotesize}
\begin{eqnarray}\label{ec4}
\omega_n(\nu)&=&\frac{\Gamma(2\nu+2n+d)\Gamma(\mu+2\nu+1)\Gamma(\nu)\delta^{2n+d}_{\nu,\mu,\beta}}{n!\Gamma(2\nu)\Gamma(\mu+2\nu+2n+d+1)\Gamma\left(\nu+\frac{d}{2}+n\right)}\nonumber\\
&=&\frac{2^{d+2n}\Gamma\left(\frac{d+1}{2}+\nu\right)\Gamma(\mu+2\nu+1)}{n!\Gamma(\mu+2\nu+2n+d+1)\Gamma\left(\nu+\frac{1}{2}\right)}\left[\beta\left(\frac{\Gamma(\mu+2\nu+1)}{\Gamma(\mu)}\right)^{\frac{1}{1+2\nu}}\right]^{d+2n}\nonumber\\
&\sim&\frac{2^{d+2n}\Gamma\left(\frac{d+1}{2}+\nu\right)\left(\frac{d+1}{2}+\nu\right)_n\beta^{d+2n}}{n!\Gamma\left(\nu+\frac{1}{2}\right)}.
\end{eqnarray}
\end{footnotesize}
Combining Equations (\ref{ec3}) and (\ref{ec4}), we obtain
\begin{eqnarray}\label{ec5}
\lim_{\mu\to\infty}\widehat{\varphi}_{\nu,\mu,\beta}(z)&=&\frac{\pi^{-\frac{d}{2}}\Gamma\left(\frac{d+1}{2}+\nu\right)\beta^d}{\Gamma\left(\nu+\frac{1}{2}\right)}\sum\limits_{n=0}^{\infty}\frac{\left(\frac{d+1}{2}+\nu\right)_n}{n!}[-(z\beta)^2]^n.
\end{eqnarray}
Finally, considering the convergent series $\sum\limits_{n=0}^{\infty}\frac{(a)_n}{n!}(-x)^n=(1+x)^{-a}$ we obtain
\begin{equation*}
\lim_{\mu\to\infty}\widehat{\varphi}_{\nu,\mu,\beta}(z)=\frac{\pi^{-\frac{d}{2}}\Gamma\left(\frac{d+1}{2}+\nu\right)\beta^d}{\Gamma\left(\nu+\frac{1}{2}\right)(1+z^2\beta^2)^{\frac{d+1}{2}+\nu}}
=\widehat{{\cal M}}_{\nu+0.5,\beta}(z).
\end{equation*}
This proves pointwise convergence of a sequence of continuous functions, which is
necessarily uniform on a bounded interval.


\end{proof}

The following result will be useful for the main result in Theorem 3.
\begin{lemma}\label{lem22}
Let $\widehat{\varphi}_{\nu,\mu,\beta}$ be the  spectral density of the   isotropic correlation function  defined in Equation (\ref{newbb}).
Let $\widehat{{\cal M}}_{\nu,\beta}$ be
the isotropic spectral density of the Matérn  isotropic correlation function as defined through (\ref{stein1}).Then,
\begin{equation} \label{lemeq}
\int_{0}^{\infty}z^{d-1} \widehat{\varphi}_{\nu,\mu,\beta}(z){\rm d} z=\int_{0}^{\infty}z^{d-1}\widehat{{\cal M}}_{\nu+0.5,\beta}(z) {\rm d}  z=\frac{\Gamma\left(\frac{d}{2}\right)}{2\pi^{d/2}}.
\end{equation}
\end{lemma}

\begin{proof}
First, using Equation (2) 
 in the main document, in concert with  $3.241.4^{11}$ of \cite{Gradshteyn:Ryzhik:2007}, we obtain
\begin{footnotesize}
\begin{eqnarray} \label{id1}
\int_{0}^{\infty}z^{d-1}\widehat{{\cal M}}_{\nu+0.5,\beta}(z){\rm d}  z&=&\frac{\Gamma\left(\nu+\frac{d+1}{2}\right)\beta^d}{\pi^{d/2} \Gamma\left(\nu+\frac{1}{2}\right)}\int_{0}^{\infty}
\frac{z^{d-1}}{(1+\beta^2z^2)^{\nu+(d+1)/2}}{\rm d}z=\frac{\Gamma\left(\frac{d}{2}\right)}{2\pi^{d/2}}.
\end{eqnarray}
\end{footnotesize}
We now invoke (\ref{ec1}) to obtain
\begin{footnotesize}
\begin{eqnarray} \label{id1}
\int_{0}^{\infty}z^{d-1} \widehat{\varphi}_{\nu,\mu,\beta}(z) {\rm d}  z&=&\frac{2^{-d}\pi^{-\frac{d}{2}}\Gamma(\mu+2\nu+1)\Gamma(2\nu+d)\Gamma(\nu)\delta^d_{\nu,\mu,\beta}}{\Gamma\left(\nu+\frac{d}{2}\right)\Gamma(\mu+2\nu+d+1)\Gamma(2\nu)}\nonumber\\ &\times&\int_{0}^{\infty}z^{d-1}\mathstrut_1 F_2\left(\frac{d+1}{2}+\nu;\frac{d+\mu+1}{2}+\nu,\frac{d+\mu}{2}+\nu+1 ;-\frac{(z\delta_{\nu,\mu,\beta})^{2}}4\right) {\rm d}  z\nonumber\\
&=&\frac{2^{-d}\pi^{-\frac{d}{2}}\Gamma(\mu+2\nu+1)\Gamma(2\nu+d)\Gamma(\nu)\delta^d_{\nu,\mu,\beta}}{\Gamma\left(\nu+\frac{d}{2}\right)\Gamma(\mu+2\nu+d+1)\Gamma(2\nu)}I(d,\mu,\nu).
\end{eqnarray}
\end{footnotesize}
with
$$ I(d,\mu,\nu) := \int_{0}^{\infty}z^{d-1}\mathstrut_1 F_2\left(\frac{d+1}{2}+\nu;\frac{d+\mu+1}{2}+\nu,\frac{d+\mu}{2}+\nu+1 ;-\frac{(z\delta_{\nu,\mu,\beta})^{2}}4\right) {\rm d}  z. $$
Using the identity (8.4.48.1) of \cite{Prudnikov1986} given by $$\int\limits_{0}^{\infty}z^{a-1}\mathstrut_1 F_2\left(a_1;b_1,c_1;-z\right){\rm d} z=\frac{\Gamma(a)\Gamma(a_1-a)\Gamma(b_1)\Gamma(c_1)}{\Gamma(a_1)\Gamma(b_1-a)\Gamma(c_1-a)}$$ and with the change in variable $u=z^2\delta^2_{\nu,\mu,\beta}/4$, we obtain
\begin{footnotesize}
\begin{eqnarray} \label{id2}
I(d,\mu,\nu)&=&\frac{2^{d-1}}{\delta^d_{\nu,\mu,\beta}}\int_{0}^{\infty}u^{d/2-1}\mathstrut_1 F_2\left(\frac{d+1}{2}+\nu;\frac{d+\mu+1}{2}+\nu,\frac{d+\mu}{2}+\nu+1;-u\right){\rm d} u\nonumber\\
&=&\frac{2^{d-1}\Gamma\left(\frac{d}{2}\right)\Gamma\left(\nu+\frac{1}{2}\right)\Gamma\left(\frac{d+\mu+1}{2}+\nu\right)\Gamma\left(\frac{d+\mu}{2}+\nu+1\right)}
{\delta^d_{\nu,\mu,\beta}\Gamma\left(\frac{d+1}{2}+\nu\right)\Gamma\left(\frac{\mu+1}{2}+\nu\right)\Gamma\left(\frac{\mu}{2}+\nu+1\right)}.
\end{eqnarray}
\end{footnotesize}
Combining Equations (\ref{id1}), and (\ref{id2}) and using the duplication formula for the gamma function $\Gamma(x)\Gamma\left(x+\frac{1}{2}\right)=2^{1-2x}\Gamma(2x)$, we obtain
\begin{footnotesize}
\begin{eqnarray} \label{id3}
\int_{0}^{\infty}z^{d-1} \widehat{\varphi}_{\nu,\mu,\beta}(z){\rm d}z&=&\frac{\Gamma(\mu+2\nu+1)\Gamma(2\nu+d)\Gamma(\nu)\Gamma\left(\frac{d}{2}\right)\Gamma\left(\nu+\frac{1}{2}\right)\Gamma\left(\frac{d+\mu+1}{2}+\nu\right)\Gamma\left(\frac{d+\mu}{2}+\nu+1\right)}
{2\pi^{d/2}\Gamma\left(\nu+\frac{d}{2}\right)\Gamma(\mu+2\nu+d+1)\Gamma(2\nu)\Gamma\left(\frac{d+1}{2}+\nu\right)\Gamma\left(\frac{\mu+1}{2}+\nu\right)\Gamma\left(\frac{\mu}{2}+\nu+1\right)}\nonumber\\
&=&\frac{\Gamma\left(\frac{d}{2}\right)}{2\pi^{d/2}}.
\end{eqnarray}
\end{footnotesize}
The proof is completed.
\end{proof}

We are now able  to state the main result of this paper. We establish the uniform convergence of the $\varphi_{\nu,\mu,\beta}
$ correlation model
to the Mat{\'e}rn  ${\cal M}_{\nu+1/2,\beta}$  correlation model as $\mu \to \infty$.
\begin{theorem}\label{the4}
 Let $\varphi_{\nu,\mu,\beta}$ be the isotropic correlation function  defined in Equation (\ref{newbb}).
 Then,
\begin{equation} \label{llim}
\lim_{\mu\to\infty} \varphi_{\nu,\mu,\beta}(r)={\cal M}_{\nu+1/2,\beta}(r),\quad \nu\geq 0
\end{equation}
with uniform convergence for $r \in (0,\infty)$.
\end{theorem}
\begin{proof}
We need to  verify that, for all $\epsilon>0$, there exists $N\in\N$. such that
\begin{footnotesize}
$$|\varphi_{\nu,\mu,\beta}(r)-{\cal M}_{\nu+1/2,\beta}(r)|\leq \epsilon,\;\;\mu>N$$
\end{footnotesize}
Let ${\cal D}=|\varphi_{\nu,\mu,\beta}(r)-{\cal M}_{\nu+1/2,\beta}(r)|$. Using Equation (\ref{FT}) and invoking the H$\ddot{o}$lder inequality, we have
\begin{footnotesize}
\begin{eqnarray} \label{idt2}
{\cal D}&=&
\bigg|r^{1-d/2}\int_{0}^{\infty}z^{d/2} \widehat{\varphi}_{\nu,\mu,\beta}(z)J_{d/2-1}(rz){\rm d}z
-r^{1-d/2}\int_{0}^{\infty}z^{d/2}\widehat{{\cal M}}_{\nu+0.5,\beta}(z)J_{d/2-1}(rz) {\rm d} z\bigg|\nonumber\\
&=&r^{1-d/2}\left|\int_{0}^{\infty}\left(\widehat{\varphi}_{\nu,\mu,\beta}(z)-\widehat{{\cal M}}_{\nu+0.5,\beta}(z)\right)z^{d/2}J_{d/2-1}(rz){\rm d}z\right|\nonumber\\
&\leq&r^{1-d/2}\int_{0}^{\infty}\left|\left(\widehat{\varphi}_{\nu,\mu,\beta}(z)-\widehat{{\cal M}}_{\nu+0.5,\beta}(z)\right)z^{d/2}J_{d/2-1}(rz)\right| {\rm d}z. \nonumber
\end{eqnarray}
\end{footnotesize}
In particular, by the  inequality $|J_{d/2-1}(rz)|\leq{|rz|^{d/2-1}}/({2^{d/2-1}\Gamma\left({d}/{2}\right)})$  \citep{Ch:2014},  and by direct inspection, we obtain
\begin{footnotesize}
\begin{eqnarray} \label{idt3}
{\cal D}&\leq&\frac{1}{2^{d/2-1}\Gamma\left(\frac{d}{2}\right)}
\int_{0}^{\infty}\left|\widehat{\varphi}_{\nu,\mu,\beta}(z)-\widehat{{\cal M}}_{\nu+0.5,\beta}(z)\right|z^{d-1}{\rm d} z\nonumber\\
&\leq&\frac{1}{2^{d/2-1}\Gamma\left(\frac{d}{2}\right)}\bigg\{\int_{0}^{B}\left|\widehat{\varphi}_{\nu,\mu,\beta}(z)-\widehat{{\cal M}}_{\nu+0.5,\beta}(z)\right|z^{d-1} {\rm d} z
+ \int_{B}^{\infty}z^{d-1}\widehat{\varphi}_{\nu,\mu,\beta}(z) {\rm d} z\nonumber\\
&+&\int_{B}^{\infty}z^{d-1}\widehat{{\cal M}}_{\nu+0.5,\beta}(z) {\rm d} z\bigg\}\nonumber\\
&=&\frac{1}{2^{d/2-1}\Gamma\left(\frac{d}{2}\right)}\bigg\{\int_{0}^{B}\left|\widehat{\varphi}_{\nu,\mu,\beta}(z)-\widehat{{\cal M}}_{\nu+0.5,\beta}(z)\right|z^{d-1} {\rm d} z
+\int_{0}^{B}\left[\widehat{{\cal M}}_{\nu+0.5,\beta}(z)-\widehat{\varphi}_{\nu,\mu,\beta}(z)\right]z^{d-1} {\rm d} z\nonumber\\
&+&2\int_{B}^{\infty}z^{d-1}\widehat{{\cal M}}_{\nu+0.5,\beta}(z) {\rm d} z
+\int_{0}^{\infty}z^{d-1}\widehat{\varphi}_{\nu,\mu,\beta}(z) {\rm d} z-\int_{0}^{\infty}z^{d-1}\widehat{{\cal M}}_{\nu+0.5,\beta}(z)  {\rm d} z\bigg\}\nonumber\\
&\leq&\frac{1}{2^{d/2-1}\Gamma\left(\frac{d}{2}\right)}\bigg\{2\int_{0}^{B}\left|\widehat{\varphi}_{\nu,\mu,\beta}(z)-\widehat{{\cal M}}_{\nu+0.5,\beta}(z)\right|z^{d-1} {\rm d} z+2\int_{B}^{\infty}z^{d-1}\widehat{{\cal M}}_{\nu+0.5,\beta}(z) {\rm d} z\bigg\},
\end{eqnarray}
\end{footnotesize}
\noindent
where the last inequality is a direct consequence of Lemma \ref{lem22}. Set $K(d)=( {2^{d/2-1}\Gamma(d/2)})^{-1}$.
From the integrability of $z^{d-1}\widehat{{\cal M}}_{\nu+0.5,\beta}(z)$ over $\R^{+}$,
given an arbitrary $\epsilon> 0$ we  can choose $B$ to be sufficiently large to ensure that
$$\int_{B}^{\infty}z^{d-1}\widehat{{\cal M}}_{\nu+0.5,\beta}(z) {\rm d} z\leq \epsilon /(4K(d)).$$

For the first term,  we note  from Theorem \ref{the3}, that there exists $N \in \N$, such that
$$\int_{0}^{B}\left|\widehat{\varphi}_{\nu,\mu,\beta}(z)-\widehat{{\cal M}}_{\nu+0.5,\beta}(z)\right|z^{d-1} {\rm d} z \leq \epsilon/(4K(d)), \qquad \forall \mu>N. $$
Then,  ${\cal D}\leq K(d)[\epsilon/(2K(d))+ \epsilon/(2K(d))]=\epsilon$,  $\forall \mu>N$
which completes the proof.
\end{proof}


Some comments are in order.
First, note that for a  given smoothness parameter $\nu$ and scale parameter $\beta$,
the $\mu$ parameter allows us to increase or decrease the compact support $\delta_{\nu,\mu,\beta}$ of the proposed model
$ \varphi_{\nu,\mu,\beta}={\cal GW}_{\nu,\mu,\delta_{\nu,\mu,\beta}}$
since $\delta_{\nu,\cdot,\beta}$ is strictly increasing on $[\lambda(d,\nu),\infty)$.
In addition, Theorem \ref{the4} states that when $\mu \to \infty$ the Mat{\'e}rn model with global compact support is achieved.
Hence, the parameter $\mu$ is crucial to fix the sparseness of the associated correlation matrix
and it allows to switch from the world of flexible compactly supported covariance functions to the world of flexible globally supported covariance functions.
In principle, $\mu$ can be estimated from the data (see Section \ref{sec5} and the real data Application  in Section \ref{sec6})
or can be fixed by the user when seeking highly sparse matrices for computational reasons.

As an illustrative example,
Figure    \ref{cova} (b) gives a graphical representation of  ${\cal GW}_{\nu,\mu,\delta_{\nu,\mu,\beta}}$ when $\nu=2$
and $\mu=5, 10, 15$ and when $\mu \to \infty$, that is the  Mat{\'e}rn model
${\cal M}_{\nu+1/2,\beta}$.
The parameter $\beta$ is chosen so that the practical range of the Mat{\'e}rn model is equal to $0.2$ (with practical range, we mean
 the value  $x$ such that ${\cal M}_{\nu+1/2,\beta}(r)$ is lower than $0.05$ when $r > x$).
Apparently, when increasing $\mu$, the ${\cal GW}_{\nu,\mu,\delta_{\nu,\mu,\beta}}$  model approaches the ${\cal M}_{\nu+1/2,\beta}$ model.
Figure   \ref{cova} (b) also reports the associated increasing compact supports  $\delta_{\nu,\mu,\beta}$ ($0.231$, $0.403$ and $0.911$).
Figure    \ref{cova} (a) gives a graphical representation  of the generalized Wendland model using the original parameterization i.e.,
${\cal GW}_{\nu,\mu,\beta}$ when $\nu=2$, $\beta=0.5$  when increasing $\mu$. Using the original parameterization the behavior of the correlation  changes drastically
when increasing $\mu$. In particular   as $\mu \to \infty$,  it can be shown that ${\cal GW}_{\nu,\mu,\beta}(r)=0$ if $r>0$  and   ${\cal GW}_{\nu,\mu,\beta}(r)=1$ if $r=0$.

\begin{figure}[h!]
\begin{center}
\begin{tabular}{cc}
\includegraphics[width=5.2cm, height=5.2cm]{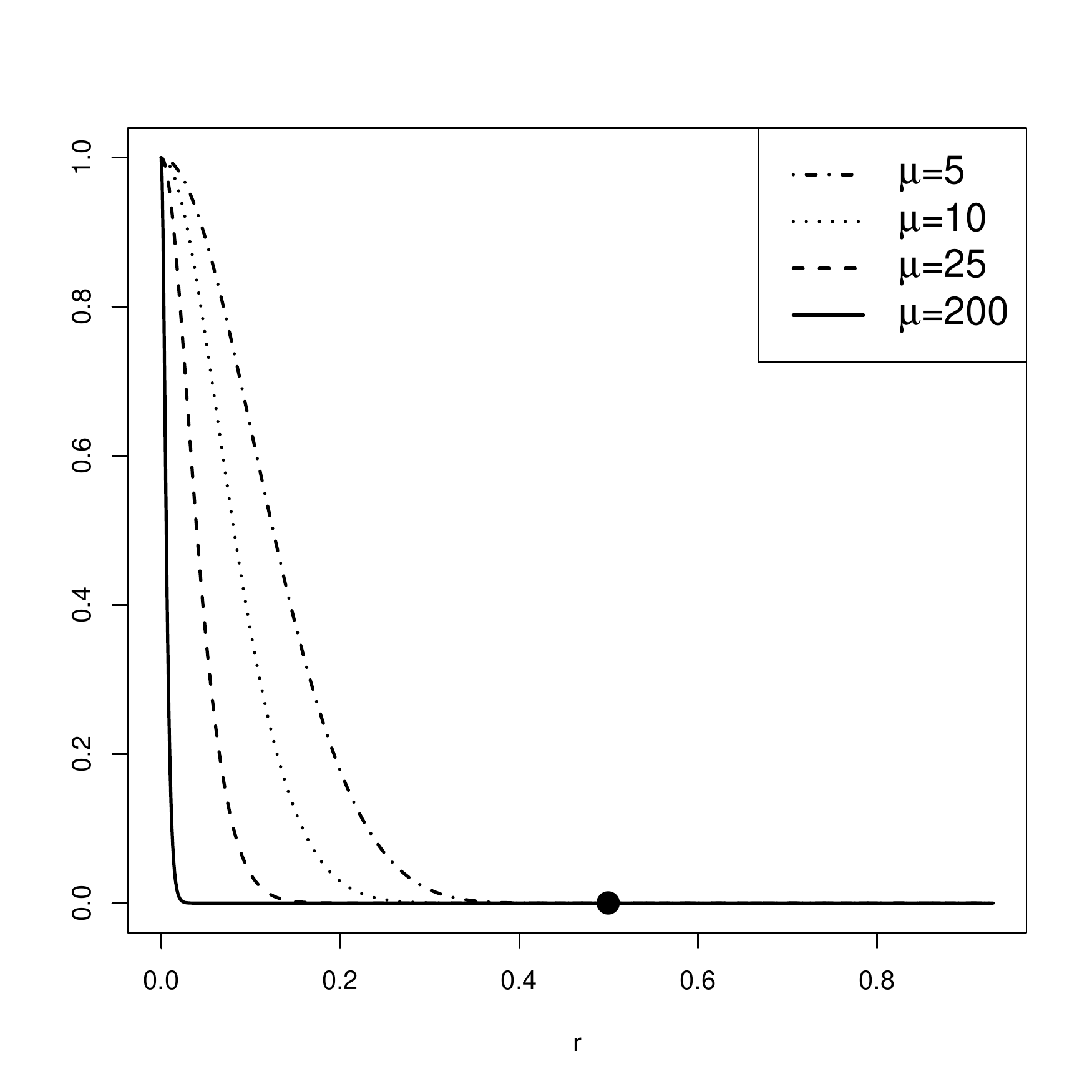}
&
\includegraphics[width=5.2cm, height=5.2cm]{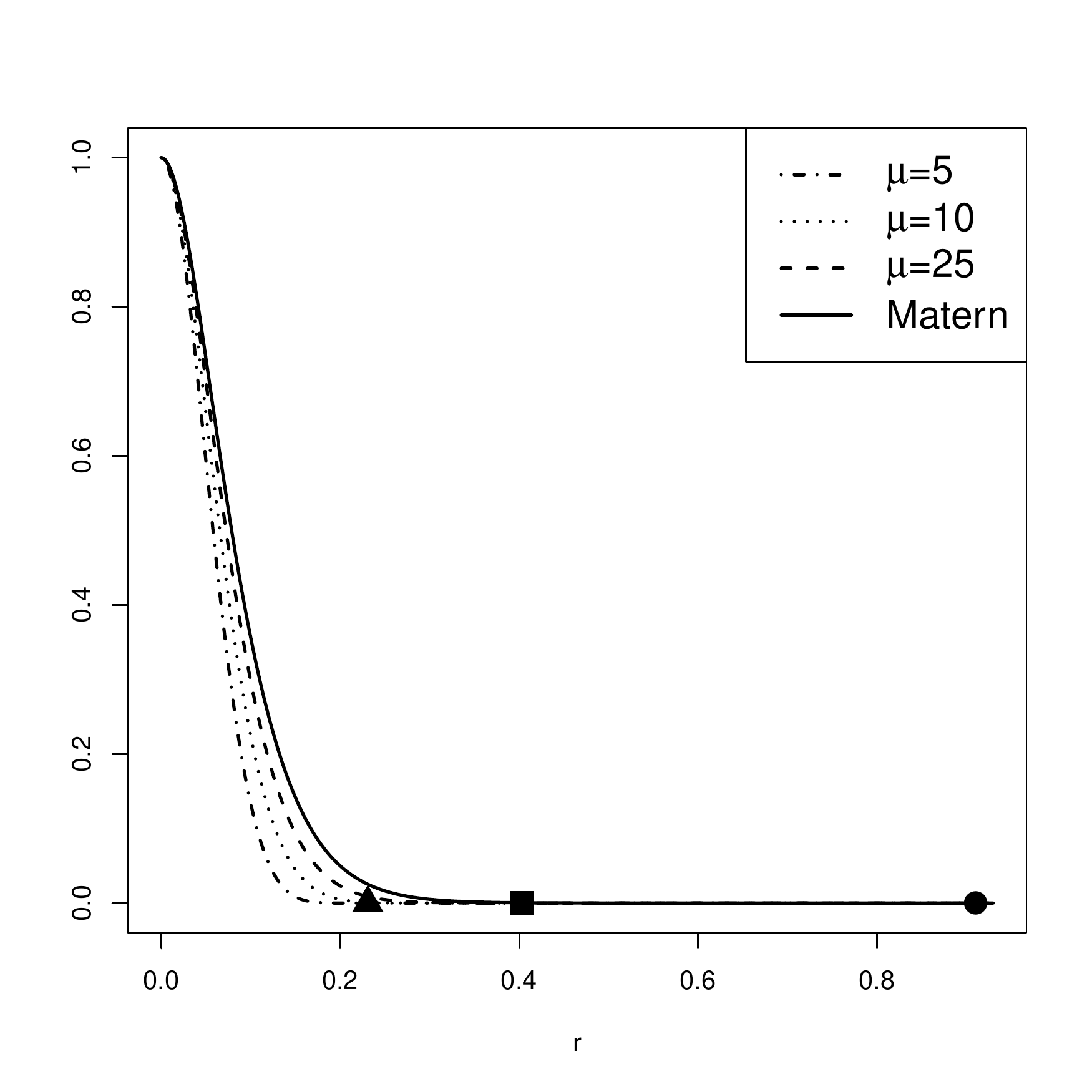}
\\
\hspace{-0.5cm}
(a)  & (b)\\
\end{tabular}
\end{center}
\caption{a): The Generalized Wendland  model ${\cal GW}_{\nu,\mu,0.5}$  when $\nu=2$, $\mu=5, 10, 25, 200$. b):
the proposed reparametrized Generalized Wendland  model  $\varphi_{\nu,\mu,\beta}= {\cal GW}_{\nu,\mu,\delta_{\nu,\mu,\beta}}$
when $\nu=2$, $\beta= 0.0338$ and $\mu=5, 10, 25$ and the limit case when  $\mu \to \infty$ that is the Mat{\'e}rn model ${\cal M}_{\nu+1/2,\beta}$.
In b) the points    ($\blacktriangle$, $\blacksquare$, $\bullet$)   (from left to right) denote the increasing compact support    $\delta_{\nu,\mu,\beta}=0.231, 0.403, 0.911$  associated with $\mu=5, 10, 25$ respectively.}
 \label{cova}
\end{figure}

\begin{figure}[h!]
\begin{center}
\begin{tabular}{cc}
\includegraphics[width=5.2cm, height=5.2cm]{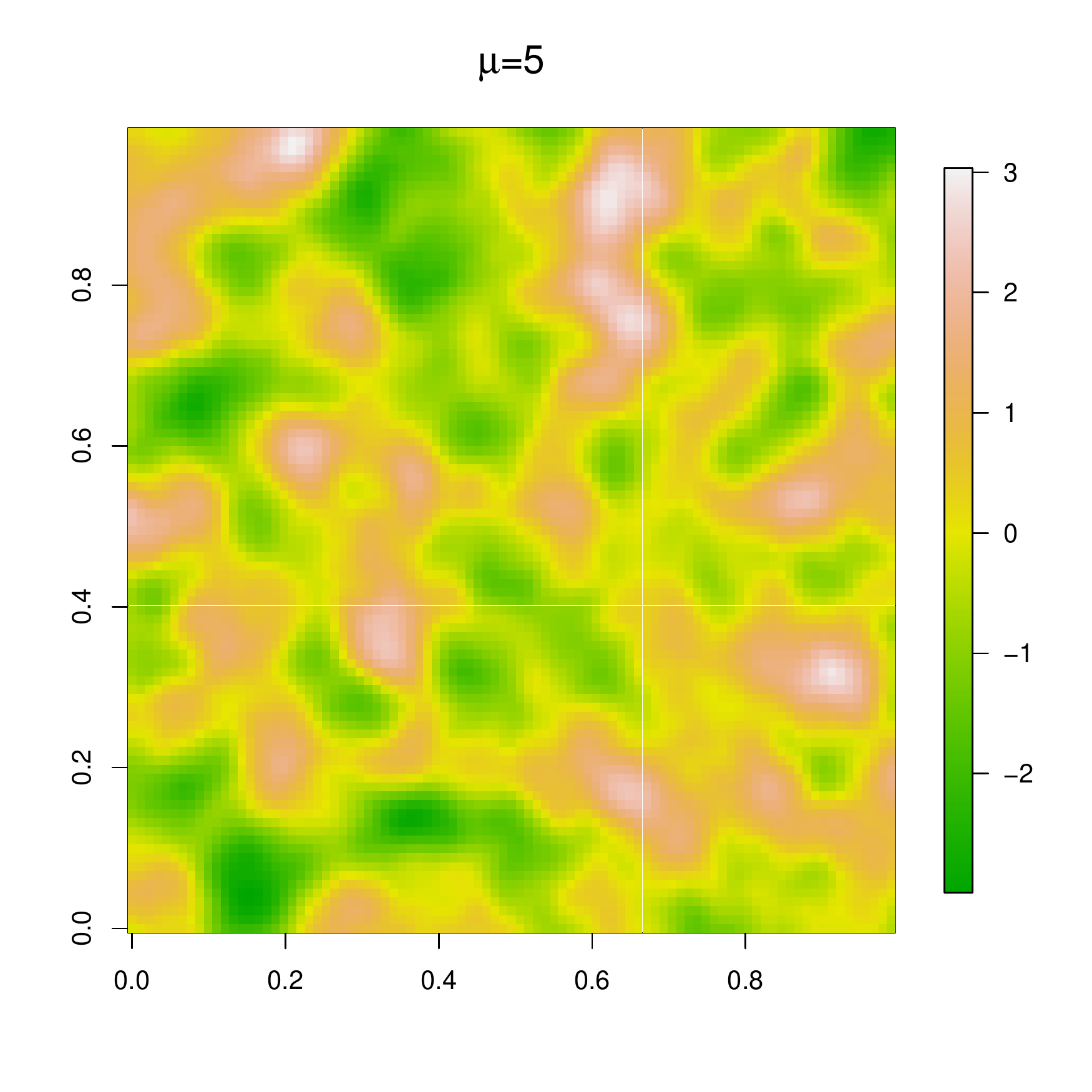}
&
\includegraphics[width=5.2cm, height=5.2cm]{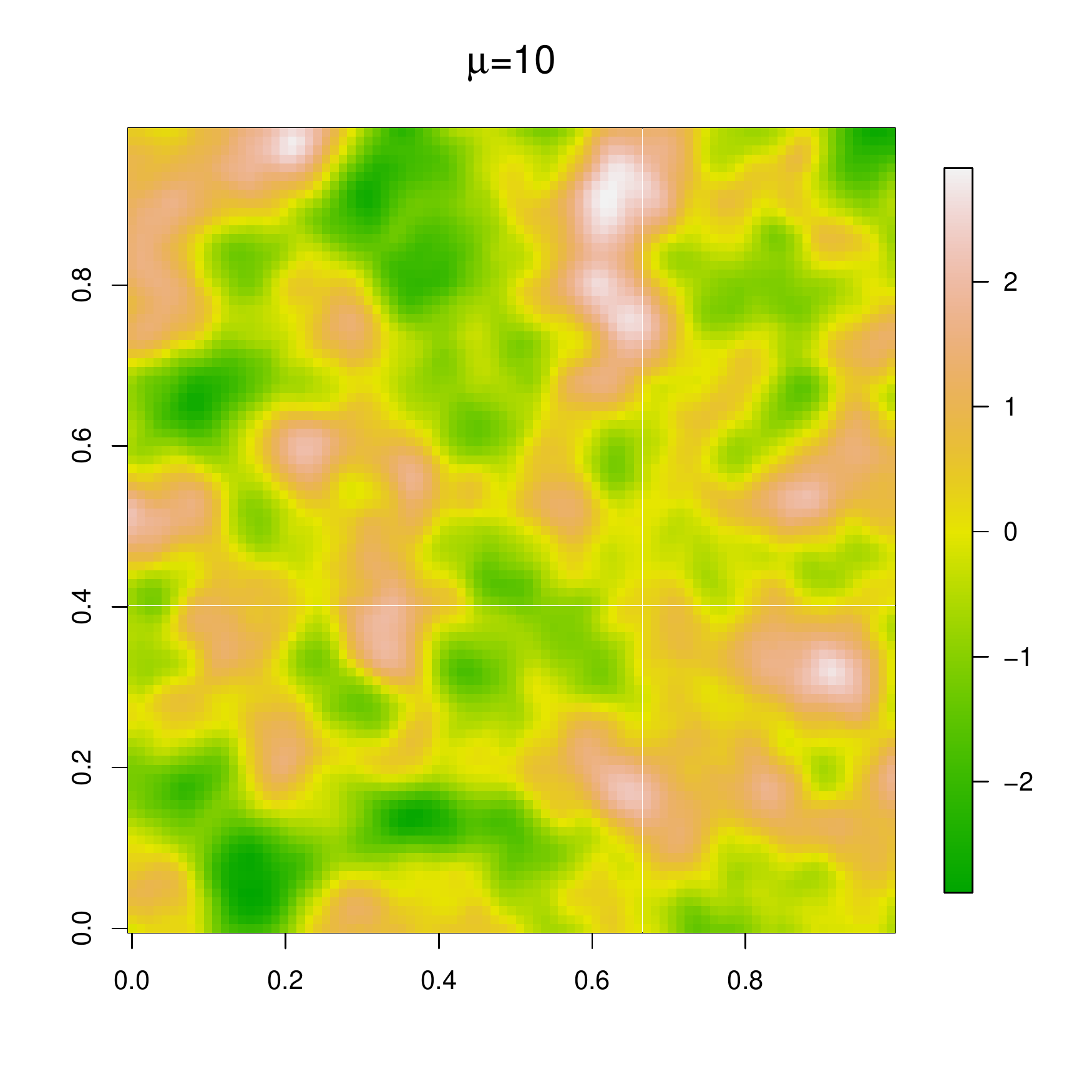}
\\
\includegraphics[width=5.2cm, height=5.2cm]{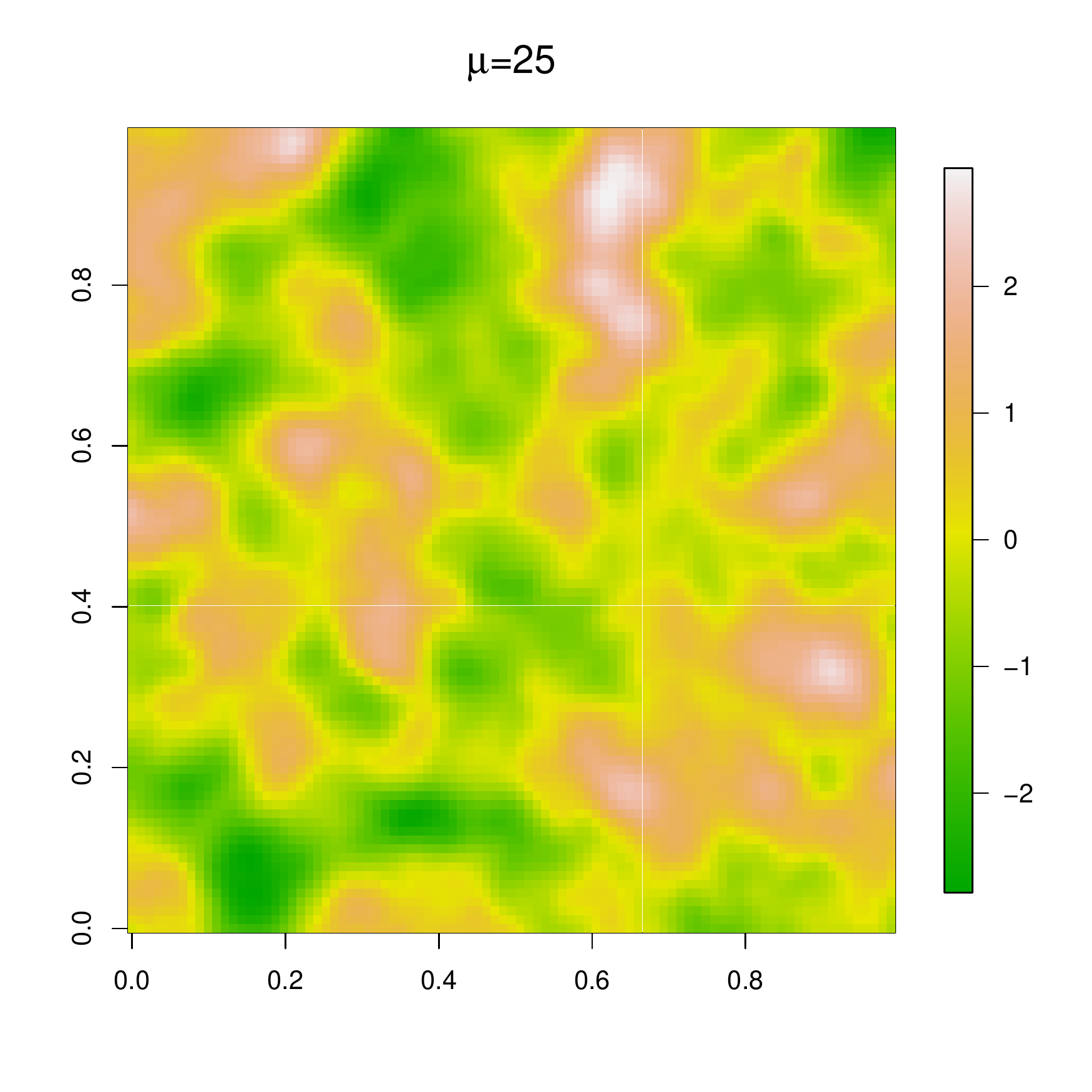}  &
\includegraphics[width=5.2cm, height=5.2cm]{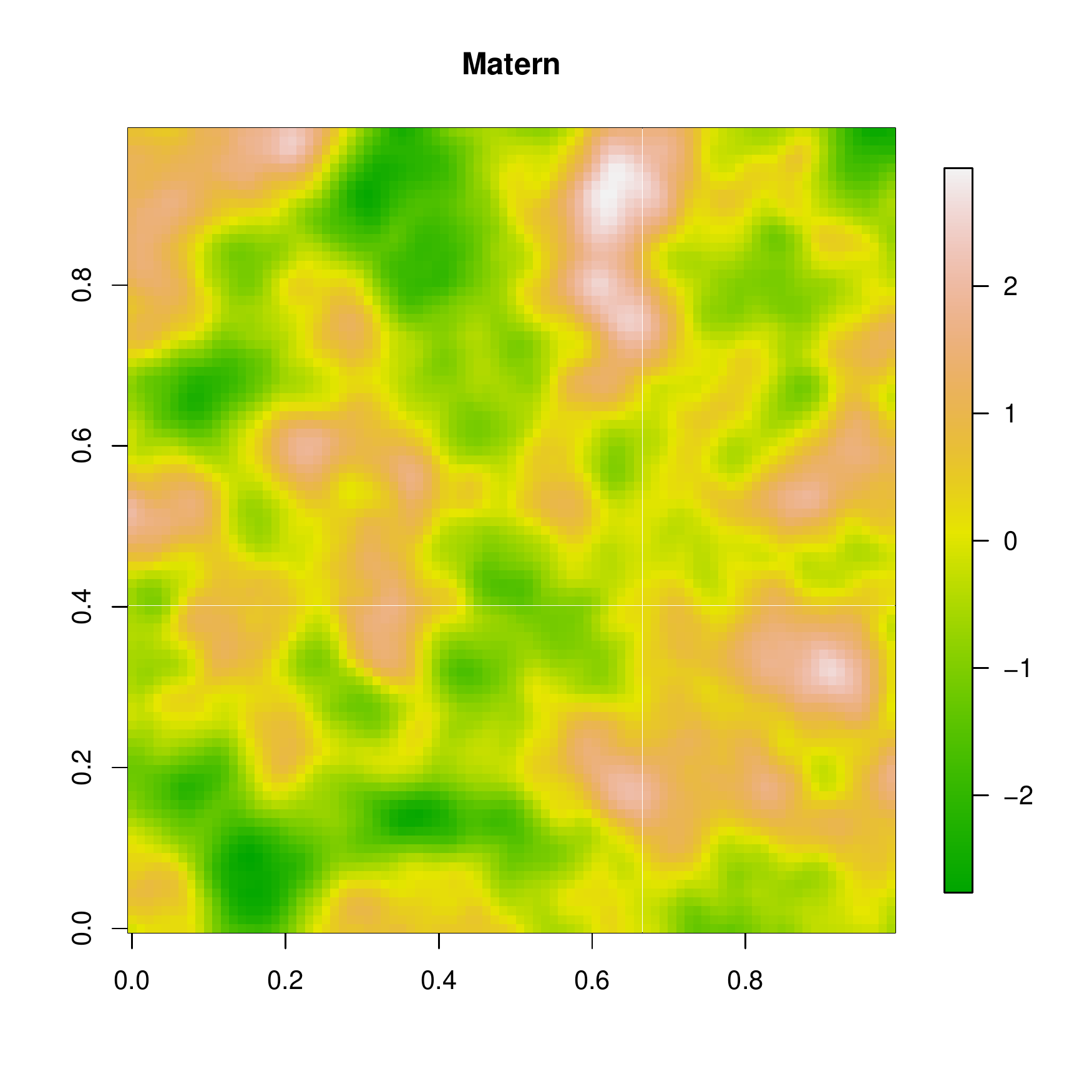}
\\
\end{tabular}
\end{center}
\caption{Four realizations of a Gaussian RF with  
$\varphi_{\nu,\mu,\beta}$ correlation  model
when $\nu=2$, $\beta= 0.0338$ and $\mu=5, 10, 25$ and the limit case when $\mu \to \infty$, that is the
Mat{\'e}rn model
 ${\cal M}_{\nu+1/2,\beta}$ (on the bottom right corner).}
\label{fig:map-ned}
\end{figure}

Figure \ref{fig:map-ned} shows four realizations of a zero-mean  Gaussian RF with
${\cal GW}_{\nu,\mu,\delta_{\nu,\mu,\beta}}$ correlation model using the same parameter settings of Figure \ref{cova}.
For the four realizations we use a common Gaussian simulation using Cholesky decomposition.
It can be appreciated that the realizations are very smooth (the sample paths are   $2$ times differentiable in this case), and they look very similar, even if the first three realizations come from Gaussian RFs with compactly supported correlation functions.

Finally, we point out that the Mat{\'e}rn model is attained as limit when the smoothness parameter is greater than or  equal than $0.5$. This implies
that  the full range of validity of the  smoothness parameter is not covered.  In particular, the proposed model is not able to parameterize
the fractal dimension \citep{gneiting2012a} of the associated Gaussian RF as in the Mat{\'e}rn case.


\section{Numerical experiments} \label{sec5}
\subsection{{\bf Speed of convergence}}
In the absence of theoretical rates of convergence, we show some simple numerical results on the convergence of the $\varphi_{\nu,\mu,\beta}$  to the
Mat{\'e}rn model when increasing $\mu$.  Specifically, we analyze  the absolute  error
\begin{equation}\label{abserr}
E_{\mu,\nu}(r):= |\varphi_{\nu,\mu,\beta}(r) -{\cal M}_{\nu+1/2,\beta}(r)|,   \quad r\geq0,
\end{equation}
 when increasing $\mu$  given   $\nu$ and $\beta$.

 In particular in Figure  \ref{cont22} (first row) we plot  $\varphi_{\nu,\mu,\beta}$ , 
 for $\mu=\lambda(2,\nu), 5, 10, 20, 40, 60, 80$
 and ${\cal M}_{\nu+1/2,\beta}$
for $\nu=0, 1, 2$. Here the $\beta$ parameter is chosen such that the practical range of the Mat{\'e}rn model model is approximately equal to $0.5$ ($\beta=0.167, 0.105, 0.084$, respectively,
for $\nu=0, 1, 2$).
 The second row displays the associated values of $E_{\mu,\nu}$. It can be appreciated that $E_{\mu,\nu}$ decreases when increasing $\mu$
 for each $\nu$, as expected from Theorem \ref{the4} and  the magnitude of the absolute error is increasing with $\nu$.
 In addition, the third row depicts the spectral densities associated to the correlation models in the first row.
 Note that the approximation is getting better for the high-frequency components as $\mu$  increases and it deteriorates when increasing $\nu$.
 These simple numerical  examples shows that the speed of convergence
   depends on the smoothness parameter $\nu$.
Table  \ref{tabxxx} more  deeply depicts  the convergence of the proposed model to    Mat{\'e}rn
by reporting the maximum   absolute error  under a more general parameter setting.  Table  \ref{tabxxx} confirms that
$\varphi_{\nu,\mu,\beta}$ approaches ${\cal M}_{\nu+1/2,\beta}$
when increasing  $\mu$
and the maximum absolute error between them strongly depends on $\nu$.

\begin{figure}[h!]
\begin{tabular}{ccc}
\includegraphics[width=5.25cm, height=5.3cm]{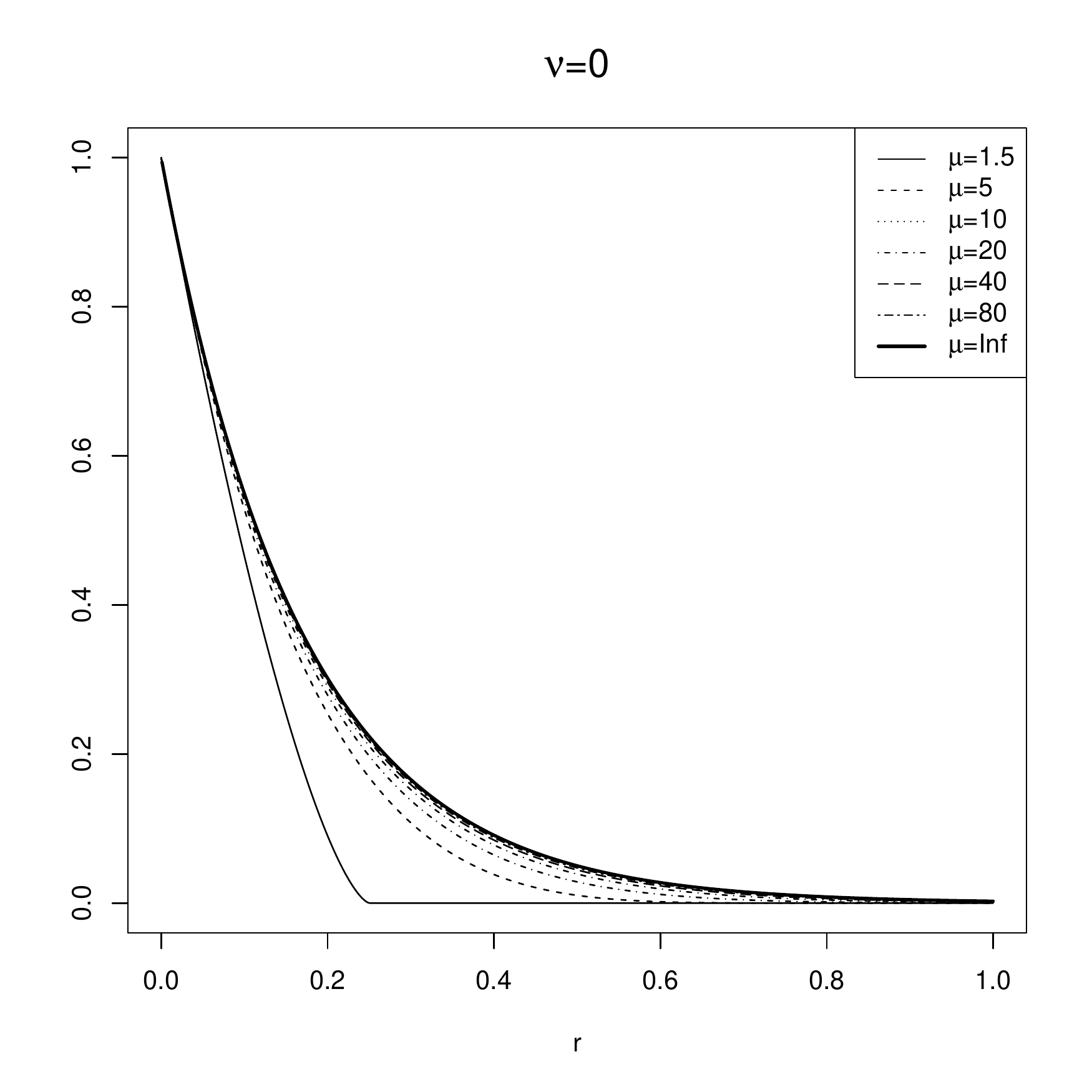} & \includegraphics[width=5.25cm, height=5.3cm]{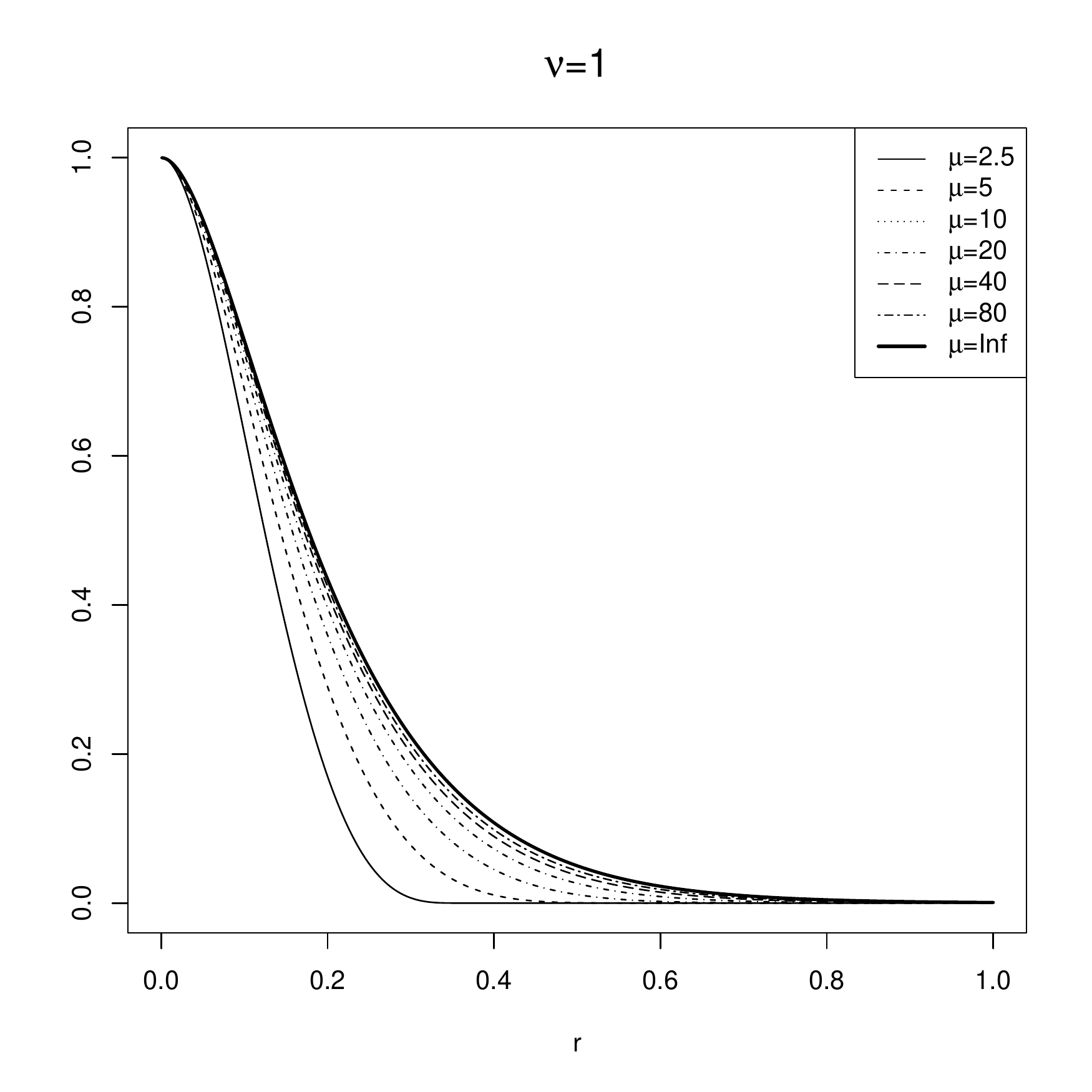} & \includegraphics[width=5.25cm, height=5.3cm]{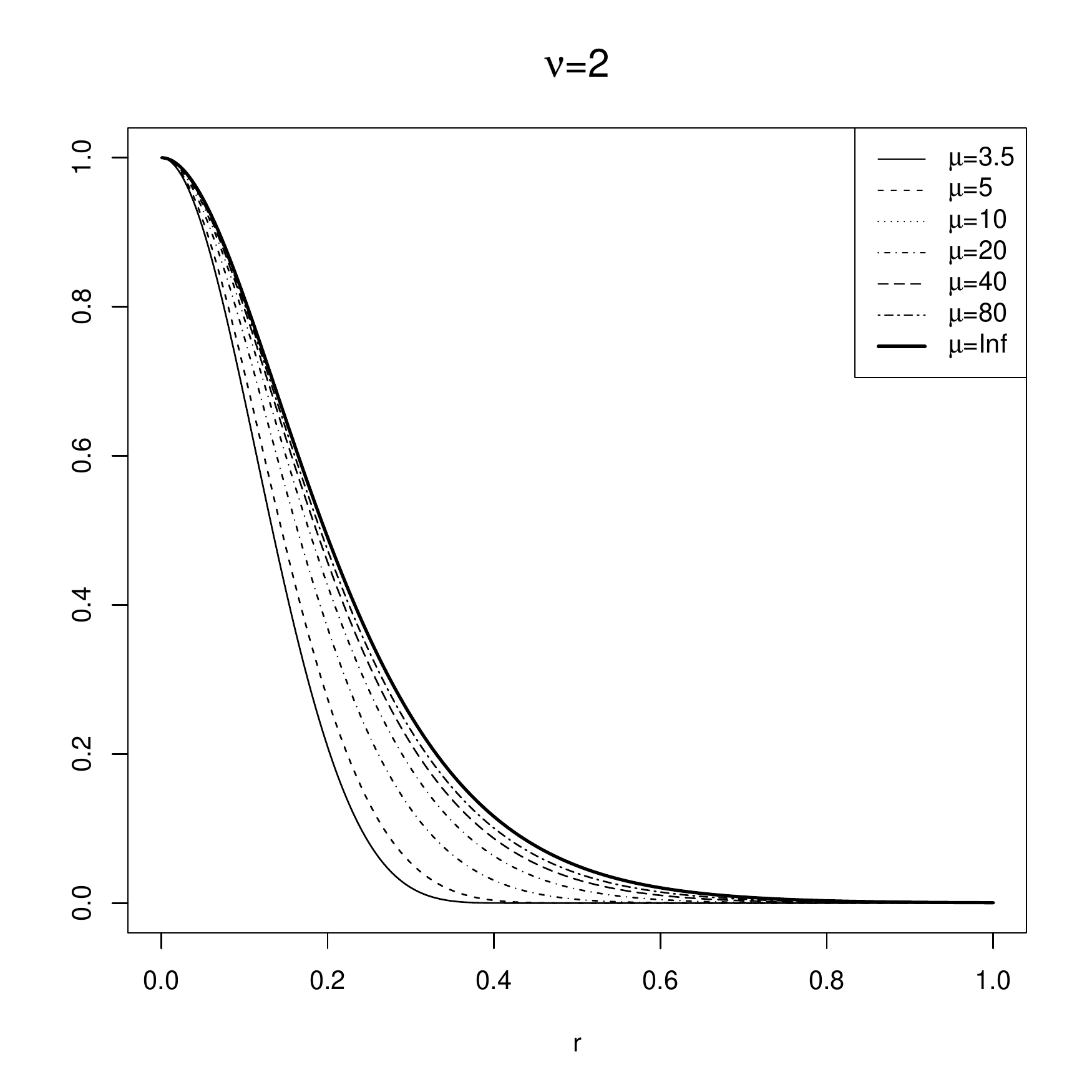}\\
\includegraphics[width=5.25cm, height=5.3cm]{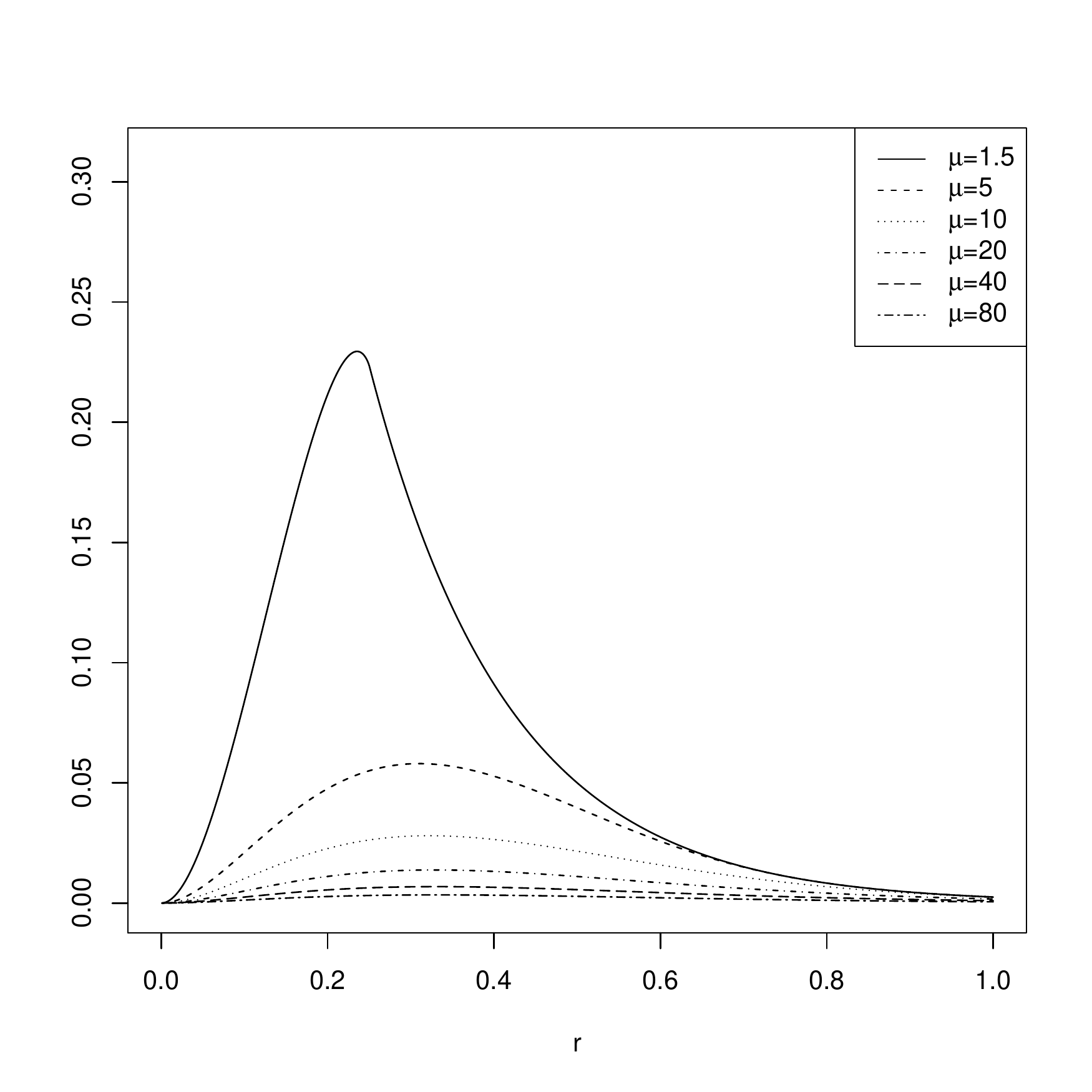} & \includegraphics[width=5.25cm, height=5.3cm]{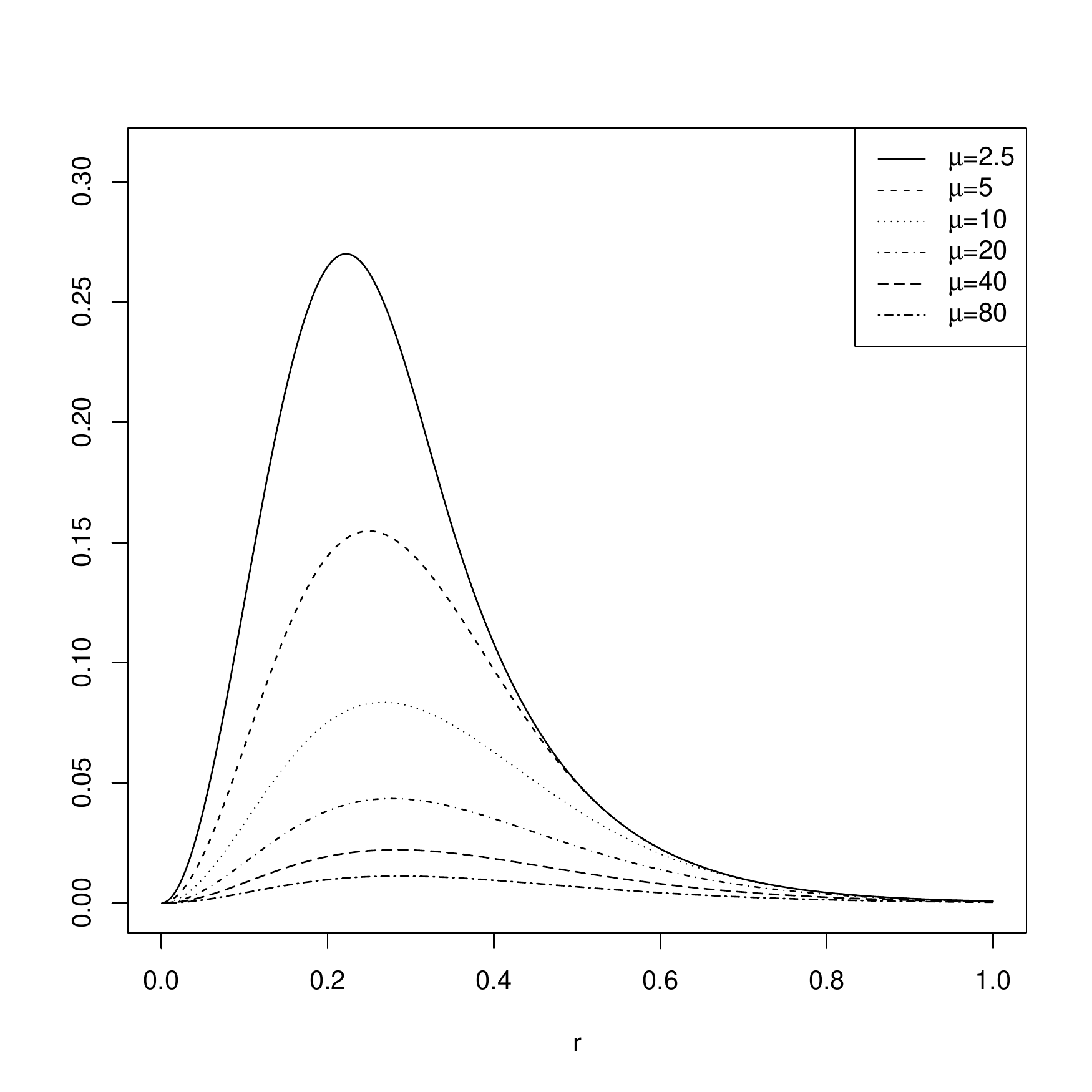}& \includegraphics[width=5.25cm, height=5.3cm]{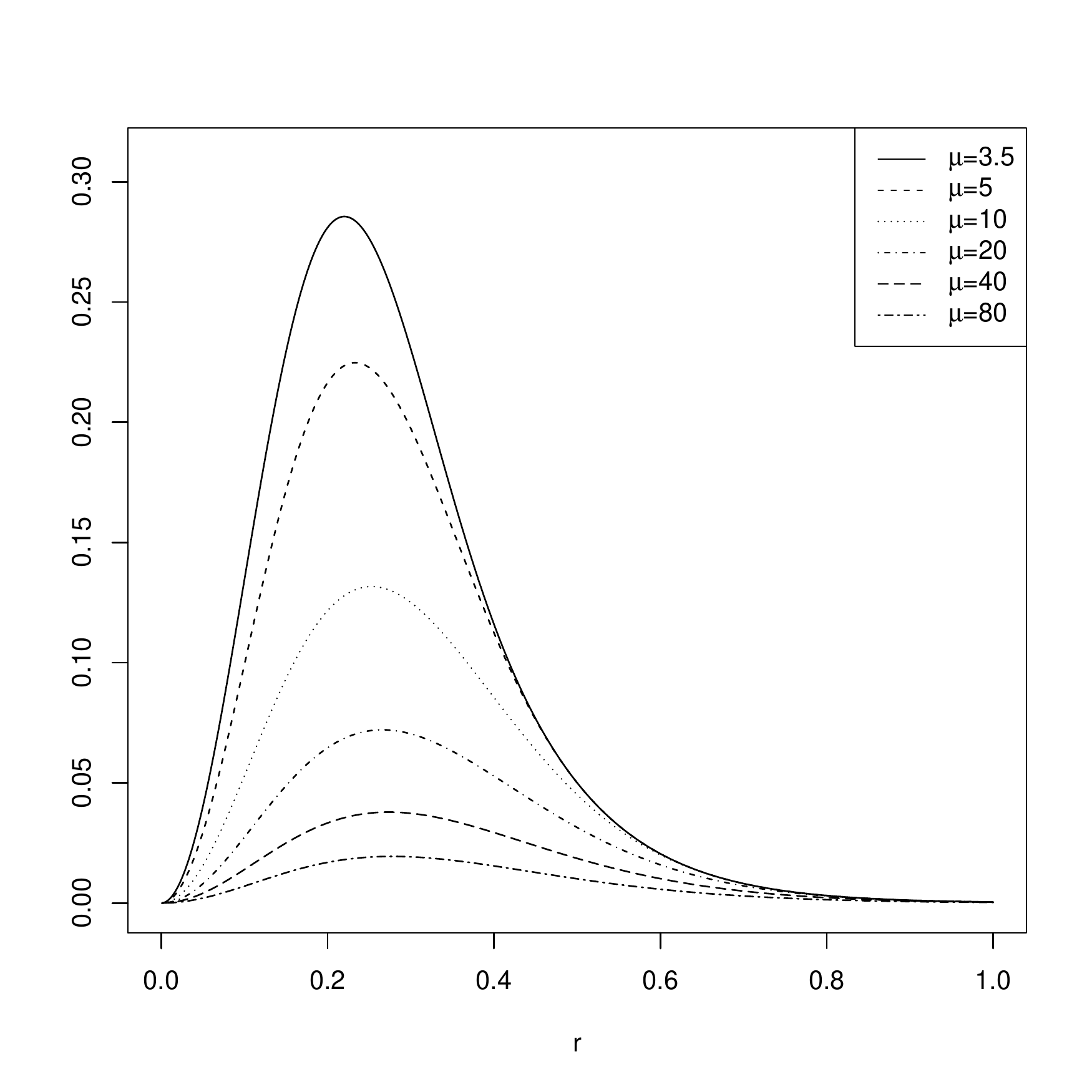} \\
\includegraphics[width=5.25cm, height=5.3cm]{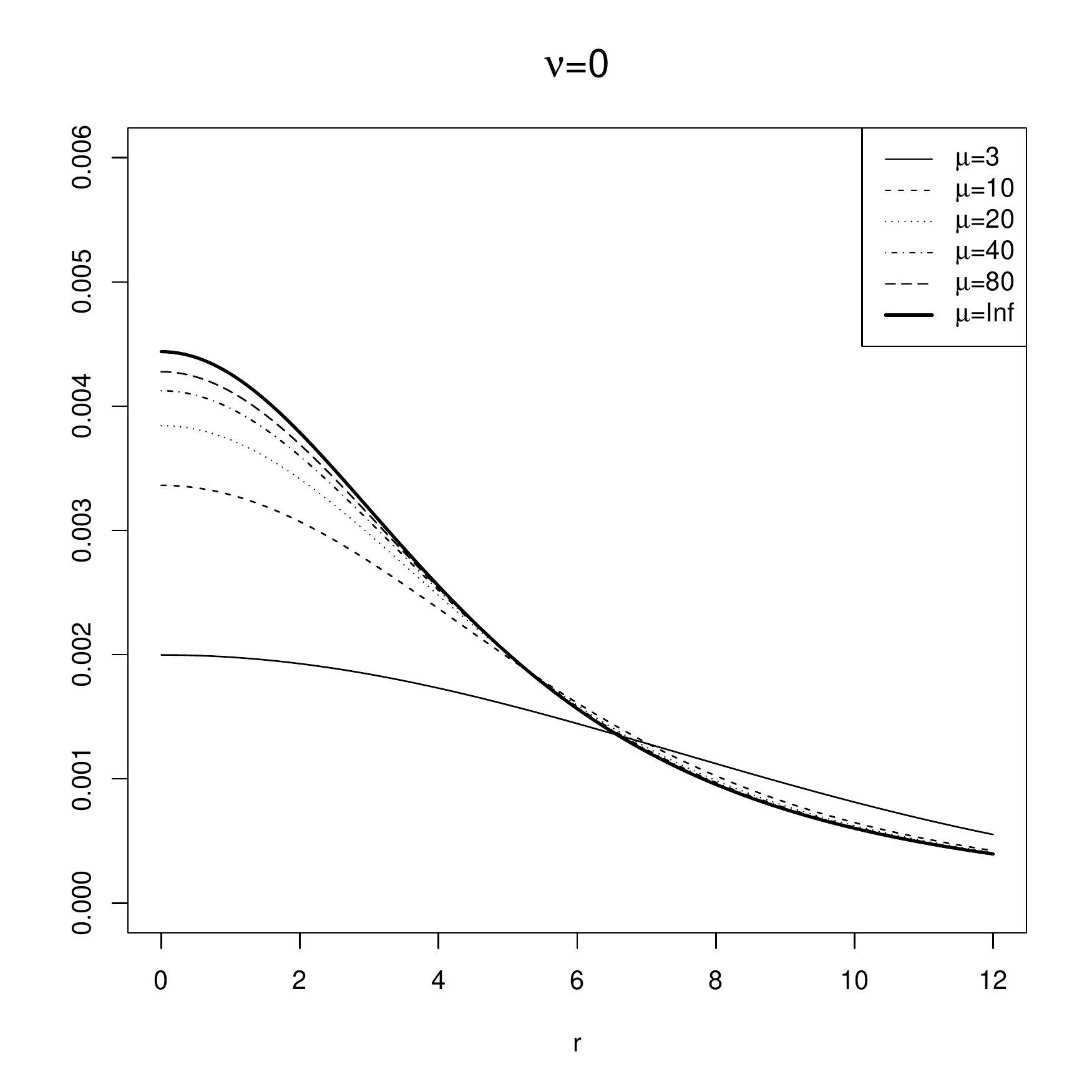} & \includegraphics[width=5.25cm, height=5.3cm]{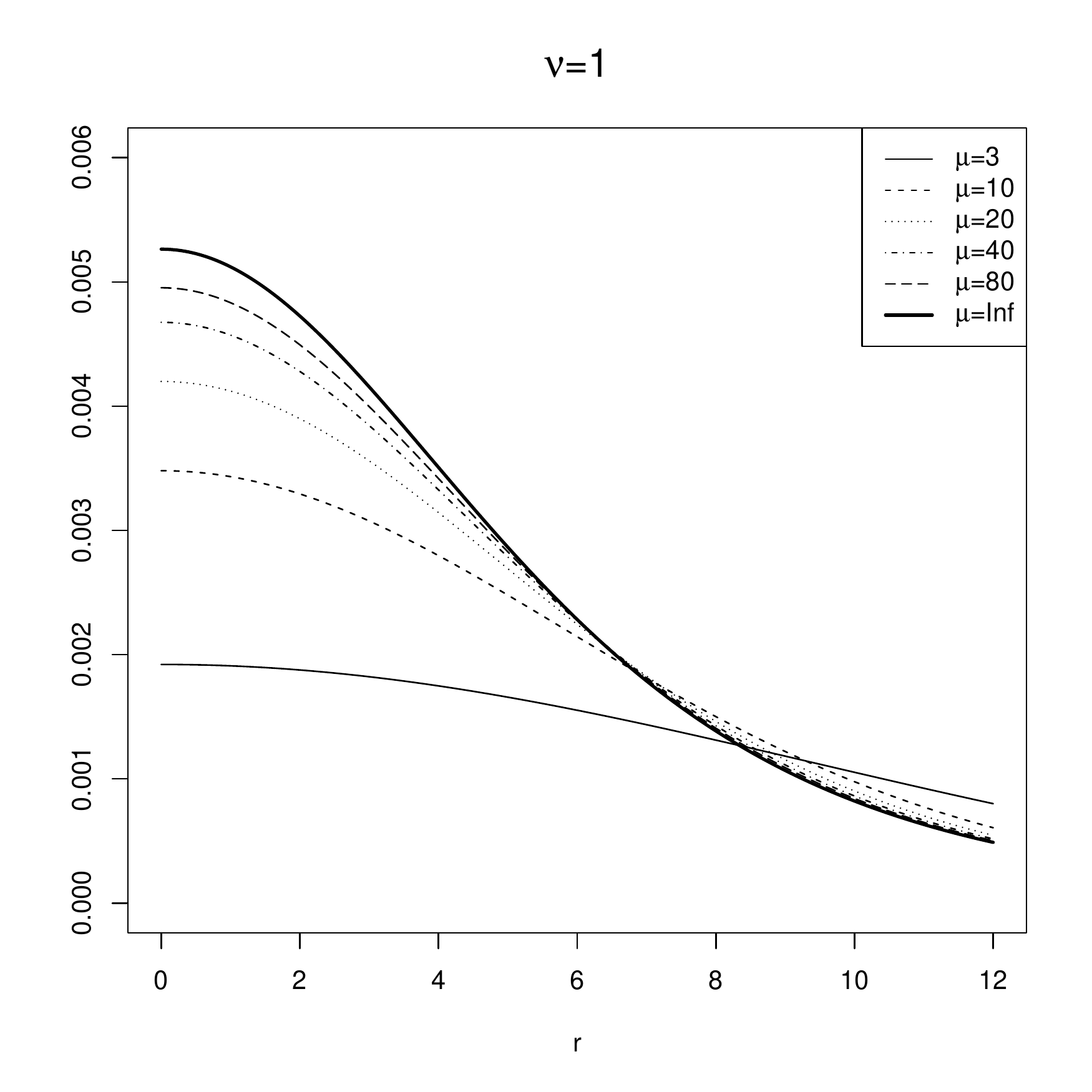}& \includegraphics[width=5.25cm, height=5.3cm]{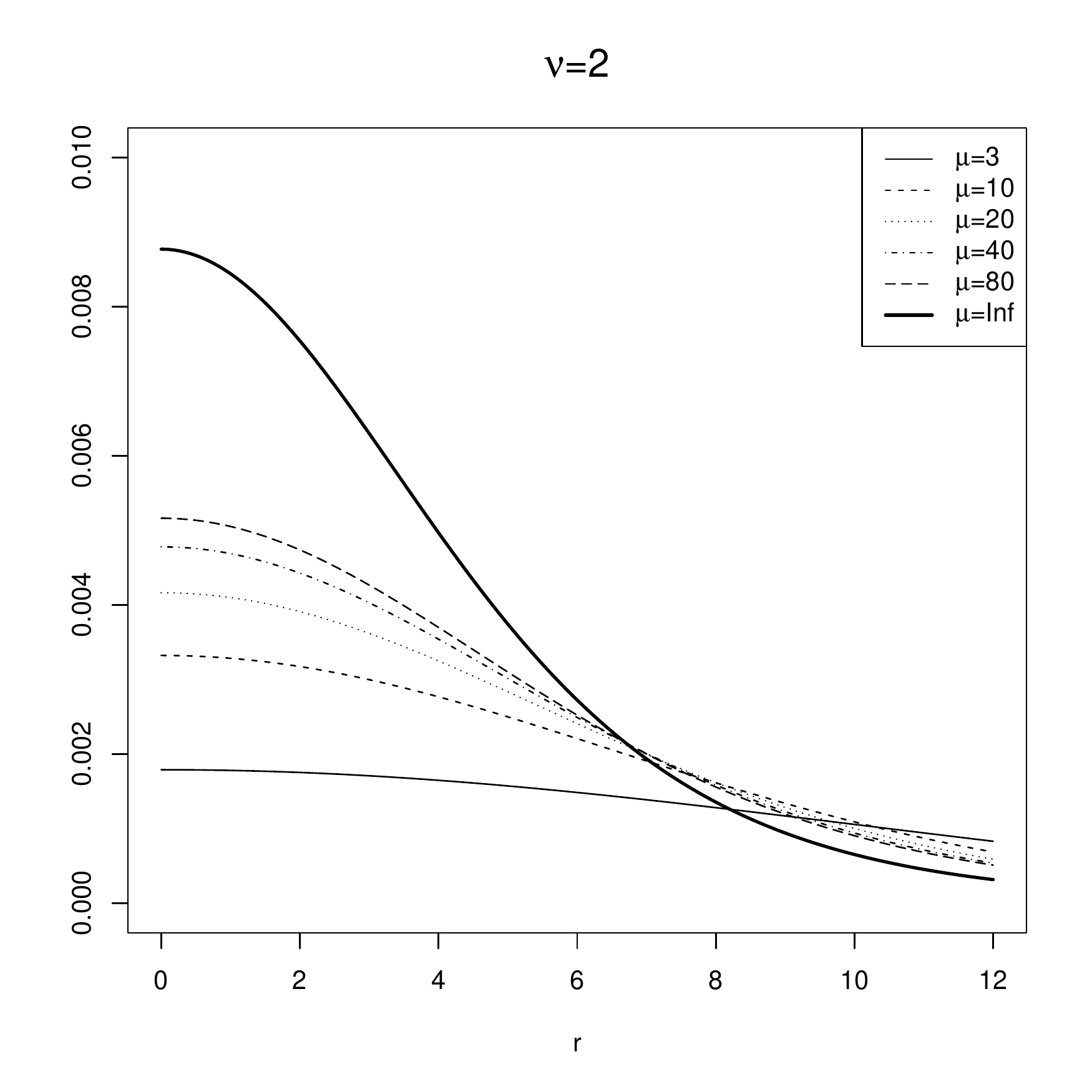} \\
\end{tabular}
\caption{First row: the   $\varphi_{\nu,\mu,\beta}(r)$ model  with $\mu=\lambda(2,\nu), 5, 10,20,40,80$ and $\mu \to \infty$ (the Mat{\'e}rn model)
and  with $\beta= 0.167, 0.105, 0.084$ and  $\nu=0,1,2$ (from left to right) respectively.
Second row: associated absolute value error $E_{\mu,\nu}(r)$ as defined in (\ref{abserr}). Third row:   spectral densities associated with the correlation models in the first row.}  \label{cont22}
\end{figure}

\begin{table}[h!]
\caption{Maximum of    $E_{\mu,\nu}$
as defined in (\ref{abserr})  when increasing $\mu$ for $\nu=0, 0.5,\ldots, 2.5$.} \label{tabxxx}
\begin{center}
\scalebox{0.7}{
\begin{tabular}{|c||c|c|c|c|c|c|c|c|c|}
  \hline
$\mu$     &     $\lambda(d,\nu)$  & 5&    $10$   &   $20$    &  $40$     & $80$   &  $160$   &   $320$    &  $640$\\  \hline\hline
$\nu=0.0$   &   $0.22944$ &0.05799 & $0.02800$ & $0.01376$ & $0.00682$ & $0.00340$ & $0.00170$ & $0.00085$ & $0.00042$\\
$\nu=0.5$ &   $0.25586$ & $0.11010 $& $0.05643$ & $0.02857$ & $0.01438$ & $0.00721$ & $0.00361$ & $0.00181$ & $0.00090$\\
$\nu=1.0$   &   $0.27001$ & 0.15470 & $0.08346$ & $0.04345$ & $0.02218$ & $0.01121$ & $0.00564$ & $0.00283$ & $0.00141$\\
$\nu=1.5$ &   $0.27914 $ & $0.19257$ &$0.10856$ & $0.05800$ & $0.03004$ & $0.01529$ & $0.00772$ & $0.00388$ & $0.00194$\\
$\nu=2.0$   &   $ 0.28554 $ &$0.22475$ & $0.13164$ & $0.07205$ & $0.03782$ & $0.01940$ & $0.00983$ & $0.00494$ & $0.00248$\\
$\nu=2.5$ &   $0.29029$  & $0.25230$ & $0.15279$ & $0.08552$ & $0.04549$ & $0.02350$ & $0.01195$ & $0.00603$ & $0.00303$\\
  \hline
\end{tabular}
}
\end{center}
\end{table}

\subsection{{\bf On the asymptotic distribution of the maximum likelihood estimator}}\label{mlll}
This Section focus on the ML estimation of the proposed covariance model. Let $D$ be  a 
subset of $ \R^d$
and  $S_n=\{ \ss_1,\ldots,\ss_n \in D \}$
 denote any set of distinct locations.
Let $\bZ_n=(Z(\boldsymbol{s}_1),\ldots,Z(\boldsymbol{s}_n))^{\top}$
be a finite  realization of   a zero-mean stationary Gaussian RF   $Z=\{ Z(\ss), \ss \in D  \}$, with isotropic covariance function
$\sigma^2\varphi_{\nu,\mu,\beta}$. Here, $\top$ denotes transposition.

We then write 
$R_{n}(\btau)=[\varphi_{\btau}(\|\boldsymbol{s}_i-\boldsymbol{s}_j\| )]_{i,j=1}^n$  with $\btau=(\nu,\mu,\beta)^{\top}$ for the associated correlation matrix.
If $\bm{\theta}=(\sigma^{2},\btau)^{\top}$,  the  Gaussian log-likelihood function is defined as follows:
\begin{equation}\label{eq:17}
\mathcal{L}_{n}(\bm{\theta})=-\frac{1}{2} \left(n\log(2\pi\sigma^{2})+\log(|R_{n}(\btau)|)+\frac{1}{\sigma^{2}}\bZ_n^{\top}R_{n}(\btau)^{-1}\bZ_n \right),
\end{equation} and  $\widehat{\bm{\theta}}_n:=\text{argmax}_{\bm{\theta} \in
\bm{\theta}}\mathcal{L}_{n}(\bm{\theta})$ is the ML estimator of $\bm{\theta}$.
 \cite{Mardia:Marshall:1984}  provide  general conditions for the
consistency and the asymptotic normality of the ML estimator irrespective of the correlation model. Under
suitable conditions, 
$\widehat{\bm{\theta}}_n$  is consistent and
asymptotically normal,  that is  $\widehat{\bm{\theta}}_n-\bm{\theta}  \stackrel{\mathcal{D}}{\longrightarrow}  \mathcal{N}\left(\bm{0},F^{-1}_n(\bm{\theta})\right)$ as $n\to\infty$
where
 \begin{equation}\label{eq:fisher}
\bm F_n(\bm{\theta})=\left[\frac{1}{2}\textrm{tr}\left(  \Sigma_n(\bm{\theta})^{-1} \frac{d  \Sigma_n(\bm{\theta})}{d \bm{\theta}_i}
 \Sigma_n(\bm{\theta})
^{-1}  \frac{d  \Sigma_n(\bm{\theta})}{d \bm{\theta}_j}\right) \right]_{i,j=1}^p.
\end{equation}
is
the Fisher Information matrix and $\Sigma_n(\bm{\theta})=\sigma^2R_{n}(\btau)$.
The conditions are normally  difficult to verify and they assume indirectly that the sample set grows
in such a way that the sampling domain
increases in extent as $n$ increases (i.e., $||\bm{s}_i-\bm{s}_j||\geq c>0$), which implies that  the set $D$ is unbounded.

Under fixed domain asymptotics, no general results are available for the asymptotic properties of ML estimator.
For the Generalized Wendland   model they have been studied in  \cite{Bevilacqua:20189}.
In particular, using Theorem 4 in  \cite{Bevilacqua:20189}, it can be shown that
if $P(\sigma_i^2\varphi_{\nu,\mu_i,\beta_i})$,  $i=0, 1$, are two zero mean Gaussian  measures
 and   if  $\mu_i > \nu+d+0.5$
 then for any
bounded infinite set $D\subset \R^d$, $d=1, 2, 3$,
$P(\sigma^2_0\varphi_{\nu,\mu_0,\beta_0})\equiv P(\sigma^2_1\varphi_{\nu,\mu_1\beta_1}) $ on the paths of $Z$ if and only if
\begin{equation} \label{condition1_iff}
\frac{\sigma_0^2}{ \beta^{2\nu+1}_0 }g(\nu,\mu_0) =\frac{\sigma_1^2}{ \beta^{2\nu+1}_1 }g(\nu,\mu_1).
\end{equation}
where $g(\nu,\mu)={\Gamma(\mu+1)  } /{ \Gamma(2\nu+\mu+1)}$. A straight  consequence is that for fixed $\nu$, the $\beta$, $\mu$ and $\sigma^2$ parameters cannot be estimated consistently
under fixed domain asymptotics.  Instead, the microergodic parameter  $$c(\bm{\theta})=\frac{ \sigma^{2}}{\beta^{1+2\kappa}}g(\nu,\mu)$$
is consistently estimable. Additionally, using Theorem 8 in \cite{Bevilacqua:20189},  for any fixed  $\nu$ and $\mu\geq \lambda(d,\nu)+3$ as $n\to\infty$, the asymptotic distribution of  ML estimator  of the microergodic parameter is given by
$$ \sqrt{n}\left( \frac{\hat{\sigma}_{n}^{2}}{ \hat{\beta}_n^{2\nu+1} }g(\nu,\mu)-c(\bm{\theta}) \right)  \stackrel{\mathcal{D}}{\longrightarrow}  N\left(0,2c(\bm{\theta})^2\right). $$
where $\hat{\beta}_{n}$ and $\hat{\sigma}_{n}^{2}$ are ML estimators of $\beta$ and $\sigma^2$.

We analyze the performance of the  ML method   when estimating
the parameters  of the covariance model $\sigma^2\varphi_{\nu,\mu,\beta}$ from  both increasing  and fixed domain asymptotics  perspective.
In particular we focus on assessing the approximation given by  the asymptotic distribution of the ML estimation under both types of asymptotics.

We first simulate  $500$ realizations
of a zero mean Gaussian RF   with  covariance model  $\sigma^2\varphi_{\nu,\mu,\beta}$
observed over $n=1000$ location sites uniformly distributed in the unit square.
The smoothness  parameter is assumed  to be known and fixed  equal to $\nu=0, 1, 2$. We  set  $\sigma^2=1$, $\mu= \lambda(2,\nu)+x$ with $x=1, 2, 4$ and since the increasing as well as the fixed-domain frameworks can
be mimicked by fixing the number of location sites over a given spatial domain and  decreasing or increasing the spatial dependence
\citep{Zhang:Zimmerman:2005,Shaby:Kaufmann:2013},
we set the $\beta$ parameter, such that   the compact support  $\delta_{\nu,\mu,\beta}$   is identically equal to $0.15$ and $0.6$ for each scenario.
For instance, when $\nu=0$ and $\mu= \lambda(2,0)+2=3.5$ then $\beta=0.15/3.5$ to obtain a compact support equal to $\delta_{0,3.5,\beta} =0.15$.

In the  ML estimation of $(\sigma^2,\beta,\mu)^{\top}$ for the covariance model $\sigma^2\varphi_{\nu,\mu,\beta}$, we found  a reparameterization of the $\mu$ parameter to be useful by considering its inverse.
That is, we consider the ML estimation of $(\sigma^2,\beta,\mu^*)^{\top}$ where
  $\mu^{*}=1/\mu \in [0, 1/\lambda(2,\nu)]$ for the covariance model $\sigma^2\varphi_{\nu,1/\mu^*,\beta}$. 
  In the original parameterization, we found high variability in the ML estimates of the $\mu$ parameter, particularly for large values  of $\mu$.
  A similar pattern has been observed in literature when estimating  the degrees of freedom of the Student's $t$ distribution;
to alleviate this issue some authors  \citep{Ciccio:Monti:2011,AR22} propose to considering the estimation of   the inverse degrees of freedom.

Figure \ref{cont22m} reports
the boxplots of the centered and rescaled ML estimates $( \widehat{\mu_i^*}-\mu^* )/\sqrt{f_{11}}$, $( \widehat{\beta}_i-\beta )/\sqrt{f_{22}}$, $( \widehat{\sigma^2_i}-\sigma^2 )/ \sqrt{f_{33}}$, $i=1,\ldots,500$  (first, second and third rows, respectively),
when $\nu=0, 1, 2$ (first, second and third column respectively),  $\mu= \lambda(2,\nu)+x$ with $x=1, 2, 4$ ( for each subfigure )
 by considering increasing and fixed domain asymptotics scenarios $\delta=0.15, 0.6$ (left and right part of each subfigure respectively). Here $f_{ii}$ are the diagonal elements of the inverse of the Fisher information matrix
 in Equation (\ref{eq:fisher}).
 Using the asymptotic results under increasing domain asymptotics the displayed boxplots should be similar  to the boxplot of a Gaussian random variable.
 Overall  the asymptotic distribution seems to work reasonably well (at  least for  values between the  first and third quartiles)
 and, as expected, the asymptotic approximation worsens when switching from the increasing domain  ($\delta=0.15$)  to the fixed domain  ($\delta=0.60$) setting,
irrespective of  the values of $\mu$ and $\nu$.


\begin{figure}
\centering

\begin{tabular}{ccccccc}
& \multicolumn{2}{c}{\;\;\;\;\;\;$\nu=0$} &\multicolumn{2}{c}{\;\;\;\;\;\;$\nu=1$}&\multicolumn{2}{c}{\;\;\;\;\;\;$\nu=2$}\\
\vspace{-0.8cm}
&\;\;\;\;\;\;\; $\delta=0.15$ &$\delta=0.6$&\;\;\;\;\;\;\;  $\delta=0.15$ &$\delta=0.6$&\;\;\;\;\;\;\; $\delta=0.15$ &$\delta=0.6$\\\\[0.01pt]
\vspace{3.5cm}
\multirow{3}{*}{\begin{tabular}[c]{c}\\\\\\$\widehat{\mu^*}$ \end{tabular}}
\hspace{-0.9cm}
& \multicolumn{2}{c}{\multirow{3}{*}{\includegraphics[width=5.3cm, height=5.3cm]{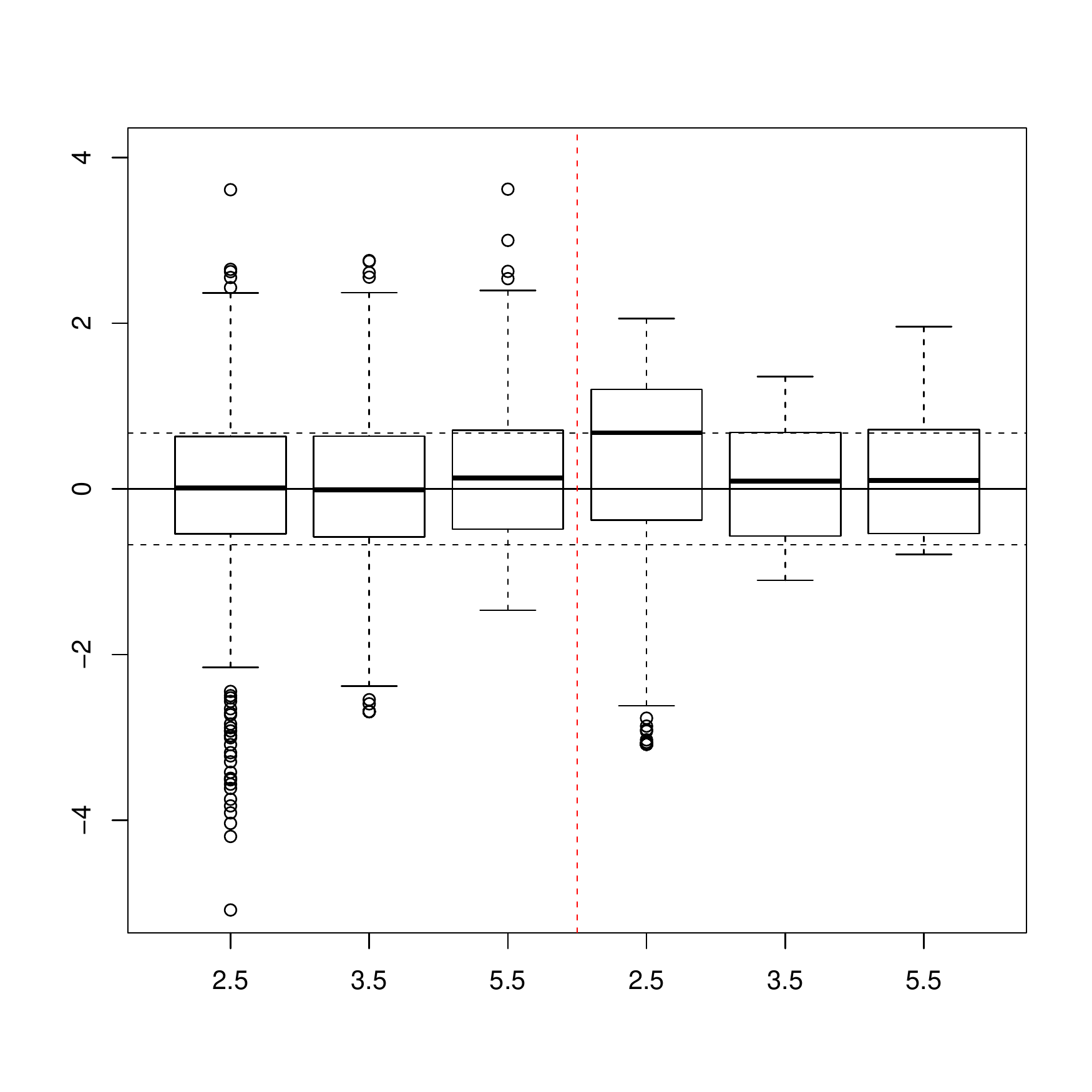}}}
\hspace{-0.9cm}
&\multicolumn{2}{c}{\multirow{3}{*}{\includegraphics[width=5.3cm, height=5.3cm]{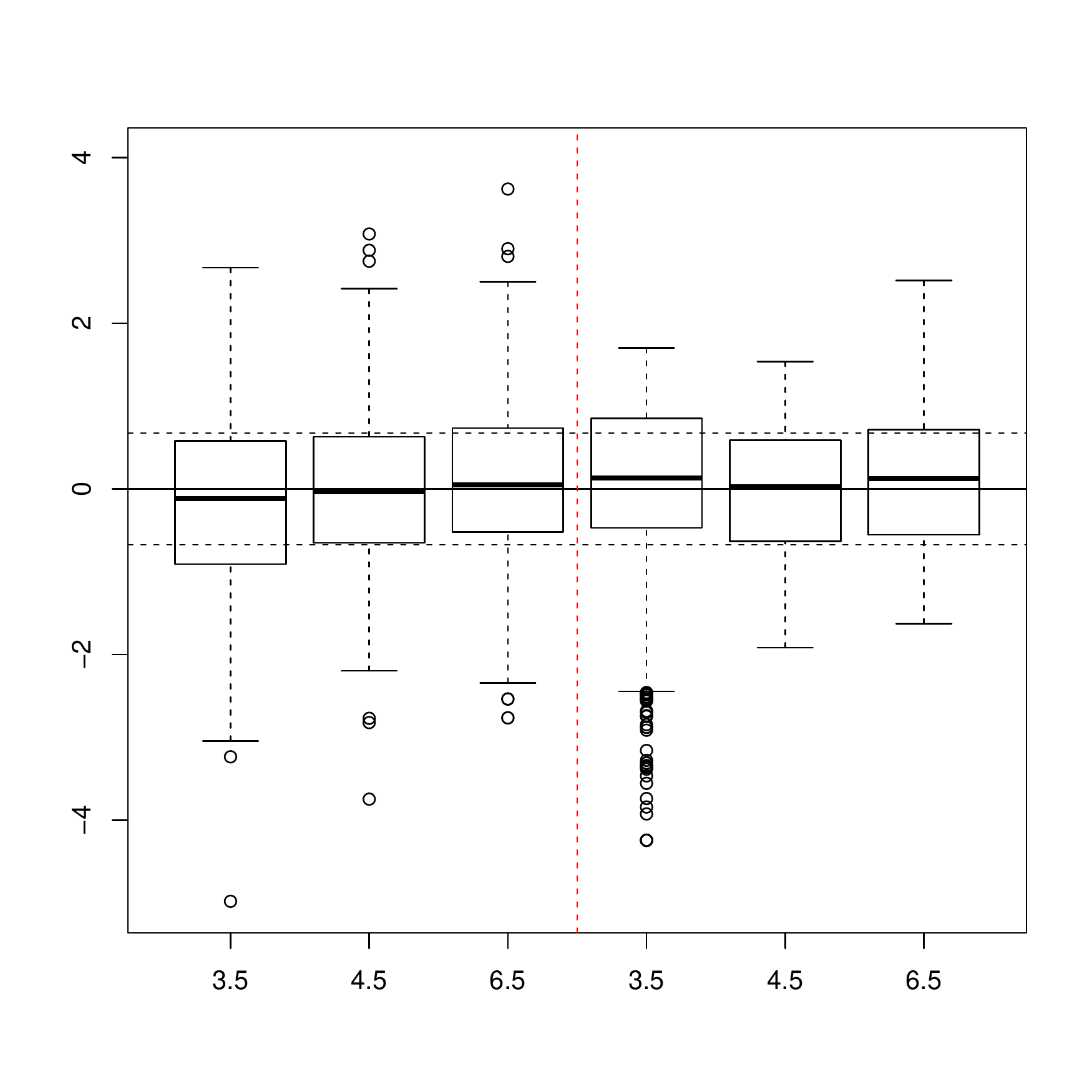}}}
\hspace{-0.9cm}
&\multicolumn{2}{c}{\multirow{3}{*}{\includegraphics[width=5.3cm, height=5.3cm]{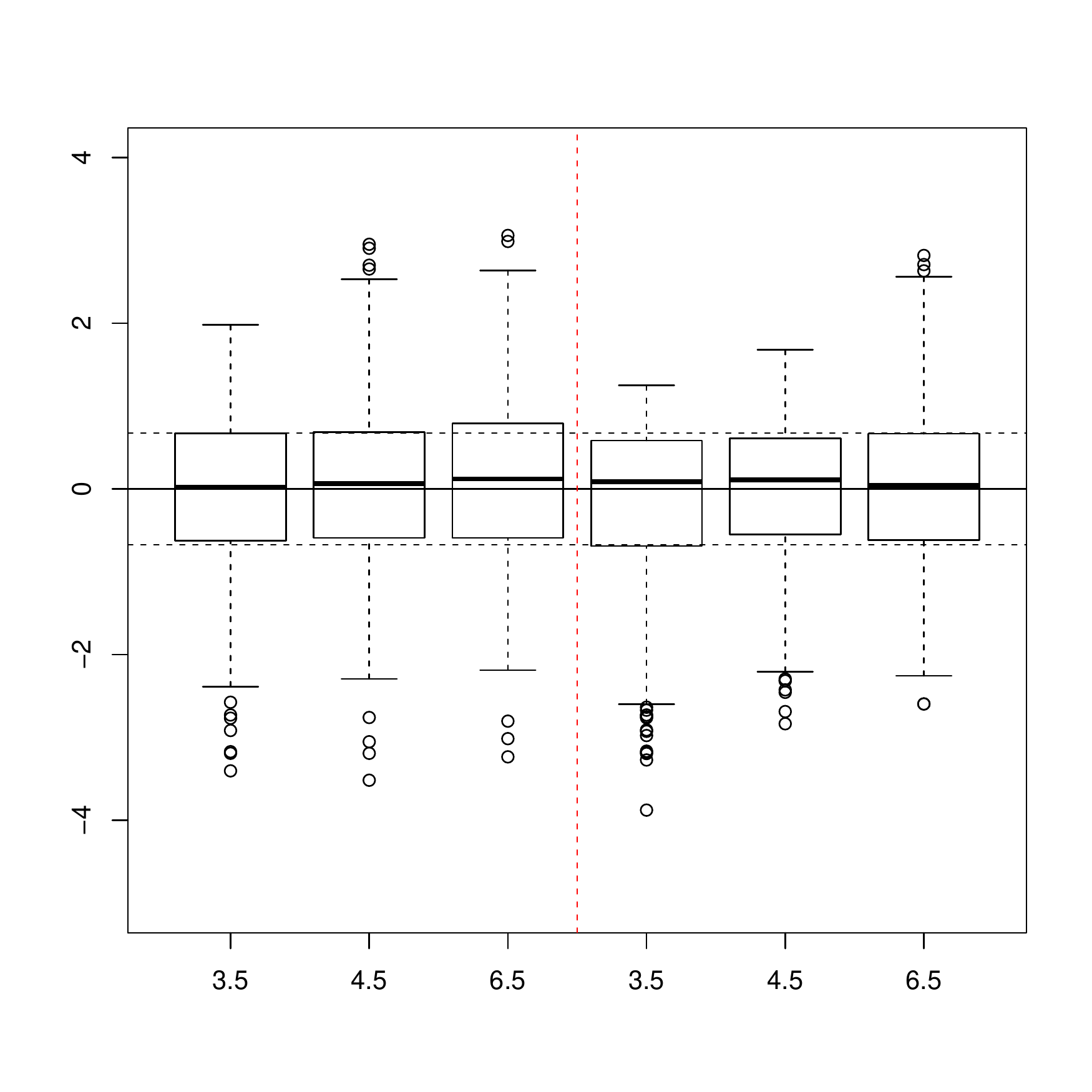}}} \\
\vspace{3.5cm}
\multirow{3}{*}{\begin{tabular}[c]{@{}ccc@{}}\\\\\\$\widehat{\beta}$\end{tabular}}
\hspace{-0.9cm}
& \multicolumn{2}{c}{\multirow{3}{*}{\includegraphics[width=5.3cm, height=5.3cm]{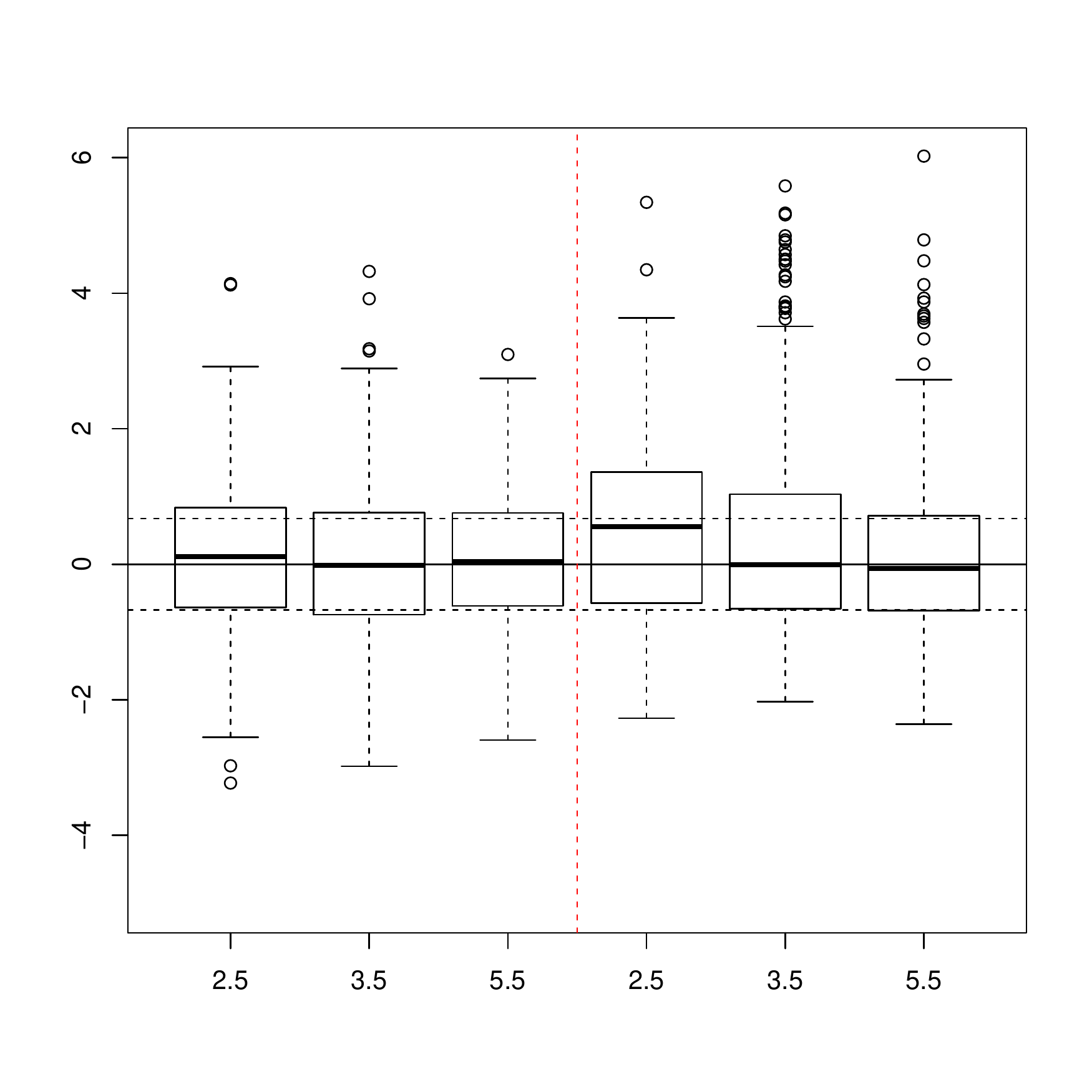}}}
\hspace{-0.9cm}
&\multicolumn{2}{c}{\multirow{3}{*}{\includegraphics[width=5.3cm, height=5.3cm]{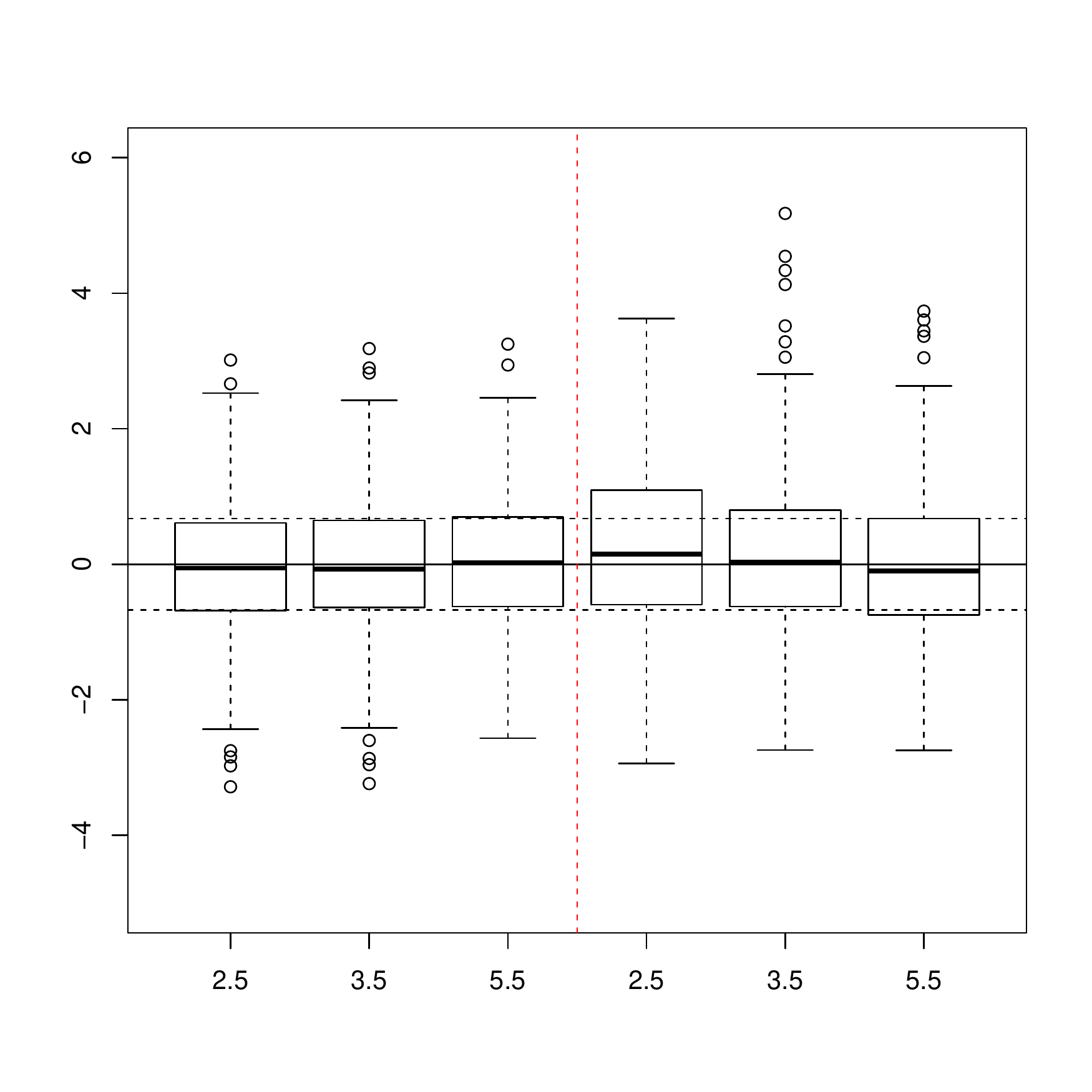}}}
\hspace{-0.9cm}
&\multicolumn{2}{c}{\multirow{3}{*}{\includegraphics[width=5.3cm, height=5.3cm]{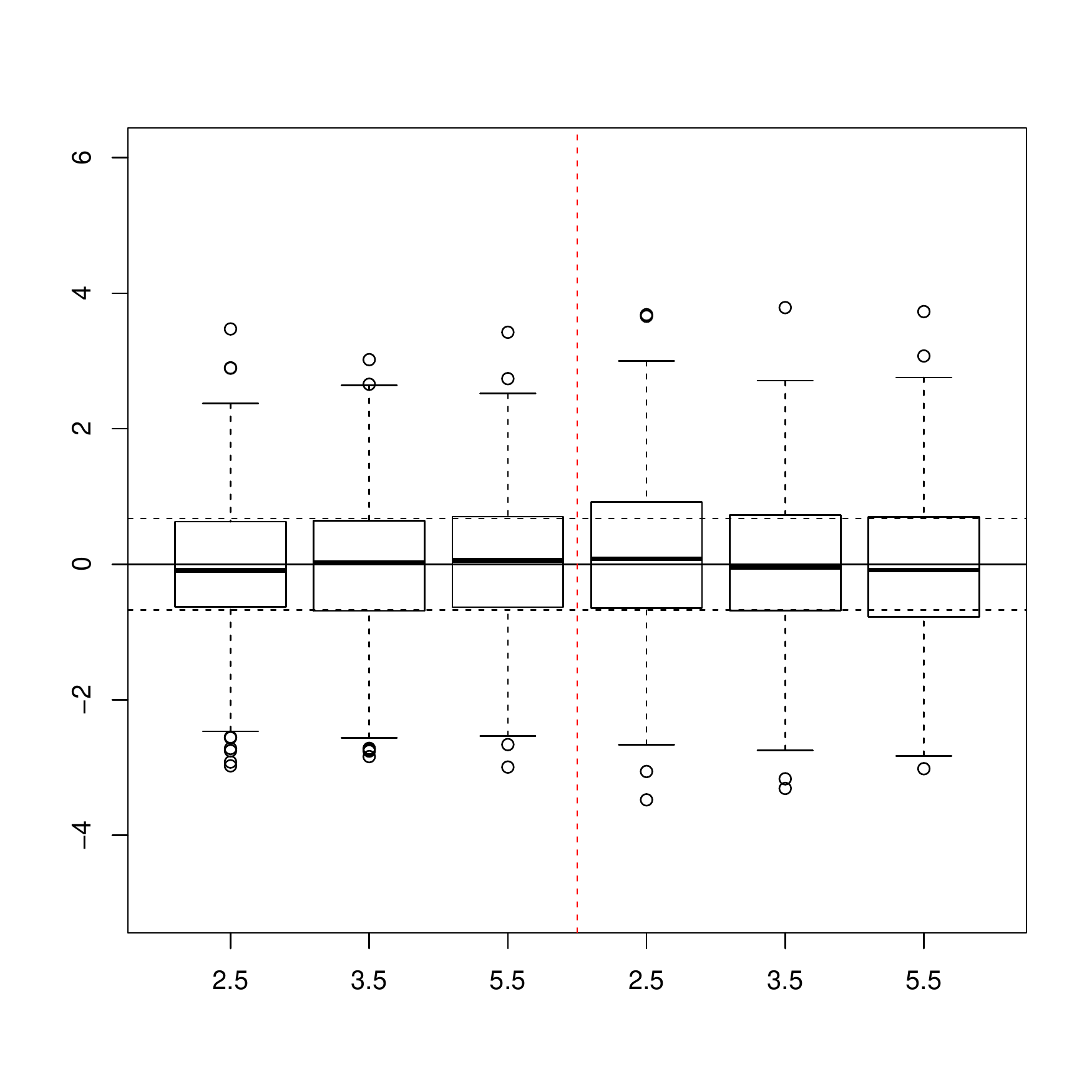}}} \\
\vspace{3.5cm}
\multirow{3}{*}{\begin{tabular}[c]{@{}ccc@{}}\\\\\\$\widehat{\sigma}^2$\end{tabular}}
\hspace{-0.9cm}
&\multicolumn{2}{c}{\multirow{3}{*}{\includegraphics[width=5.3cm, height=5.3cm]{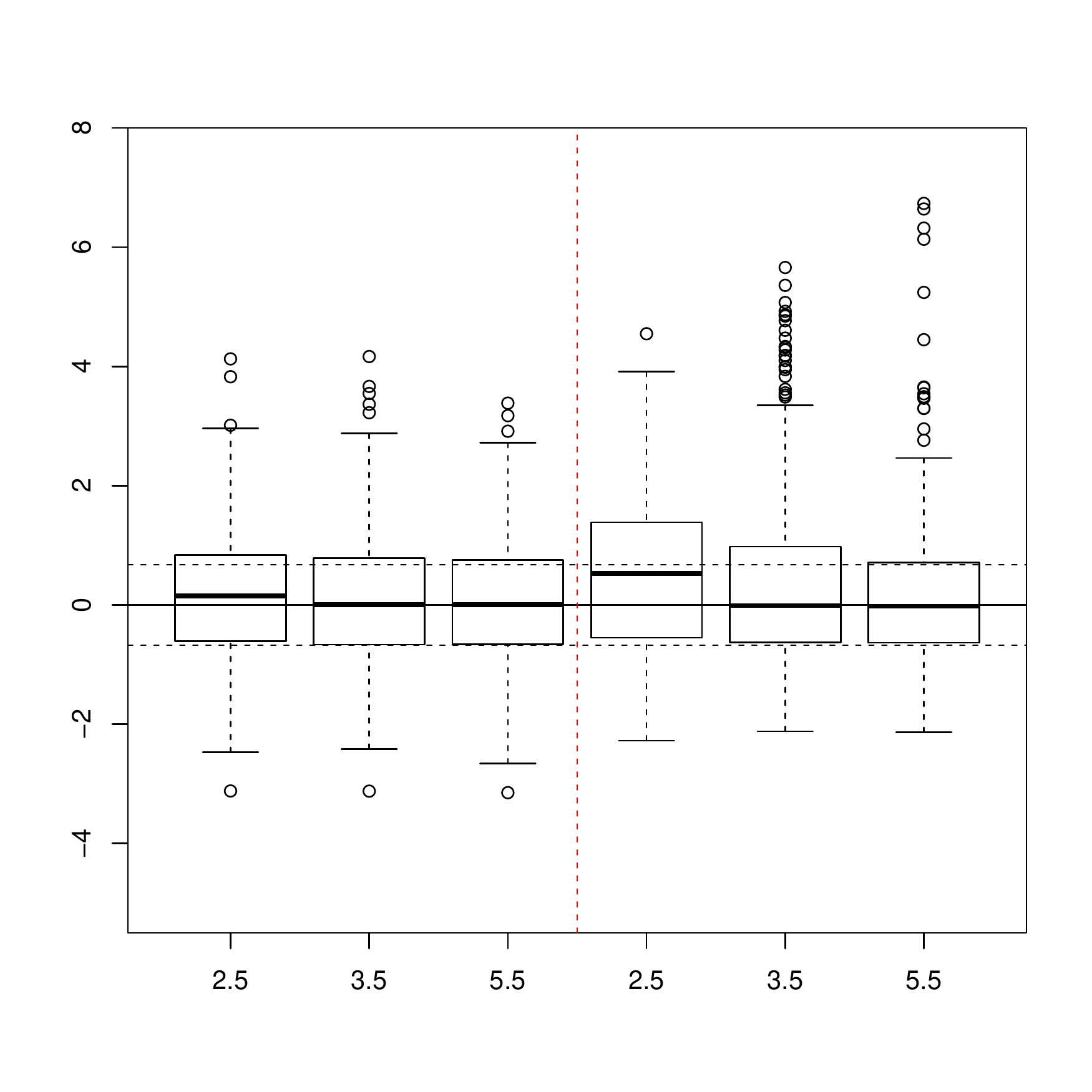}}}
\hspace{-0.9cm}
&\multicolumn{2}{c}{\multirow{3}{*}{\includegraphics[width=5.3cm, height=5.3cm]{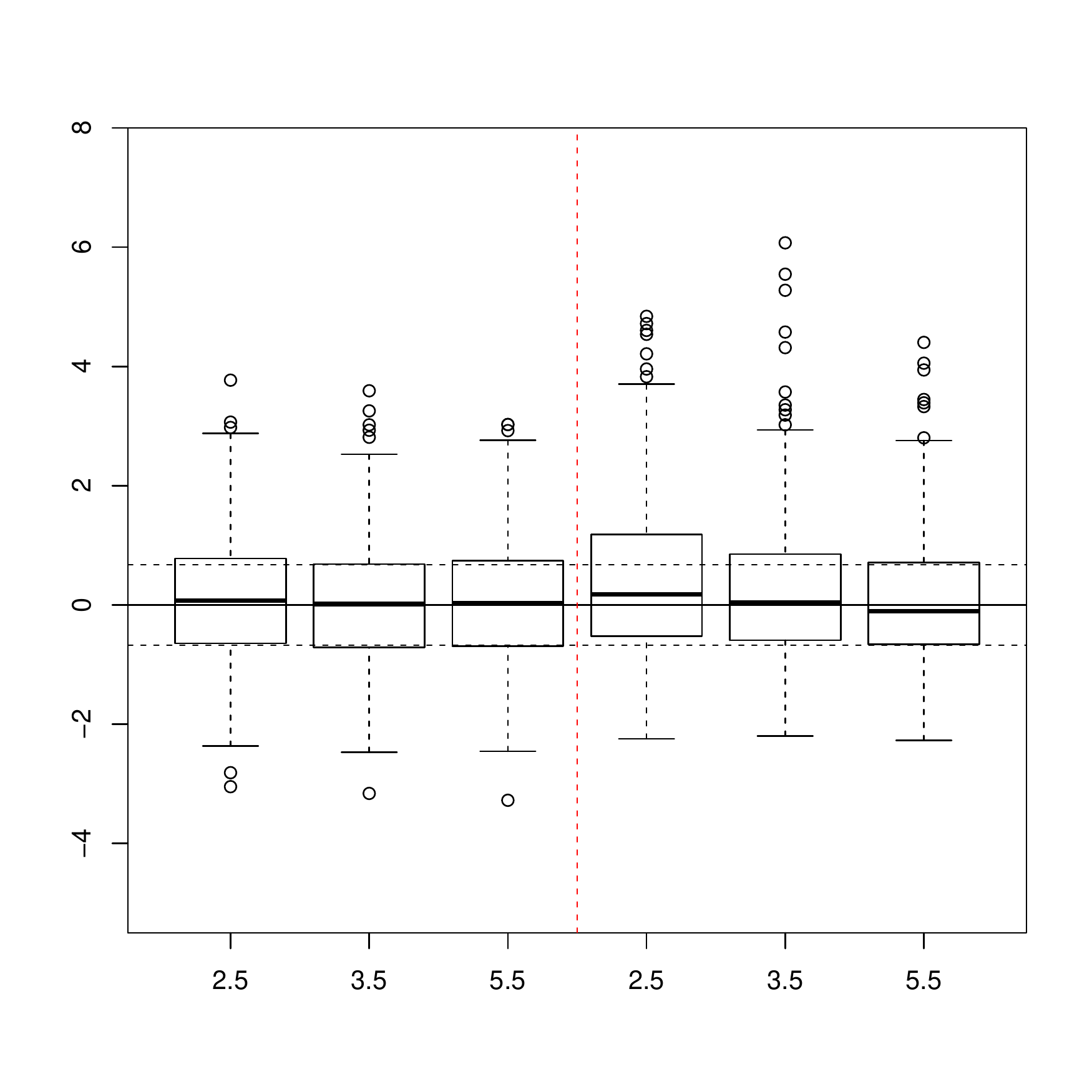}}}
\hspace{-0.9cm}
&\multicolumn{2}{c}{\multirow{3}{*}{\includegraphics[width=5.3cm, height=5.3cm]{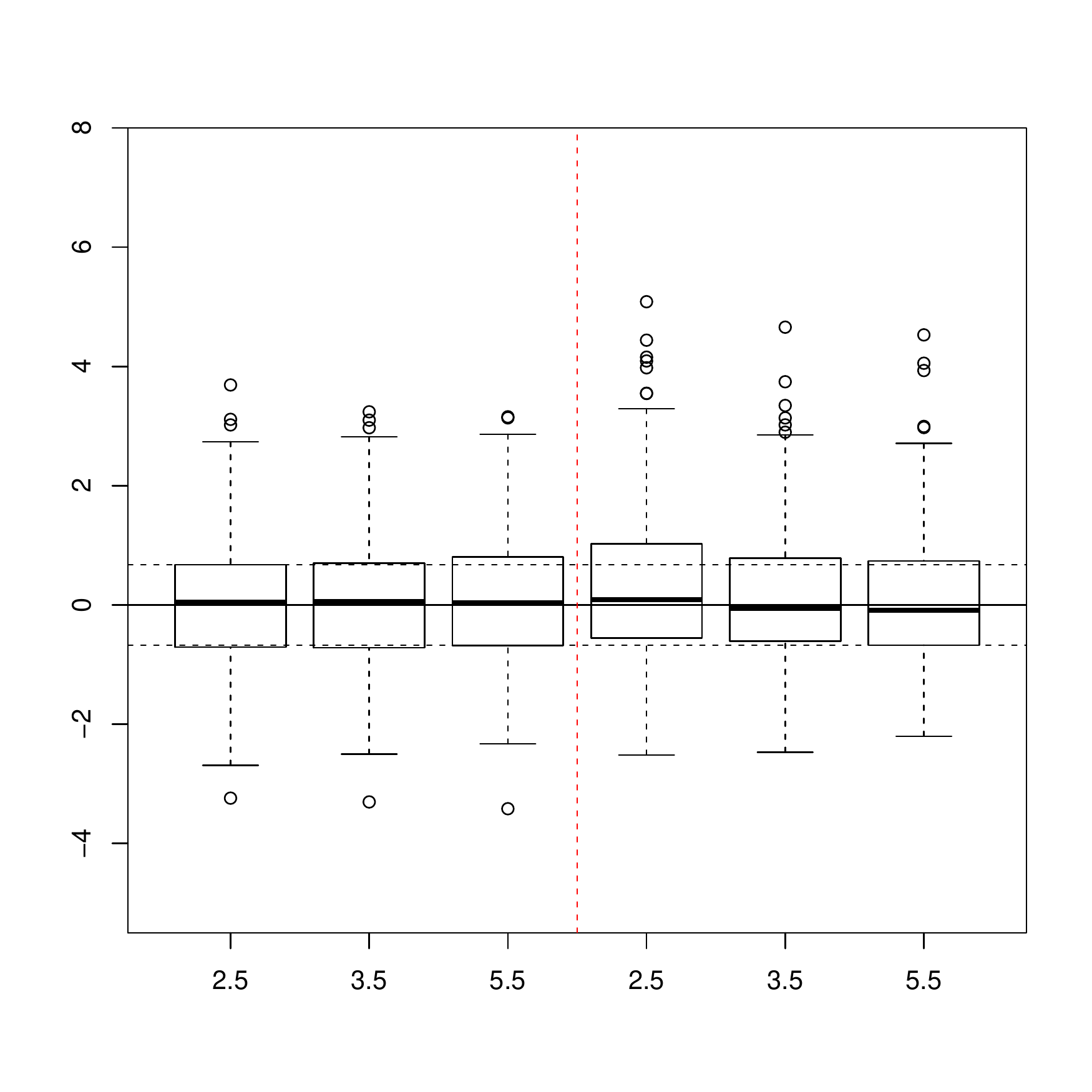}}} \\
\vspace{3.5cm}
\multirow{3}{*}{\begin{tabular}[c]{@{}ccc@{}}\\\\\\$m( \widehat{\sigma}_i^{2},\widehat{\beta}_i)$ \end{tabular}}
\hspace{-0.9cm}
&  \multicolumn{2}{c}{\multirow{3}{*}{\includegraphics[width=5.3cm, height=5.3cm]{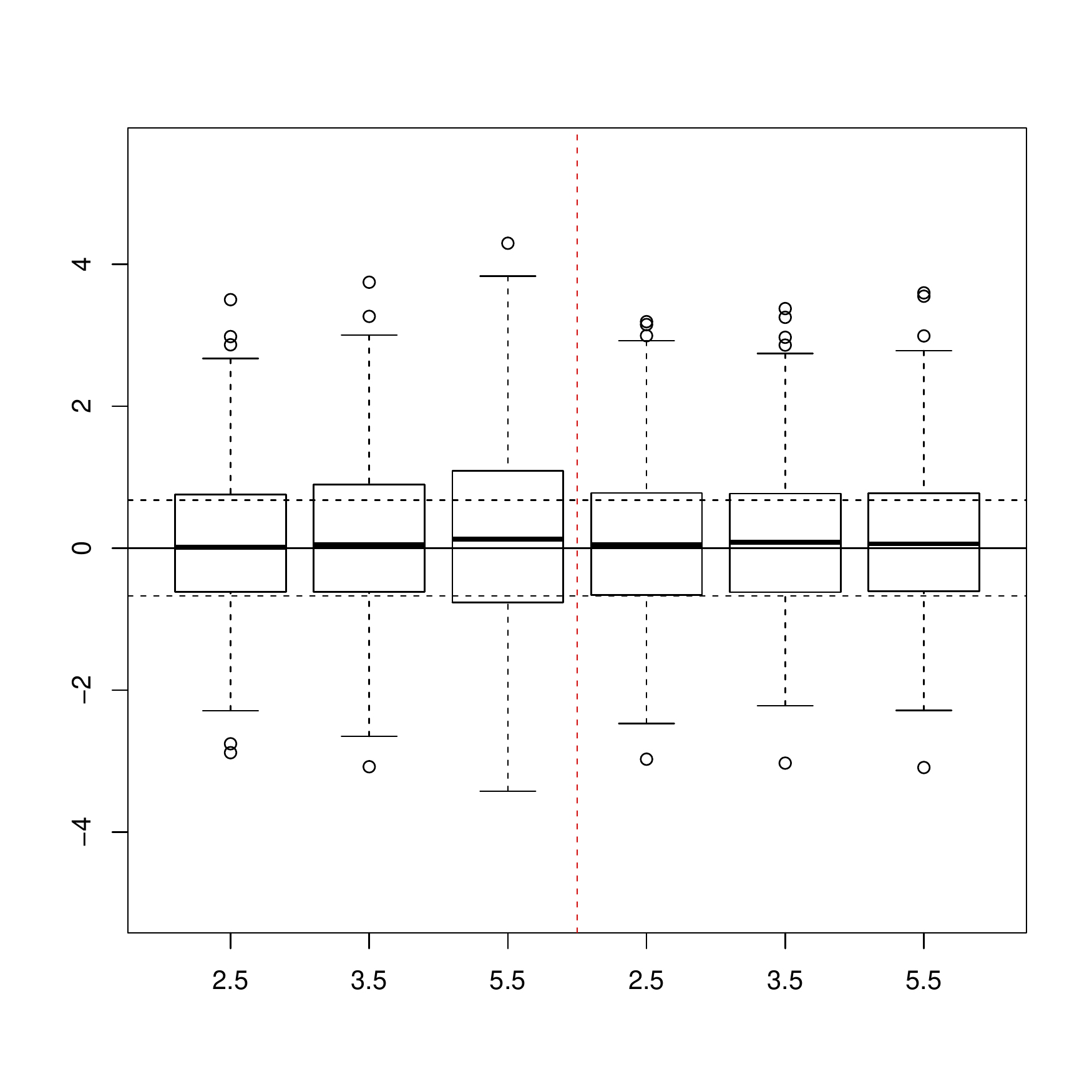}}}
\hspace{-0.9cm}
&\multicolumn{2}{c}{\multirow{3}{*}{\includegraphics[width=5.3cm, height=5.3cm]{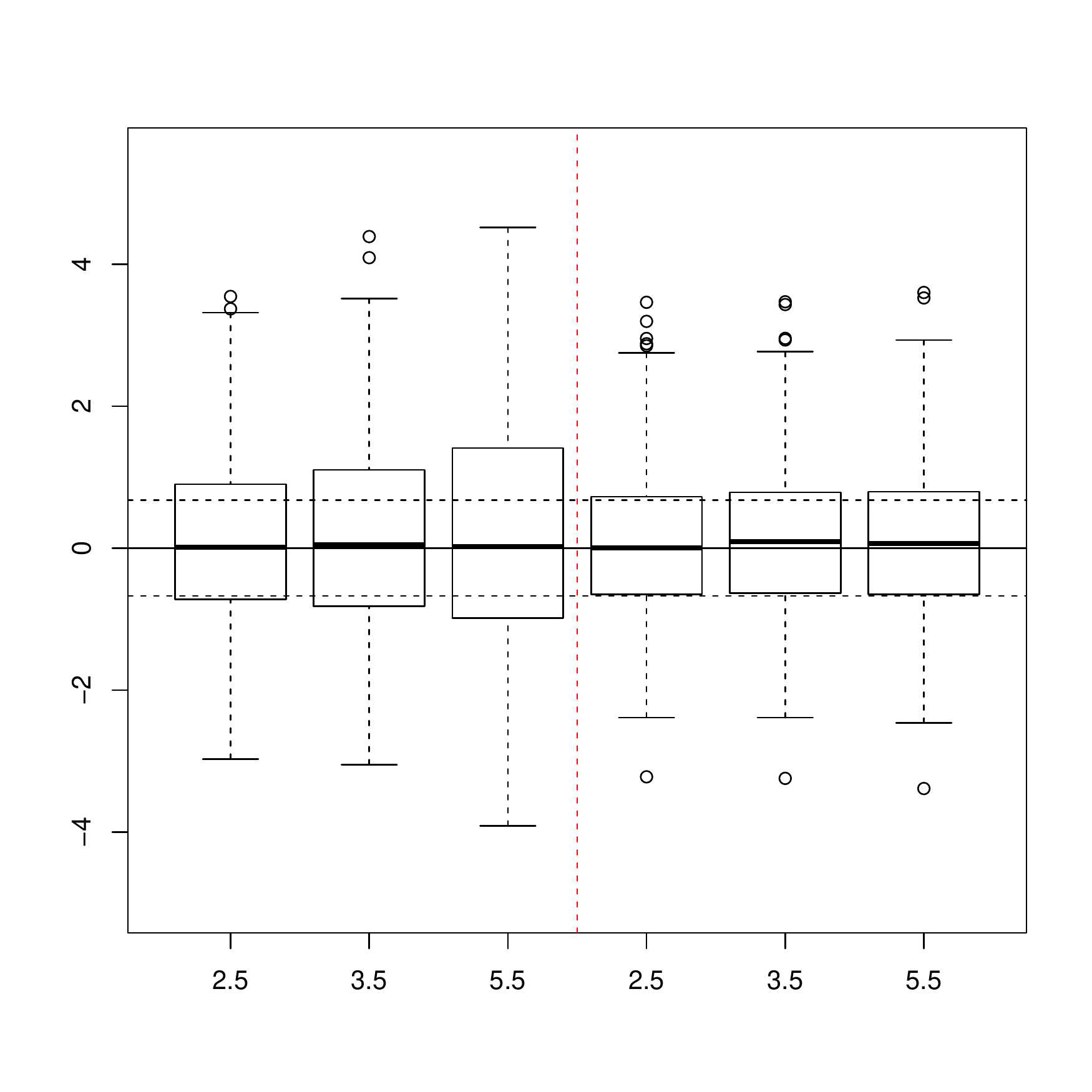}}}
\hspace{-0.9cm}
&\multicolumn{2}{c}{\multirow{3}{*}{\includegraphics[width=5.3cm, height=5.3cm]{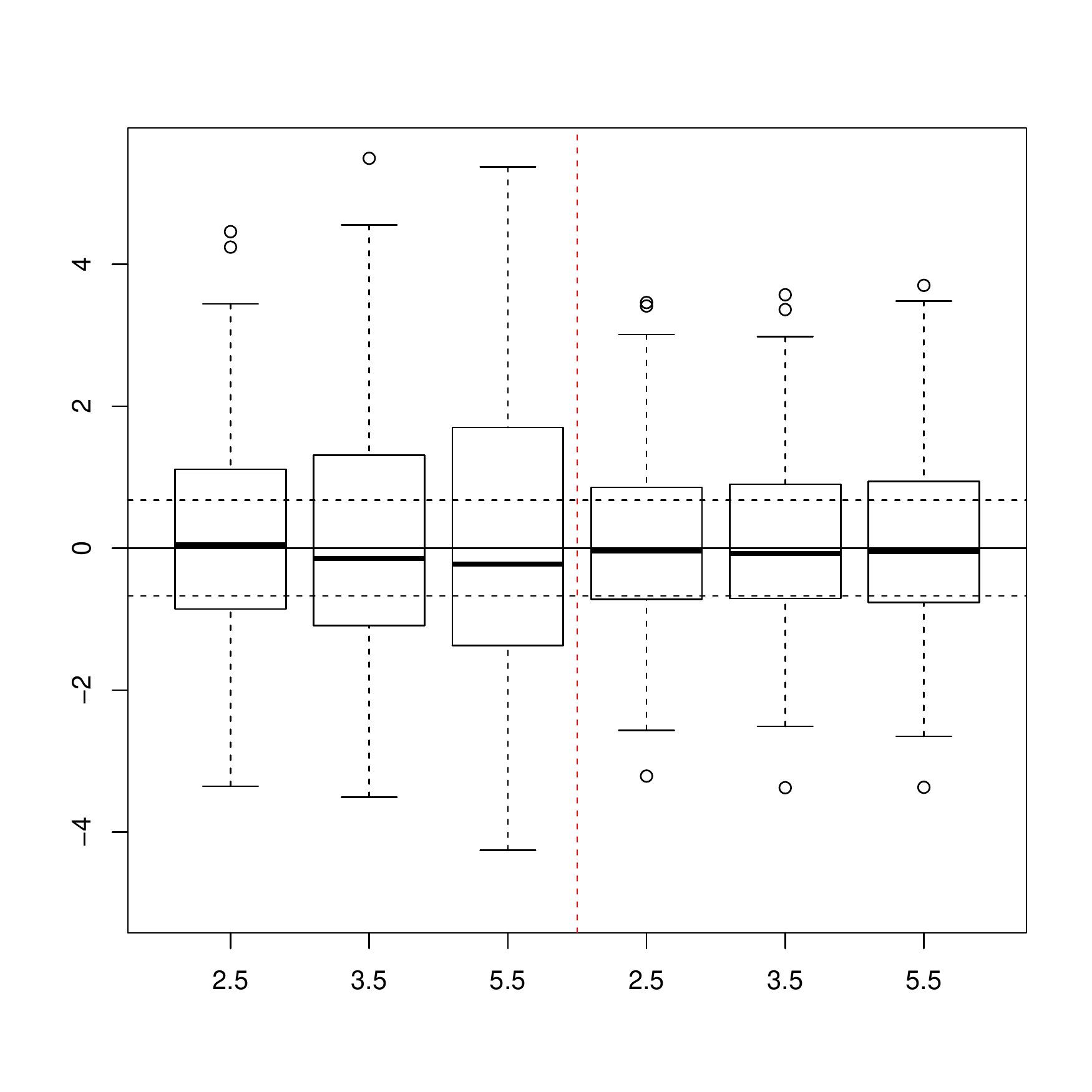}}} \\\\\\
\end{tabular}
\hspace{1.2cm}

\caption{Boxplots  of the centered and rescaled ML estimates  $( \widehat{\mu_i^*}-\mu^* )/\sqrt{f_{11}}$, $( \widehat{\beta}_i-\beta )/\sqrt{f_{22}}$, $( \widehat{\sigma^2_i}-\sigma^2 )/ \sqrt{f_{33}}$,  $i=1,\ldots,500$
of the covariance  model $\sigma^2\varphi_{\nu,1/\mu^*,\beta}$ with $\mu^*=1/\mu$ when $\nu=0, 1, 2$, $\mu= \lambda(2,\nu)+ x$, $x=1, 2, 4$ and
$\beta$ is such that  the  compact support   $\delta_{\nu,\mu,\beta}=0.15, 0.6$    (left  and right part of each subfigure). Last row:  boxplots of
  $m( \widehat{\sigma}_i^{2},\widehat{\beta}_i)=\sqrt{n/2} ( \widehat{\sigma^2}_i(\widehat{\beta}_i,\mu) \widehat{\beta}_i^{-(1+2\kappa)}/\sigma^{2} \beta^{-(1+2\kappa)}-1 )$, $i=1,\ldots,500$.
  The horizontal dotted lines represent the quantiles of the order $0.25$ and $0.75$
of the standard Gaussian distribution.}  \label{cont22m}
\end{figure}

To  analyze the approximation given by the asymptotic distribution under fixed domain  of the microergodic parameter $\sigma^{2} \beta^{-(1+2\kappa)}g(\nu,\mu)$, we replicate the previous numerical experiment using the same simulation settings
 but this time, we  assume that $\mu$ is known and fixed.
Last row  of Figure \ref{cont22m}  depicts  the  boxplots of    $m( \widehat{\sigma}_i^{2},\widehat{\beta}_i)=\sqrt{n/2} ( \widehat{\sigma}_i^{2} \widehat{\beta}_i^{-(1+2\kappa)}/\sigma^{2} \beta^{-(1+2\kappa)}-1 )$,
$i=1,\ldots, 500$, for each $\nu$ and $\mu$.
Also in this case the boxplots should be similar to the boxplot of a standard Gaussian random variable.
As expected, the asymptotic approximation works much better under fixed domain asymptotics ($\delta=0.60$), and it seems to improve with decreasing $\nu$.
In addition, under increasing domain the approximation   clearly deteriorates when increasing both $\mu$ and $\nu$.

\section{Data Examples}\label{sec6}
We consider two data examples that explain, from our perspective,
how the proposed model should be used depending on the size of the available dataset.
 The first  approach involves the estimation of the $\mu$ parameter and should be applied to (not necessarily) small  spatial datasets with the goal of looking for an improvement of the Mat{\'e}rn family
 from modeling viewpoint.
The second approach is more suitable for large datasets and considers an arbitrary fixed $\mu$. In this case  the goal is to seek
highly  sparse matrices   to reduce the computational complexity.

\subsection{{\bf Application to Mean Temperature Data}}

We consider data from
 WorldClim (\texttt{www.worldclim.org}) a global  database of high spatial resolution global weather and climate data
  for the years 1970-2000 \citep{wc2}.
 In particular, we consider  mean temperature data of September over a specific region of French
 (see Figure \ref{figusa} (a))
observed at $624$ geo-referenced location sites.
Following \cite{LI20111445}, we first detrend the data using splines to remove the cyclic
pattern of both variables along the longitude and latitude directions, and then regard the residuals  $y(\bm{s}_i)$, $i=1,\ldots,n$, $n=624$
as a realization from a zero mean  Gaussian RF with isotropic covariance function $\sigma^2\rho(r)$ (the empirical semi-variogram  is depicted in Figure \ref{figusa} (b)).
For the  isotropic correlation function  $\rho(r)$
 we specify the proposed reparametrized Generalized Wendland correlation model.
In particular, we consider  $\varphi_{\nu,\mu,\beta}$  for $\nu=0, 1$ and the associated special limit case,
that is the Mat{\'e}rn  model ${\cal M}_{\nu+1/2,\beta}$ for $\nu=0, 1$.

\begin{figure}[h!]
\centering{
\begin{tabular}{cc}
  \includegraphics[width=5.2cm, height=5.2cm]{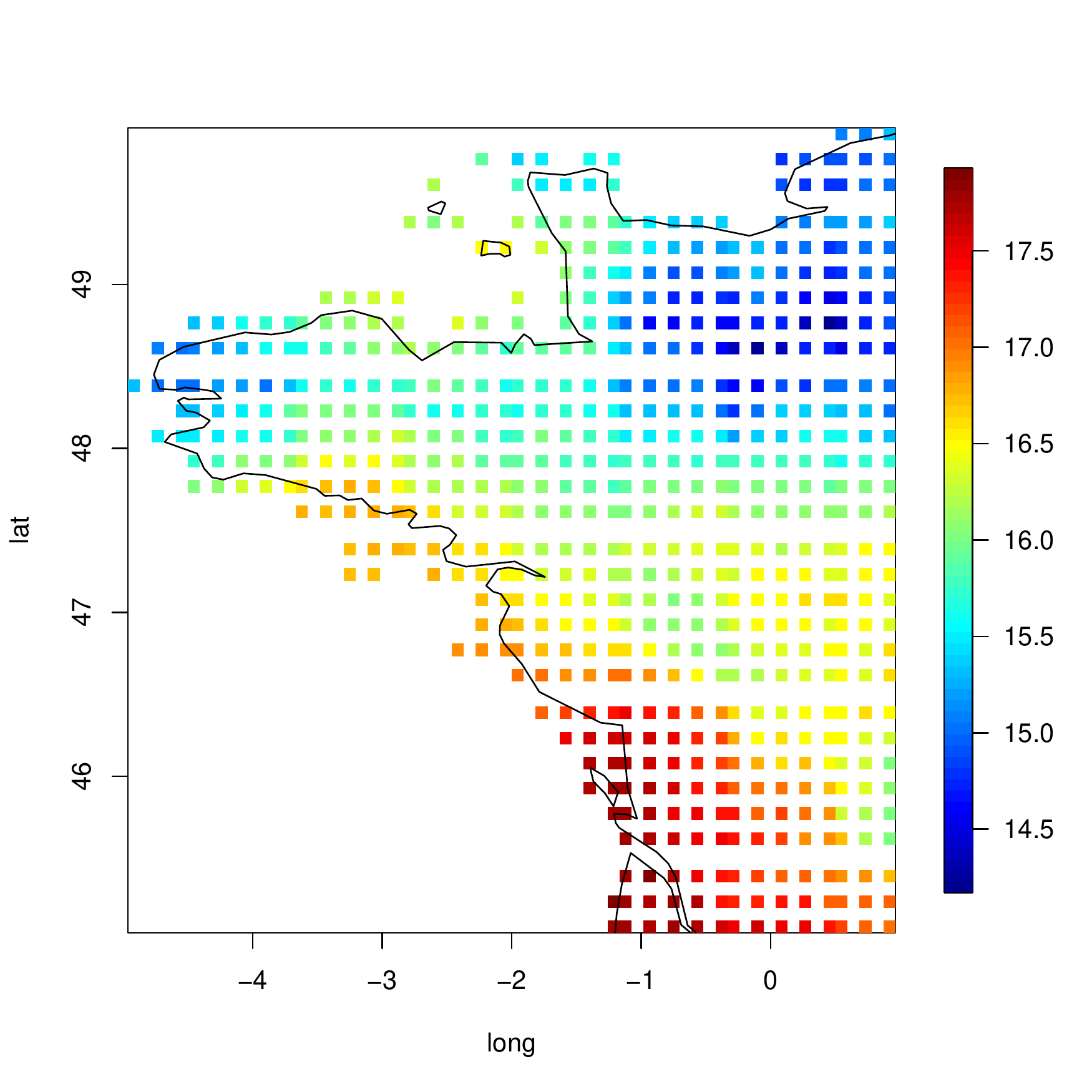}& \includegraphics[width=5.2cm, height=5.2cm]{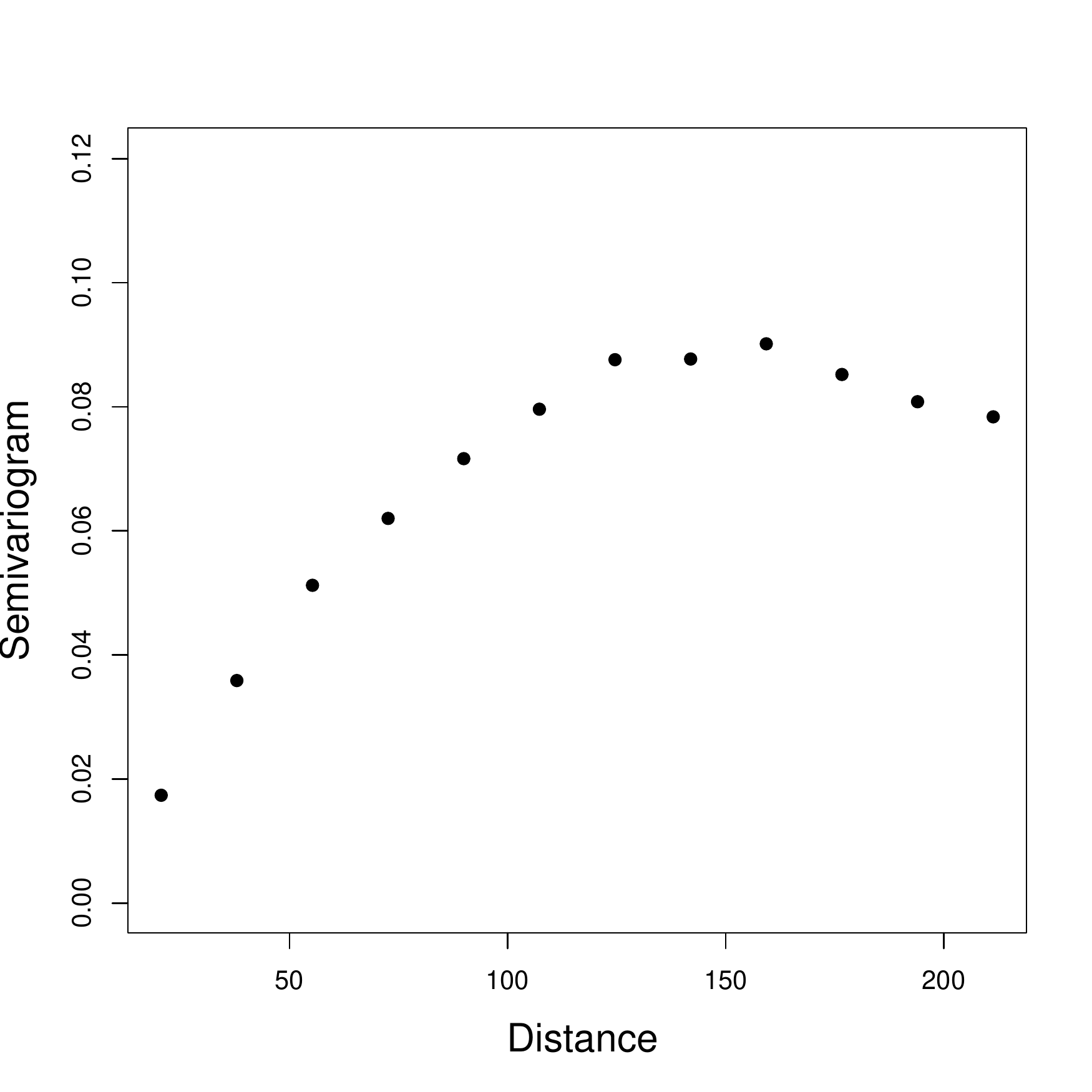}\\
  (a)&(b)
  \end{tabular}
    \caption{From left to right: a) spatial locations of   mean temperature of September  in a specific region of  France
  and b)  empirical semivariogram of the residuals after detrending the original data.}\label{figusa}}
\end{figure}

Here, we adopt an increasing domain approach  by estimating   all parameters of the covariance models  with ML  method,
  and we compute the associated standard error estimation,  as the square root of diagonal elements of the inverse of the Fisher Information matrix (\ref{eq:fisher}).
For the $\mu$ parameter  we use the parameterization described in Section \ref{mlll}  that is, $\mu^*=1/\mu \in (0,\lambda(2,\nu)^{-1}]$ and the
Mat{\'e}rn  model, under this parametrization,  is attained when  $\mu^* \to 0$.

The results of the estimation are summarized in Table   \ref{tabnn}, where we also report the values of the maximized log-likelihoods and the associated values of Akaike information criterion (AIC).
It can be appreciated that the covariance model  $\sigma^2\varphi_{0,\mu^*,\beta}$  achieves the  lower AIC with respect to the Mat{\'e}rn  model
 $\sigma^2{\cal M}_{0.5,\beta}$. When $\nu=1$, the reparametrized
 Generalized Wendland model $\sigma^2\varphi_{1,\mu^*,\beta}$
 coincides with
 the  Mat{\'e}rn  model
 ${\cal M}_{1.5,\beta}$ since the estimation of the $\mu^*$ parameter collapses to the lower bound. For this reason
  in Table   \ref{tabnn} we only report the estimates of  the covariance model $\sigma^2{\cal M}_{1.5,\beta}$.
  Overall the best fitted covariance model is $\sigma^2\varphi_{0,\mu^*,\beta}$. Figure \ref{fig114}  provides a graphical
comparison between the empirical and estimated semivariograms using  the
$\sigma^2{\cal M}_{0.5,\beta}$ and   $\sigma^2\varphi_{0,\mu^*,\beta}$ covariance models,  respectively.

  \begin{figure}[h!]
\centering{
\begin{tabular}{cc}
  \includegraphics[width=5.2cm, height=5.2cm]{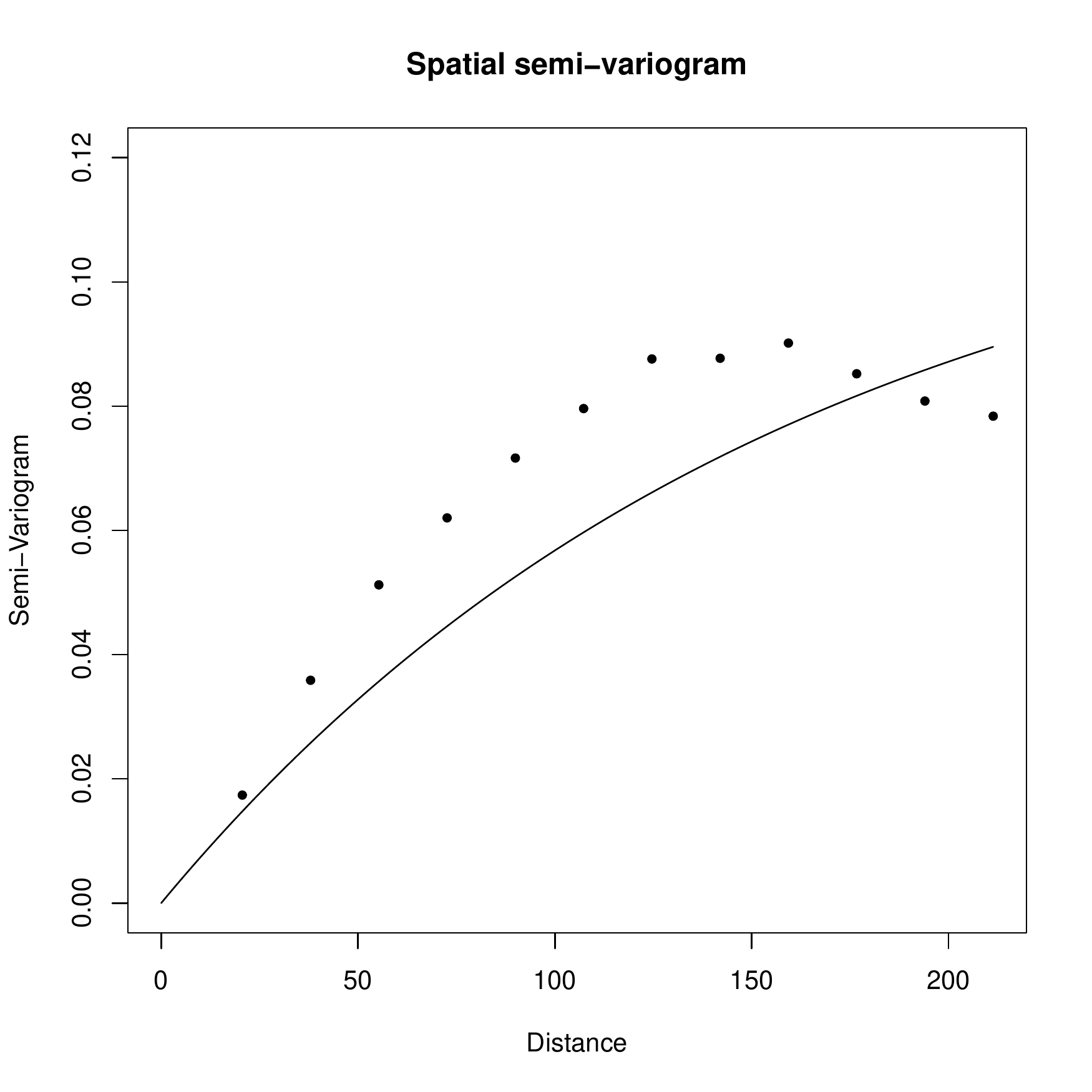}& \includegraphics[width=5.2cm, height=5.2cm]{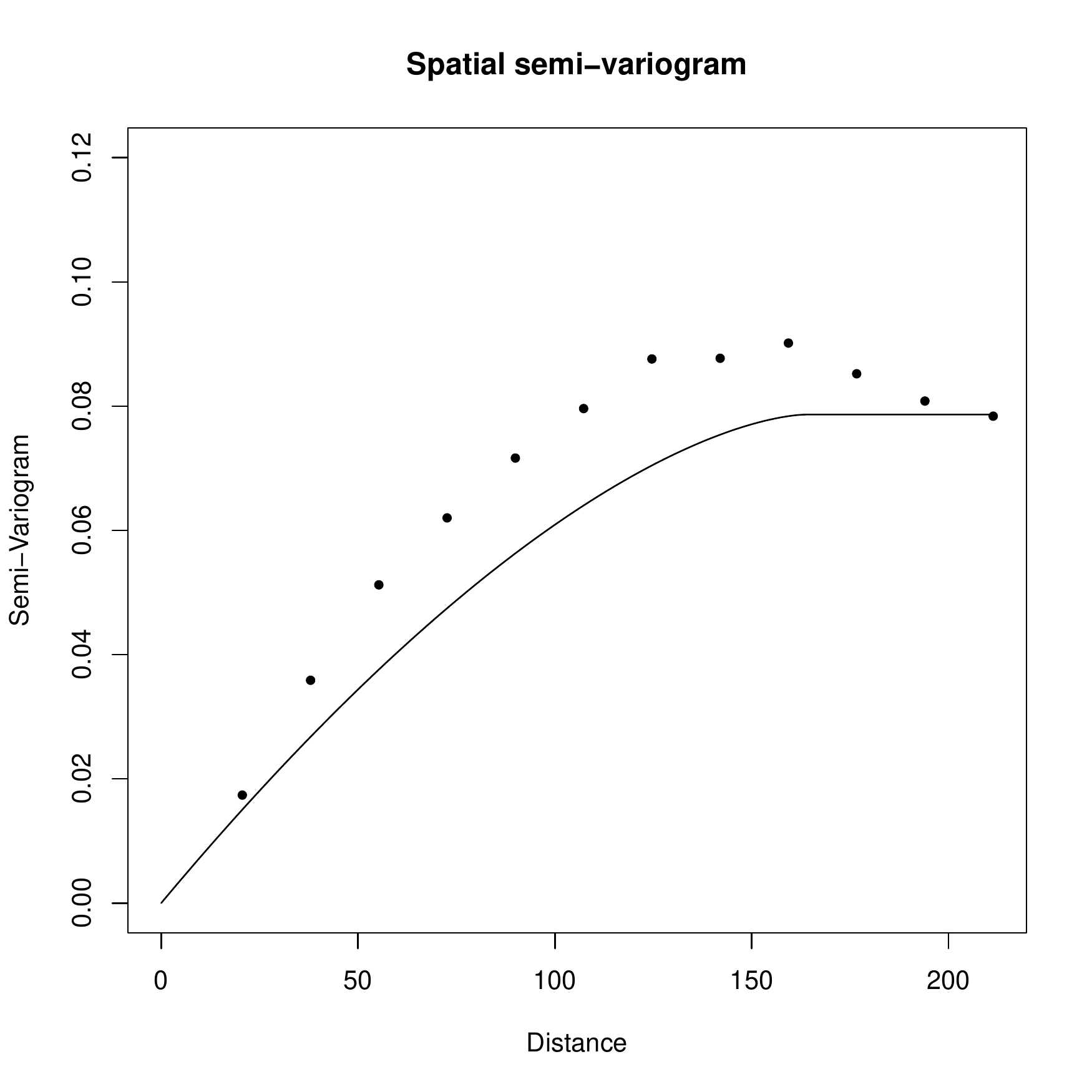}\\
  \end{tabular}
    \caption{
    Empirical semivariogram versus estimated semivariogram using the Mat{\'e}rn model  $\sigma^2{\cal M}_{0.5,\beta}$ (left part)
and the  Generalized Wendland model $\sigma^2\varphi_{0,\mu^*,\beta}$ (right part) covariance models.}\label{fig114}}
\end{figure}

\begin{table}[htbp]
\caption{ML estimates with associated standard error (in parentheses), RMSE, LSCORE and CRPS  for the three Gaussian RFs
with underlying covariance functions  $\sigma^2\varphi_{0,\mu^*,\beta}$,
its special limit case $\sigma^2{\cal M}_{0.5,\beta}$ and  $\sigma^2{\cal M}_{1.5,\beta}$.} \label{tabnn}
\begin{center}
\scalebox{0.6}{
\begin{tabular}{c|c|c|c|c|c|c|c|c|}
\cline{2-9}
 & $\beta$ & $\sigma^2$ & $\mu^*$ & -loglik & AIC & RMSE & LSCORE & CRPS \\ \hline
\multicolumn{1}{|c|}{\multirow{2}{*}{$\sigma^2{\cal M}_{0.5,\beta}$}} & \multirow{2}{*}{\begin{tabular}[c]{@{}c@{}}$160.0165$\\ $(79.263)$\end{tabular}} & \multirow{2}{*}{\begin{tabular}[c]{@{}c@{}}$0.1222$\\ $(0.059)$\end{tabular}} & \multirow{2}{*}{} & \multirow{2}{*}{$464.25$} & \multirow{2}{*}{$-924.5$} & \multirow{2}{*}{$0.0932$} & \multirow{2}{*}{$0.1965$} & \multirow{2}{*}{$0.0923$} \\
\multicolumn{1}{|c|}{} &  &  &  &  &  &  &  &  \\ \hline
\multicolumn{1}{|c|}{\multirow{2}{*}{$\sigma^2{\cal\varphi}_{0,\mu^*,\beta}$}} & \multirow{2}{*}{\begin{tabular}[c]{@{}c@{}}$103.7678$\\ $(7.239)$\end{tabular}} & \multirow{2}{*}{\begin{tabular}[c]{@{}c@{}}$0.0787$\\ $(0.008)$\end{tabular}} & \multirow{2}{*}{\begin{tabular}[c]{@{}c@{}}$0.6342$\\ $(0.047)$\end{tabular}} & \multirow{2}{*}{$470.86$} & \multirow{2}{*}{$-935.7$} & \multirow{2}{*}{$0.0926$} & \multirow{2}{*}{$0.1950$} & \multirow{2}{*}{$0.0916$} \\
\multicolumn{1}{|c|}{} &  &  &  &  &  &  &  &  \\ \hline
\multicolumn{1}{|c|}{\multirow{2}{*}{$\sigma^2{\cal M}_{1.5,\beta}$}} & \multirow{2}{*}{\begin{tabular}[c]{@{}c@{}}$16.3599$\\ $(0.899)$\end{tabular}} & \multirow{2}{*}{\begin{tabular}[c]{@{}c@{}}$0.06163$\\ $(0.008)$\end{tabular}} & \multirow{2}{*}{} & \multirow{2}{*}{$457.65$} & \multirow{2}{*}{$-911.3$} & \multirow{2}{*}{$0.0947$} & \multirow{2}{*}{$0.2627$} & \multirow{2}{*}{$0.0928$} \\
\multicolumn{1}{|c|}{} &  &  &  &  &  &  &  &  \\ \hline
\end{tabular}
}
\end{center}
\end{table}

We further evaluate the predictive performances of the three Gaussian RFs. We use the following resampling approach: we randomly choose 80\% of the spatial locations
and we use the remaining 20\%  as data for the predictions.
 We then use the estimates  in Table  \ref{tabnn}   to compute  three prediction scores \citep{Gneiting:Raftery:2007}
  for each Gaussian   RF.
Specifically, for each $j-th$ left-out  sample   $(y_j^L(\bm{s}_1), \ldots, y_j^L(\bm{s}_K))$, for $j=1,\ldots, 1000$ we compute
\begin{enumerate}
  \item the root mean squared error
\begin{equation*}
\overline{\mathrm{RMSE}}_j=\left[ \frac{1}{K} \sum_{i= 1}^K\left( y_j^L(\bm{s}_i)-\widehat{Y}_j^L(\bm{s}_i)\right)^2\right]^{\frac{1}{2}}
\end{equation*}
  \item  the logarithmic score
\begin{equation}\label{LSCORE}
\overline{\log S}_j=\frac{1}{K}\sum_{i=1}^{K}\left[  \frac{1}{2} \log \{2 \pi \sigma^L_j(\bm{s}_i)\}
+  \frac{1}{2} \{z^L_j(\bm{s}_i)\}^2 \right],
\end{equation}
 \item  the continuous ranked probability
\begin{equation}
\overline{{\rm CPRS}}_j= \frac{1}{K}\sum_{i=1}^{K}  \sigma_j(\bm{s}_i) \left (z^L_j(\bm{s}_i) \left ( 2\Phi(z^L_j(\bm{s}_i))-1\right ) +2 \Phi(z^L_j(\bm{s}_i)) -\frac{1}{\sqrt{\pi}} \right ), \label{CPRS}
\end{equation}
\end{enumerate}
where  $\widehat{Y}_j^L(\bm{s}_i)$ is the optimal linear  predictor, $\sigma^L_j(\bm{s}_i) $ is the corresponding square root variance
and $z^L_j(\bm{s}_i)=(y^L_j(\bm{s}_i)-\widehat{Y}_j^L(\bm{s}_i))/  \sigma^L_j(\bm{s}_i)$.
 Table   \ref{tabnn}   reports the overall means  $\mathrm{RMSE}=\sum_{j=1}^{1000}\overline{\mathrm{RMSE}}_j/1000$, 
 $\mathrm{\log S}=\sum_{j=1}^{1000}\overline{\log S}_j/1000$ and  $\mathrm{CRPS}=\sum_{j=1}^{1000}\overline{{\rm CPRS}}_j/1000$
  for each of the eight  Gaussian RFs.
As expected,  the covariance model
 $\sigma^2\varphi_{0,\mu^*,\beta}$
 outperforms the  Mat{\'e}rn model  limit case  ${\cal M}_{0.5,\beta}$   and the  Mat{\'e}rn   model  ${\cal M}_{1.5,\beta}$
   for the three  prediction scores considered.

 \subsection{{\bf   Application to Yearly total precipitation anomalies}}\label{secondapp}
We consider the dataset in \citet{Kaufman:Schervish:Nychka:2008} of yearly total precipitation anomalies
$\bm{z}=\{ z(\bm{s}_i)$, $i=1,\ldots,n$\}
registered at  $n=7,352$ location sites
in the USA  since 1895 to 1997.
The   yearly totals have been standardized by the
long-run mean and standard deviation for each station from 1962 (Figure \ref{figusaaa}, right part).
\citet{Kaufman:Schervish:Nychka:2008} adapted  a zero-mean Gaussian random field with an exponential covariance model using  covariance tapering
  to reduce the computational costs associated with ML estimation and  optimal linear prediction.
Here we present an improved analysis by considering  a zero mean Gaussian RF with  correlation:
\begin{equation}\label{cc2}
\rho^*(r)=(1-\tau^2)\rho(r)+\tau^2 I(r=0), \quad. r\geq 0,
\end{equation}
that includes a nugget effect $0\leq\tau^2<1$,
as suggested by inspecting the empirical semivariogram  in Figure \ref{figusaaa}, with a correlation function $\rho(r)$
specified as ${\cal M}_{0.5,\beta}$ and  its generalization $\varphi_{0,\mu,\beta}$.
 For the   $\varphi_{0,\mu,\beta}$ model,
 to obtain   sparse  covariance matrices  we fixed different values of $\mu=1.5, 1.75, 2,  2.5, 3.5, 4.5$ and let $\bm{\theta}=(\tau^2,\sigma^2,\beta)^{\top}$ to be estimated
for each of the six Gaussian RFs.

The bottleneck when maximizing  the  likelihood function or computing the optimal liner predictor 
is the Cholesky decomposition which generally has $O(n^3)$ time and $O(n^2)$ memory complexity.
If the matrix is sparse, then the computation of the Cholesky factor can be hastened by using sparse matrix algorithms
and the computational performance of the factorization depends on the percentage of zero elements of the covariance matrix and on how the locations are ordered.

 We point out that ML estimation can partially  take advantage of the  computational benefits associated with the proposed model:
 for a fixed smoothness parameter, the compact support depends on $\beta$ and $\mu$. Even when considering a fixed $\mu$,
the covariance matrix can be highly or slightly sparse, depending on the value of $\beta$ in the optimization process.
An alternative strategy is to use estimation methods with a good balance between statistical efficiency and computational complexity
that do not require any restrictions on the covariance model, such  as
composite likelihood methods  \citep{Eidsvik:Shaby:Reich:Wheeler:Niemi:2013, Bevilacqua:Gaetan:2014}
 or multi-resolution approximation  methods \citep{ka2020} or more in general using Vecchia's approximations \citep{katzfuss2021}. However, in this application we consider ML estimation which is still computational feasible
 although very slow to obtain.

Table  \ref{taba2}  depicts the ML estimates  of  $\bm{\theta}$ with associated standard error for
${\cal M}_{0.5,\beta}$ and  $\varphi_{0,\mu,\beta}$, $\mu=1.5, 1.75,  2,  2.5, 3.5, 4.5$  along with the associated maximized log-likelihood.
It can be appreciated that the  maximized log-likelihood increases with increasing $\mu$, and that
the Mat{\'e}rn  performs the best fitting in this case.
For each model, Table  \ref{taba2}  also reports the percentage of zero entries in the estimated covariance matrix $\Sigma(\hat{\bm{\theta}})$ and the estimated compact support $\widehat{\delta}_{0,\mu,\hat{\beta}} =\mu\hat{\beta}$. As expected, the percentage decreases  and  $\widehat{\delta}_{0,\mu,\hat{\beta}}$ increases with increasing $\mu$.

 Clear computational gains can be achieved using the proposed model when computing  the optimal linear kriging predictor 
which  requires the computation of the Cholesky factor of $\Sigma(\hat{\bm{\theta}})$. To provide an idea of the computational gains,
  Table \ref{taba2} reports,  the time needed for the computation of the Cholesky factor of  $\Sigma(\hat{\bm{\theta}})$ using   the R package \texttt{spam}  \citep{Furrer:Sain:2010} when using
 $\hat{\sigma}^2{\cal M}_{0.5,\hat{\beta}}$  and $\hat{\sigma}^2\varphi_{0,\mu,\hat{\beta}}$ for $\mu=1.5, 1.75, 2,  2.5, 3.5, 4.5$. The time in seconds is expressed in terms of elapsed time, using the
function \texttt{system.time} of the \textsf{R} software on a laptop with a 2.4 GHz processor
and 16 GB of memory.

It is apparent that the computational saving with respect to the Mat{\'e}rn model can be huge when decreasing  $\mu$.
 In particular when $\mu=1.5$  the computation of the Cholesky factor is approximately 50 times faster with respect the Mat{\'e}rn case.
However, the loss of prediction efficiency is generally very small. To  compare the models in terms  of prediction performance, we have used   leave-one-out cross-validation
 as described in \citet{Zhang:Wang:2010}.
 In particular the authors  show that   RMSE,  LSCORE and CRPS leave-one-out cross-validation
can be  computed in just one step by using the estimated covariance matrix. 
In Table   \ref{taba2}  we report  RMSE, LSCORE and  CRPS for the  correlation models considered and the three prediction scores for the Mat{\'e}rn model and its  generalization are quite similar when $\mu\geq2$.
In this specific example,  taking into account
the balance between computational complexity, statistical efficiency and prediction performance, a good choice for the correlation model  could be $\varphi_{0,2,\beta}$.

\begin{figure}[h!]
\centering{
\begin{tabular}{cc}
  \includegraphics[width=5.2cm, height=5.2cm]{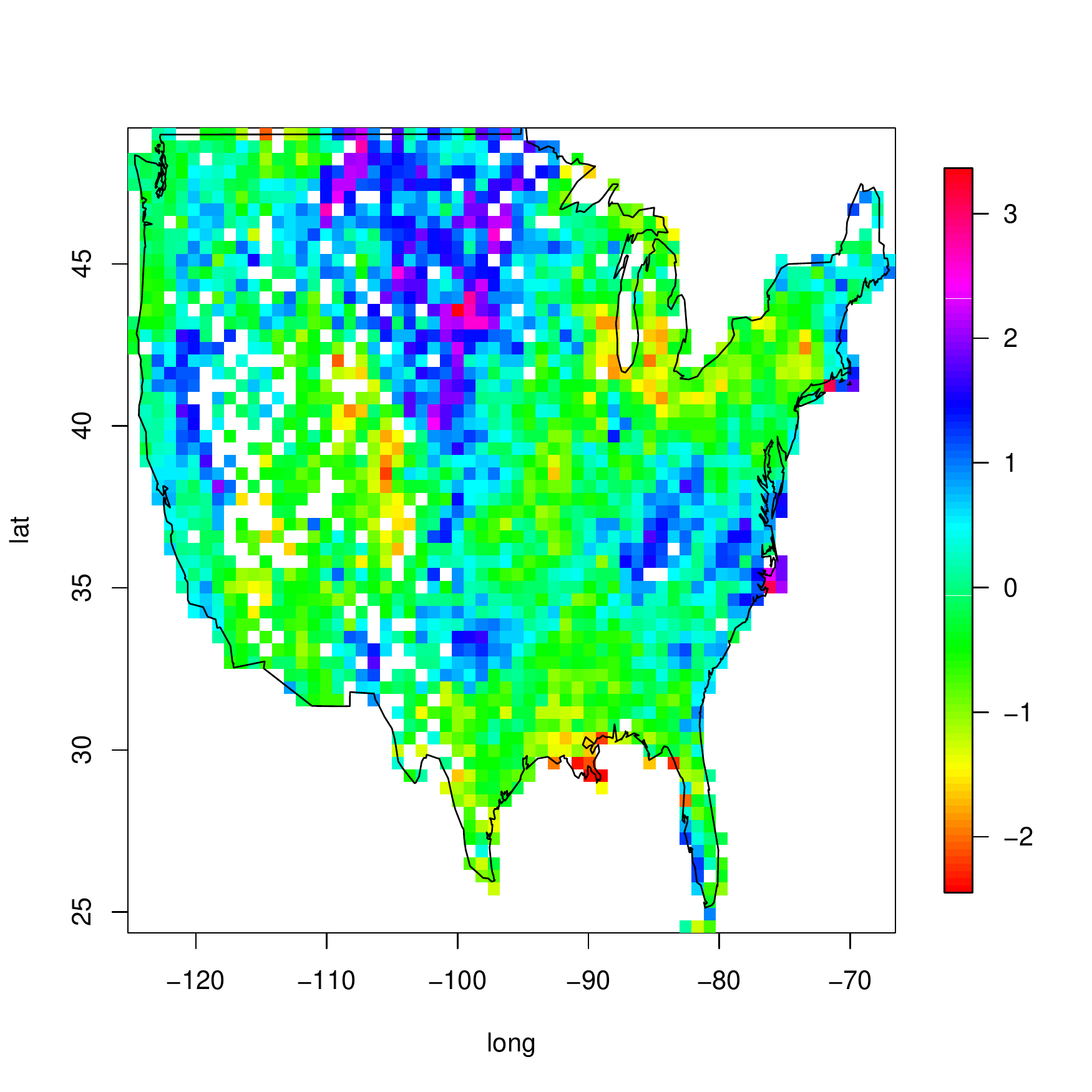}& \includegraphics[width=5.2cm, height=5.2cm]{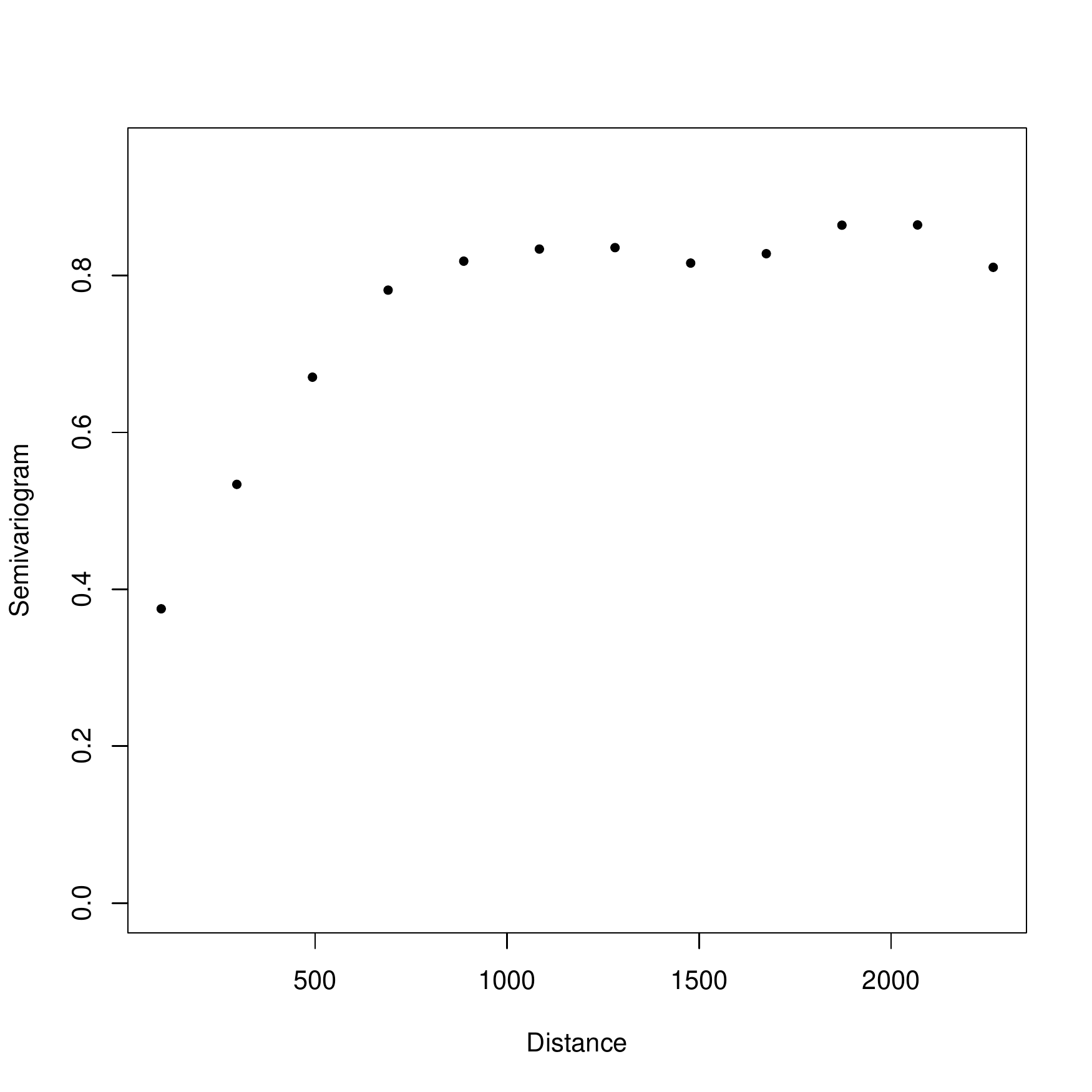}\\
  (a)&(b)
  \end{tabular}
    \caption{From left to right: a) coloured map of precipitation anomalies data.   b) empirical semivariogram of precipitation anomalies data. }\label{figusaaa}
    }
\end{figure}

\begin{table}[h!]
\caption{ML estimates for the parameters of the Mat{\'e}rn model ${\cal M}_{0.5,\beta}$ and the proposed model $\varphi_{0,x,\beta}$
for $x=1.5, 1.75,   2.5,  3.5,  4.5$.
  Prediction measures RMSE,  LSCORE, and CRPS 
  based on  leave-one-out cross-validation are also reported.
The estimated compact support $\hat{\delta}_{0,\mu,\hat{\beta}}=\mu \hat{\beta}$, the percentage of zeros in the estimated covariance matrix  and the computational time (in seconds) to perform the associated Cholesky decomposition are also reported.
} \label{taba2}
\begin{center}
\scalebox{0.6}{
\begin{tabular}{c|c|c|c|c|c|c|c|c|c|c|}
\cline{2-11}
 & $\hat{\tau}^2$ & $\hat{\beta}$ & $\hat{\sigma}^2$ & -loglik & RMSE& LSCORE& CRPS&$\widehat{\delta}_{0,\mu,\hat{\beta}} =\mu\hat{\beta}$&$\%$&TIME \\ \hline
\multicolumn{1}{|c|}{\multirow{2}{*}{$\varphi_{0,1.5,\beta}$}} & \multirow{2}{*}{\begin{tabular}[c]{@{}c@{}}$0.1002$\\ $(0.008)$\end{tabular}} & \multirow{2}{*}{\begin{tabular}[c]{@{}c@{}}$266.38$\\ $(2.20)$\end{tabular}} & \multirow{2}{*}{\begin{tabular}[c]{@{}c@{}}$1.112$\\ $(0.047)$\end{tabular}} & \multirow{2}{*}{$-5443.78$} & \multirow{2}{*}{$0.4691$} & \multirow{2}{*}{$0.9647$}&\multirow{2}{*}{$0.6444$} &\multirow{2}{*}{$399.57$} & \multirow{2}{*}{$0.939$} & \multirow{2}{*}{$1.86$}\\
\multicolumn{1}{|c|}{} &  &  &  &  &  &  & & &&\\ \hline
\multicolumn{1}{|c|}{\multirow{2}{*}{$\varphi_{0,1.75,\beta}$}} & \multirow{2}{*}{\begin{tabular}[c]{@{}c@{}}$0.0945$\\ $(0.008)$\end{tabular}} & \multirow{2}{*}{\begin{tabular}[c]{@{}c@{}}$298.88$\\ $(7.48)$\end{tabular}} & \multirow{2}{*}{\begin{tabular}[c]{@{}c@{}}$1.179 $\\ $(0.053)$\end{tabular}} & \multirow{2}{*}{$-5405.82$} & \multirow{2}{*}{$0.4674$} & \multirow{2}{*}{$0.9607$}&\multirow{2}{*}{$0.6410$} &\multirow{2}{*}{$523.04$}& \multirow{2}{*}{$0.905$} & \multirow{2}{*}{$2.78$}\\
\multicolumn{1}{|c|}{} &  &  &  &  &  &  & & &&\\ \hline
\multicolumn{1}{|c|}{\multirow{2}{*}{$\varphi_{0,2,\beta}$}} & \multirow{2}{*}{\begin{tabular}[c]{@{}c@{}}$0.0964$\\ $(0.007)$\end{tabular}} & \multirow{2}{*}{\begin{tabular}[c]{@{}c@{}}$ 295.21$\\ $(5.29)$\end{tabular}} & \multirow{2}{*}{\begin{tabular}[c]{@{}c@{}}$1.1547 $\\ $(0.053)$\end{tabular}} & \multirow{2}{*}{$-5393.02$} & \multirow{2}{*}{$0.4668$} & \multirow{2}{*}{$0.9595$}&\multirow{2}{*}{$0.6396$} &\multirow{2}{*}{$590.42$}& \multirow{2}{*}{$0.884$} & \multirow{2}{*}{$ 3.63$}\\
\multicolumn{1}{|c|}{} &  &  &  &  &  &  & && & \\ \hline
\multicolumn{1}{|c|}{\multirow{2}{*}{$\varphi_{0,2.5,\beta}$}} & \multirow{2}{*}{\begin{tabular}[c]{@{}c@{}}$0.1103$\\ $(0.008)$\end{tabular}} & \multirow{2}{*}{\begin{tabular}[c]{@{}c@{}}$247.48$\\ $(8.16)$\end{tabular}} & \multirow{2}{*}{\begin{tabular}[c]{@{}c@{}}$0.999$\\ $(0.048)$\end{tabular}} & \multirow{2}{*}{$-5391.58$}& \multirow{2}{*}{$ 0.4669$} & \multirow{2}{*}{$0.9594$}&\multirow{2}{*}{$0.6396$} &\multirow{2}{*}{$618.70$}& \multirow{2}{*}{$0.874$} & \multirow{2}{*}{$4.07$}\\
\multicolumn{1}{|c|}{} &  &  &  &  &  &  & & &&\\ \hline
\multicolumn{1}{|c|}{\multirow{2}{*}{$\varphi_{0,3.5,\beta}$}} & \multirow{2}{*}{\begin{tabular}[c]{@{}c@{}}$0.1110$\\ $(0.011)$\end{tabular}} & \multirow{2}{*}{\begin{tabular}[c]{@{}c@{}}$243.75$\\ $(23.86)$\end{tabular}} & \multirow{2}{*}{\begin{tabular}[c]{@{}c@{}}$0.9905$\\ $(0.085)$\end{tabular}} & \multirow{2}{*}{$-5388.47$}& \multirow{2}{*}{$0.4669$} & \multirow{2}{*}{$0.9594$}&\multirow{2}{*}{$0.6396$} &\multirow{2}{*}{$853.13$} & \multirow{2}{*}{$0.791$} & \multirow{2}{*}{$9.17$}\\
\multicolumn{1}{|c|}{} &  &  &  &  &  &  & && &\\ \hline
\multicolumn{1}{|c|}{\multirow{2}{*}{$\varphi_{0,4.5,\beta}$}} & \multirow{2}{*}{\begin{tabular}[c]{@{}c@{}}$0.1195$\\ $(0.013)$\end{tabular}} & \multirow{2}{*}{\begin{tabular}[c]{@{}c@{}}$216.45$\\ $(27.49)$\end{tabular}} & \multirow{2}{*}{\begin{tabular}[c]{@{}c@{}}$0.9078$\\ $(0.092)$\end{tabular}} & \multirow{2}{*}{$-5386.23$}& \multirow{2}{*}{$0.4669$} & \multirow{2}{*}{$0.9593$}&\multirow{2}{*}{$0.6393$} &\multirow{2}{*}{$974.03$}& \multirow{2}{*}{$0.743$}& \multirow{2}{*}{$11.9$} \\
\multicolumn{1}{|c|}{} &  &  &  &  &  &  & && &\\ \hline
\multicolumn{1}{|c|}{\multirow{2}{*}{${\cal M}_{0.5,\beta}$}} & \multirow{2}{*}{\begin{tabular}[c]{@{}c@{}}$0.1334$\\ $(0.012)$\end{tabular}} & \multirow{2}{*}{\begin{tabular}[c]{@{}c@{}}$167.24$\\ $(18.58)$\end{tabular}} & \multirow{2}{*}{\begin{tabular}[c]{@{}c@{}}$0.7729$\\ $(0.062)$\end{tabular}} & \multirow{2}{*}{$-5377.68$} & \multirow{2}{*}{$0.4668$} & \multirow{2}{*}{$0.9585$}&\multirow{2}{*}{$0.6383$} &\multirow{2}{*}{$\infty$}& \multirow{2}{*}{0} & \multirow{2}{*}{$95.25$}\\
\multicolumn{1}{|c|}{} &  &  &  &  &  &  & & &&\\ \hline
\end{tabular}
}
\end{center}
\end{table}

\section{Conclusions} \label{sec7}

This paper shows that the celebrated Mat{\'e}rn covariance model
is actually a special limit case of a more general compactly supported covariance model
which is a reparameterized version of the generalized Wendland family.
As a consequence, the (reparametrized) Generalized Wendland model is more flexible than the Matérn model with an extra-parameter that allows for switching from compactly to globally supported covariance functions.

On the one hand the proposed family can be potentially more efficient with respect to the Matérn  family when modeling the covariance function of point-referenced spatial data, as shown, for instance in the first real data application. On the other hand, depending on the size of the available dataset, the proposed model can potentially lead to (highly) sparse correlation matrices by fixing  the extra-parameter $\mu$, with clear computational savings with respect to the Matérn  model as shown in  the second real data application.
Further details on  computational gains when  handling sparse matrices with sparse matrices algorithms
 can be found in  \cite{Furrer:2006}, \cite{Kaufman:Schervish:Nychka:2008},  \cite{Furrer:Sain:2010},   \cite{Bevilacqua:2016tap} and  \cite{Porcu:2020sin}  just to mention a few.


Most of the literature on modeling spatial or spatiotemporal multivariate data modeling  is based on the Mat{\'e}rn model as a building block (see \cite{ste2005}, {\cite{pa2006} and \cite{Gneiting:Kleibler:Schlather:2010}, to name a few).
Thus,
our results open new doors and opportunities in   spatial statistics. 
For instance, \cite{Lindgren:Rue:Lindstrom:2011} developed an approximation of Gaussian RFs
with the Mat{\'e}rn covariance model using a Gaussian Markov RF. The connection is established through a specific stochastic partial differential equation (SPDE),
 formulation in that a Gaussian RF with  Mat{\'e}rn covariance is a solution to the SPDE. It could be of  theoretical interest to find a generalization of this specific SPDE exploiting, for instance, the results
 given in \cite{carrizo2018general}.
However, the spectral density  of the proposed model
  cannot be written as the reciprocal of
a polynomial. As a consequence the associated  Gaussian RF is Markovian only when $\mu \to \infty$.

For some  important special cases the proposed covariance model can be easily calculated,  as in  the  Mat{\'e}rn case (see Table  \ref{tab1}).
More generally, the proposed model
can be easily implemented
since efficient numerical computation of the Gaussian hypergeometric function can be found
in different
libraries such as the GNU scientific library \citep{gough2009gnu} and  the most important
statistical softwares including \textsf{R},
MATLAB and Python. In particular, the \textsf{R} package \texttt{GeoModels}  \citep{Bevilacqua:2018aa} used in this paper  for the numerical experiments and the application  computes  the proposed model using the Python implementation of the Gaussian hypergeometric function in the SciPy library \citep{2020SciPy-NMeth}.

\section*{Acknowledgments}

Partial support was provided by FONDECYT grant 1200068 of Chile,  by regional MATH-AmSud program, grant number 20-MATH-03 
and by ANID/PIA/ANILLOS ACT210096 for Moreno Bevilacqua,
 and by FONDECYT grant 11220066 of Chile, 
DIUBB 2120538 IF/R (University of  B\'io-B\'io) for Christian Caama\~no-Carrillo.

\section*{References}

\bibliographystyle{myjmva}
\bibliography{newbib}


\end{document}